\documentclass[12pt, reqno]{amsart} 
\usepackage{mathrsfs}
\usepackage{tikz}
\usetikzlibrary{calc, decorations.pathreplacing, matrix, patterns, arrows.meta}
\usepackage{pgfplots}
\pgfplotsset{compat=1.18, width=12cm, height=8cm}
\usepackage[margin=1in]{geometry}
\definecolor{deepblue}{RGB}{0,74,155}
\definecolor{crimson}{RGB}{178,34,52}
\usepackage[colorlinks=true,linkcolor=blue,citecolor=blue,urlcolor=blue]{hyperref} 
\newcommand{\BMO}{\mathrm{BMO}}

\usepackage{verbatim}
\usepackage{amssymb,amscd,amsfonts,amsbsy}
\usepackage{latexsym}
\usepackage{exscale}
\usepackage{amsmath,amsthm,amsfonts}
\usepackage{mathrsfs}
\usepackage{xcolor} 
\usepackage[colorlinks=true,linkcolor=blue,citecolor=blue,urlcolor=blue]{hyperref} 
\usepackage{esint} 
\usepackage{amssymb} 
\usepackage{stmaryrd}
\usepackage{cite} 
\usepackage{tikz}
\usepackage{enumerate}
\calclayout

\allowdisplaybreaks[3] 

\newtheorem{theorem}{Theorem}[section]
\newtheorem{proposition}[theorem]{Proposition}
\newtheorem{lemma}[theorem]{Lemma}
\newtheorem{corollary}[theorem]{Corollary}
\newtheorem{definition}[theorem]{Definition}
\newtheorem{remark}[theorem]{Remark}

\numberwithin{equation}{section}

\makeatletter
\@addtoreset{equation}{section}
\makeatother

\def\rn{{\mathbb R^n}}

\def\cc{{\mathbb C}}
\def\BMO{{\rm BMO}}

\def\B{{\mathscr B}}
\def\M{{\mathscr M}}
\def\I{{\mathscr I}}
\def\T{{\mathscr T}}
\def\B{{\mathscr B}}
\def\A{{\mathscr A}}
\def\S{{\mathscr S}}

\def\F{{\mathscr F}}

\DeclareMathAlphabet
{\mathpzc}{OT1}{pzc}{m}{it}

\newcommand{\dada}{\beta}

\def\TB{{\mathcal B}_{\eta ,\S,\tau,\tau',{\vec r},s'}^\mathbf{b,k}}

\newcommand{\mici}{\eta}
\newcommand{\lla}{\left\langle}
\newcommand{\rra}{\right\rangle}

\newcommand{\laa}{\big\langle}
\newcommand{\raa}{\big\rangle}
\newcommand{\laaa}{\Big\langle}
\newcommand{\raaa}{\Big\rangle}

\def\Xint1{\mathchoice
   {\XXint\displaystyle\textstyle{1}}%
   {\XXint\textstyle\scriptstyle{1}}%
   {\XXint\scriptstyle\scriptscriptstyle{1}}%
   {\XXint\scriptscriptstyle\scriptscriptstyle{1}}%
   \!\int}
\def\XXint123{{\setbox0=\hbox{$1{23}{\int}$}
     \vcenter{\hbox{$23$}}\kern-.5\wd0}}

\def\vf{\mathbf{f}}

\newcommand{\aver}[1]{-\hskip-0.46cm\int_{1}}
\newcommand{\textaver}[1]{-\hskip-0.40cm\int_{1}}

\def\d{\mathcal{D}}

\def\F{\mathcal{F}}
\def\N{\mathbb{N}}
\def\Z{\mathbb{Z}}

\def\R{\mathbb{R}}

\def\S{\mathcal{S}}

\def\BMO{\operatorname{BMO}}

\newcommand{\norm}[1]{\left\lVert1\right\rVert}
\newcommand\restr[2]{\ensuremath{\left.1\right|_{2}}}

\def\avint_#1{\mathchoice{\mathop{\kern 0.2em\vrule width 0.6em height 0.69678ex depth -0.58065ex \kern -0.8em \intop}\nolimits_{\kern -0.4em#1}}{\mathop{\kern 0.1em\vrule width 0.5em height 0.69678ex depth -0.60387ex \kern -0.6em \intop}\nolimits_{#1}} {\mathop{\kern 0.1em\vrule width 0.5em height 0.69678ex depth -0.60387ex \kern -0.6em \intop}\nolimits_{#1}} {\mathop{\kern 0.1em\vrule width 0.5em height 0.69678ex depth -0.60387ex \kern -0.6em \intop}\nolimits_{#1}}}

\allowdisplaybreaks

\begin{document}
\title[The multilinear fractional sparse operator theory ...]
{\bf The multilinear fractional sparse operator theory II: refining weighted estimates via multilinear fractional sparse forms}

\author[X. Cen]{Xi Cen}
\address{Xi Cen\\
School of Science\\
China University of Mining and
Technology (Beijing)\\
Beijing 100083 \\
People's Republic of China}\email{xicenmath@gmail.com}

\date{February 24, 2025.}

\subjclass[2020]{42B20, 42B25, 47B47, 35J05.}

\keywords{Multilinear fractional sparse form, sparse domination, Calder\'on-Zygmund operator, higher-order commutator, Space of homogeneous type, Fractional Laplacian equations.}


\begin{abstract} 

This paper refines the main results from our previous study on sparse bounds of generalized commutators of multilinear fractional singular integral operators in \cite{CenSong2412}. The key improvements are:
\begin{itemize}
\item We replace pointwise domination with the $(m+1)$-linear fractional sparse form ${\mathcal A}_{\eta,\mathcal{S},\tau,{\vec{r}},s'}^\mathbf{b,k,t}$, advancing the vector-valued multilinear fractional sparse form domination principle, and relax conditions from multilinear weak type boundedness to multilinear locally weak type boundedness $W_{\vec{p}, q}(X)$.
\item We introduce a multilinear fractional $\vec{r}$-type maximal operator $\mathscr{M}_{\eta,\vec{r}}$ and develop a new class of weights $A_{(\vec{p},q),(\vec{r}, s)}(X)$ to characterize it, establishing norm equivalence with the sparse forms.
\item This norm equivalence provides sharp quantitative weighted estimates for $(m+1)$-linear fractional sparse form, removing exponent parameter limitations and achieving sharp operator norm bounds.
\item We demonstrate applications in two ways:
  \begin{itemize}
  \item Providing sharp or Bloom type estimates for generalized commutators of multilinear fractional Calderón--Zygmund operators and multilinear fractional rough singular integral operators.
  \item Investigating sparse form type weighted Lebesgue $L^p(\omega)$ and weighted Sobolev $W^{s,p}(\omega)$ regularity estimates for solutions of fractional Laplacian equations with higher-order commutators.
  \end{itemize}
\end{itemize}

\end{abstract}

\maketitle
\tableofcontents
\section{\bf Introduction}
\subsection{Background}
~~

In \cite{CenSong2412}, we initially extend the sparse operator theory developed over the past decade. However, relying solely on those conclusions is far from sufficient, since only essential technical innovations can truly drive the advancement of sparse theory. The sparse form appears to be an effective structure, as it invariably emerges in proofs over the past ten years, always manifesting in the following three parts.
\begin{align*}
\left| {\left\langle {T(f),g} \right\rangle } \right|\lesssim \underline{\sum\limits_{Q \in \mathcal S}} \underline{{\mu {{(Q)}^{\eta  + 1}} {{\left\langle f \right\rangle }_{r,Q}}}} \cdot \underline{{{\left\langle g \right\rangle }_{t,Q}}}
\end{align*}

This leads that the starting point of this work is originated from the following. 

On the one hand, based on the experiences of outstanding predecessors, such as, Lerner, Li, Lorist, Moen,  Nieraeth, Pérez, Rivera-Rios, and others (c.f. \cite{Lorist2024,Lorist2020,Li2017,Nier2022,Nier2019,Nier2018,Moen2014,Hy.pz2013,Lerner2022,Lerner2019,Lerner2018,Lorist2019}), they provided the most classical approach: The final step of any sparse theory always depends on exploiting the sparsity of sparse cubes in conjunction with the weighted boundedness of the classical dyadic maximal operators, one obtains quantitative weighted estimates for the target operators. 

On the other hand, in \cite{CenSong2412}, we studied  the multilinear fractional sparse form domination principle by the {\tt higher-order multi-symbol $(m+1)$-linear fractional sparse operator} ${\mathscr A}_{\mici ,\S,\tau,{\vec r},s'}^\mathbf{b,k,t}$, which is defined by (in this paper)
\begin{align*}
&{\mathscr A}_{\mici ,\S,\tau,{\vec r},s'}^\mathbf{b,k,t}(\vec{f})(x)\\
&=\sum_{Q \in \mathcal{S}} \mu(Q)^{\eta} \prod_{i \in \tau}  \left|b_i(x) - b_{i,Q}\right|^{k_i - t_i} \langle\left|f_i (b_i - b_{i,Q})^{t_i}\right|\rangle_{r_i,Q} \prod_{j \in \tau^c} \langle\left|f_j\right|\rangle_{r_j,Q} \chi_Q(x).
\end{align*}
This operator always gives us some desired results, such as it always bounds the generalized commutators of the multilinear fractional singular integral operators.

Due to the duality of $L^p(\omega)$, we adopt the following more convenient sparse forms in this paper.



\begin{definition}\label{def.form}
Let $(X,d,\mu)$ be a space of homogenous type (see Sect.\ref{pre.} for its definition), $\eta \in [0,\infty)$, $ r_i, s' \in  [1,\infty)$, for every $i \in \tau \subseteq \tau_m:=\{1,\cdots,m\}$. Let $\mathcal{S}$ be a sparse family and $\mathbf{b} = (b_1, \ldots, b_m) \in (L_{loc}^1(X))^m$. Suppose that $\mathbf{t}$ and $\mathbf{k}$ are multi-indices with $\mathbf{t} \leq \mathbf{k}$. The {\tt higher-order multi-symbol ($m+1$)-linear fractional sparse form} is defined as
   \begin{align*}
{\mathcal A}_{\mici ,\S,\tau,{\vec r},s'}^\mathbf{b,k,t} (\vec f,g) &=\sum_{Q \in \mathcal{S}} \mu(Q)^{\mici + 1} \prod_{i \in \tau}\left\langle\left(\left|f_i (b_i - b_{i,Q})^{t_i}\right|\right)\right\rangle_{r_i,Q}\\ &\times\left\langle\left( \prod_{i \in \tau}  \left|b_i(x) - b_{i,Q}\right|^{k_i - t_i}\right)g \right\rangle_{s',Q} \prod_{j \in \tau^c} \langle\left|f_j\right|\rangle_{r_j,Q},
     \end{align*}

Meanwhile, given $\tau' \subseteq \tau$, we also define the {\tt  reducing  higher-order multi-symbol $(m+1)$-linear fractional sparse forms} by
\begin{align*}
   {\mathcal B}_{\eta ,\S,\tau,\tau',{\vec r},s'}^\mathbf{b,k} (\vec f,g) &= \sum_{Q \in \S} \mu(Q)^{\mici + 1}  
   \prod_{i_1 \in \tau'}\left\langle\left|b_{i_1}- b_{i_1,Q}\right|^{k_{i_1}} f_{i_1}\right\rangle_{r_{i_1},Q}  \\
   &\quad \times \left\langle \left(\prod_{i_2 \in \tau \backslash \tau'} \left|b_{i_2}- b_{i_2,Q}\right|^{k_{i_2}}\right) g \right\rangle_{s',Q}  
   \prod_{j \in (\tau')^c } \left\langle |f_j| \right\rangle_{r_j,Q}.
\end{align*}
\end{definition}

\begin{remark}
Compared with ${\mathcal A}_{\mici ,\S,\tau,{\vec r},s'}^\mathbf{b,k,t}$, the merit of ${\mathcal B}_{\eta ,\S,\tau,\tau',{\vec r},s'}^\mathbf{b,k}$ lies in the disappearance of the oscillation intrinsic factor ${\mathbf{t}}$. This means that once ${\mathcal A}_{\mici ,\S,\tau,{\vec r},s'}^\mathbf{b,k,t}$ and ${\mathcal B}_{\eta ,\S,\tau,\tau',{\vec r},s'}^\mathbf{b,k}$ are comparable, the oscillation intrinsic factor $\mathbf{t}$ cannot have any fundamental impact on the characterization of ${\mathcal A}_{\mici ,\S,\tau,{\vec r},s'}^\mathbf{b,k,t}$. 
\end{remark}

Therefore, we propose the following ideas. 

\begin{proposition}\label{reduce}
Under the same assumption of Definition \ref{def.form}, 
\begin{align*}
  {\mathcal A}_{\mici ,\S,\tau,{\vec r},s'}^\mathbf{b,k,t} (\vec f,g) \le \sum_{\tau' \subseteq \tau} {\mathcal B}_{\eta ,\S,\tau,\tau',{\vec r},s'}^\mathbf{b,k} (\vec f,g).
\end{align*}
\end{proposition}

We are grateful for the work of Lorist et al. \cite[Lemma 3.4]{Lorist2024} here, as they proposed a 1-symbol linear version of this result. We provide the corresponding ($m+1$)-linear  fractional  version as above. This result will be reduced to prove the following lemma.

\begin{lemma}\label{pre_1}
   Let  $1 \leq r_i, t \leq \infty$ with $i \in \tau$,  \(\mathbf{b} = (b_1, \ldots, b_{m}) \in (L_{loc}^1(X))^m\). Suppose that $\mathbf{t}$ and $\mathbf{k}$ are both multi-indexs with \(\mathbf{t} \leq \mathbf{k}\), given a cube $Q$, we write
   \begin{align*}
      C_{{\bf k}, {\bf t}}= \prod_{i \in \tau} \laa \left|f_i (b_i - b_{i,Q})^{t_i}\right|\raa_{r_i,Q}  \laaa \prod_{i \in \tau}  \left|b_i - b_{i,Q}\right|^{k_i - t_i}g\raaa_{t,Q}.
   \end{align*}
   Then we have
   \begin{align*}
      C_{{\bf k}, {\bf t}}\leq  \sum_{\tau' \subseteq \tau} \left( \prod_{i_1 \in \tau'} \left\langle \left| b_{i_1} - b_{i_1,Q} \right|^{k_{i_1}} f_{i_1} \right\rangle_{r_{i_1},Q} 
      \left\langle \prod_{i_2 \in \tau \backslash \tau'} \left| b_{i_2} - b_{i_2,Q} \right|^{k_{i_2}} g \right\rangle_{t,Q} 
      \prod_{i_3 \in \tau \backslash \tau'} \left\langle f_{i_3} \right\rangle_{r_{i_3},Q} \right).
   \end{align*}
\end{lemma}
\begin{proof}[Proof of Lemma \ref{pre_1}]
   For $x_i, y \in X$ with $i \in \tau$ denote
\begin{align*}
   \varphi(\vec x, y) = \left(\prod_{i \in \tau}\left|f_i(x_i) (b_i(x_i) - b_{i,Q})^{t_i}\right|\right)  \left(\prod_{i \in \tau}  \left|b_i(y) - b_{i,Q}\right|^{k_i - t_i}g(y)\right) \chi_{Q^{|\tau| + 1}}(\vec x,y),
\end{align*}
where $\chi_{Q^{|\tau| + 1}}(\vec x,y)=\chi_{Q}(y)\prod\limits_{i \in \tau}\chi_{Q}(x_i)$.

Then we can know  that:
$$
C_{{\bf k}, {\bf t}} = \big\|\|\varphi(\vec x, y)\|_{L^t_{y}(\frac{d y}{\mu(Q)})}\big\|_{L^{\vec{r}}_{\vec{x}}\Big(\frac{d {\vec x}}{\mu(Q)^{|\tau|}}\Big)}
$$
where the mixed norm is defined by
\begin{align*}
\|\cdot\|_{L^{\vec{r}}_{\vec{x}}\Big(\frac{d {\vec x}}{\mu(Q)^{|\tau|}}\Big)}:&=\|\cdot\|_{L^{\vec{r}}_{\vec{x}}\left(\frac{d x_{\tau(1)}}{\mu(Q)} \cdots \frac{d x_{\tau(|\tau|)}}{\mu(Q)}\right)}\\
   &=\Big\| \cdots \big\|\|\cdot\|_{L^{r_{\tau(1)}}_{x_{\tau(1)}}(\frac{d x_{\tau(1)}}{\mu(Q)})}\big\|_{L^{r_{\tau(2)}}_{x_{\tau(2)}}(\frac{d x_{\tau(2)}}{\mu(Q)})} \cdots \Big\|_{L^{r_{\tau(|\tau|)}}_{x_{\tau(|\tau|)}}(\frac{d x_{\tau(|\tau|)}}{\mu(Q)})}.\\
\end{align*}
Note that 
$\left|b_i(x_i)- b_{i,Q}\right|^{k_i-t_i}\left|b_i(y)- b_{i,Q}\right|^{t_i} \leq\left|b_i(x_i)- b_{i,Q}\right|^{k_i}+\left|b_i(y)- b_{i,Q}\right|^{k_i}$, then we obtain
\begin{align*}
   \varphi(\vec x, y) &\leq \prod_{i \in \tau}\big(\left|b_i(x_i)- b_{i,Q}\right|^{k_i}+\left|b_i(y)- b_{i,Q}\right|^{k_i}\big) \times F(\vec x) g(y)\chi_{Q^{|\tau| + 1}}(\vec x,y)\\
   &\leq \sum_{\tau' \subseteq \tau} \left(\prod_{i \in \tau'}\left|b_i(x_i)- b_{i,Q}\right|^{k_i}\prod_{j \in \tau \backslash \tau'} \left|b_j(y)- b_{j,Q}\right|^{k_j} \right)F(\vec x) g(y)\chi_{Q^{|\tau| + 1}}(\vec x,y)
\end{align*}
where $F(\vec x)=\prod\limits_{i \in \tau}f_i(x_i)$.

Considering the fact that $\|\chi_{Q}(y)\|_{L^t_y\left(\frac{dy}{\mu(Q)}\right)} = 1$, it follows readily that 
\begin{align*}
   &\quad C_{{\bf k}, {\bf t}}\\
   &\leq  \left\|\Big\|\sum_{\tau' \subseteq \tau} \left(\prod_{i \in \tau'}\left|b_i(x_i)- b_{i,Q}\right|^{k_i}\prod_{j \in \tau \backslash \tau'} \left|b_j(y)- b_{j,Q}\right|^{k_j} \right)F(\vec x) g(y)\chi_{Q^{|\tau| + 1}}(\vec x,y)\Big\|_{L^t_{y}(\frac{d y}{\mu(Q)})}\right\|_{L^{\vec{r}}_{\vec{x}}\Big(\frac{d {\vec x}}{\mu(Q)^{|\tau|}}\Big)}\\
   &\le \sum_{\tau' \subseteq \tau} \left( \prod_{i_1 \in \tau'} \left\langle \left| b_{i_1} - b_{i_1,Q} \right|^{k_{i_1}} f_{i_1} \right\rangle_{r_{i_1},Q} 
   \left\langle \prod_{i_2 \in \tau \backslash \tau'} \left| b_{i_2} - b_{i_2,Q} \right|^{k_{i_2}} g \right\rangle_{t,Q} 
   \prod_{i_3 \in \tau \backslash \tau'} \left\langle f_{i_3} \right\rangle_{r_{i_3},Q} \right).\qedhere
\end{align*}

\end{proof}

Based on the above discussion, we have realized that ${\mathcal B}_{\eta ,\S,\tau,\tau',{\vec r},s'}^\mathbf{b,k}$ is a reduced version of ${\mathcal A}_{\mici ,\S,\tau,{\vec r},s'}^\mathbf{b,k,t}$. Therefore, we will study them in detail later, including yielding sparse bounds via them, and the weighted estimates of them. 

The notations in the following that are the same as the above are consistent with the above statements and concepts, unless we make special explanations for them.

\subsection{Organization}
~~

The structure of the rest is as follows. 
\begin{itemize}
\item In Sect. \ref{pre.}, we recall the concepts of space of homogeneous type, dyadic lattices and adjacent systems of dyadic cubes, and state preliminary properties of oscillation symbols.

\item In Sect. \ref{MSFSDP}, we establish
 the vector-valued multilinear fractional sparse form domination principle via higher-order multi-symbol $(m+1)$-linear fractional sparse forms.

\item In Sect. \ref{Maximal.control.}, we will set up the multilinear fractional $\vec{r}$-type maximal operators and a new class of multilinear fractional weights $ A_{(\vec{p},q),(\vec{r}, s)}(X)$, which can characterize this maximal operator, followed by revealing the norm equivalence between this maximal operator and the sparse forms.

\item In Sect. \ref{Bloom.estimate.}, the new Bloom type weighted estimates for multilinear fractional   sparse forms ${\mathcal B}_{\eta ,\S,\tau,\tau',{\vec r},s'}^\mathbf{b,k}$ is presented, in which we will use the new technic established in the Sect. \ref{Maximal.control.}.

\item In Sect. \ref{sharp}, we shall construct the sharp weighted estimates for higher-order multi-symbol multilinear fractional sparse form.

\item In Sect. \ref{A1}, we will demonstrate that the classical multilinear fractional singular integral operators and their generalized commutators are applicable to the multilinear fractional sparse form domination principle proposed in  Sect. \ref{MSFSDP}.

\item In Sect. \ref{A2}, we investigate sparse form type weighted Lebesgue $L^p(w)$ and weighted Sobolev $W^{s,p}(w)$ regularity estimates of the solution for a class of fractional Laplacian equations associated with the higher-order commutators via the multilinear fractional sparse bounds.

\end{itemize}

\section{\bf Preliminaries}\label{pre.}

\subsection{Spaces of homogeneous type }\label{def_hom}
~~

To illustrate the idea of this paper, we first provide the relevant definitions for spaces of homogeneous type (SHT).

\begin{definition}
	Let $d: X \times X \rightarrow [0, \infty)$ be a positive function and $X$ be a set, then the quasi-metric space $(X, d)$ satisfies the following conditions:
	\begin{enumerate}
		\item  When $x=y$, $d(x, y)=0$.
		\item $d(x, y)=d(y, x)$ for all $x, y \in X$.
		\item  For all $x, y, z \in X$, there is a constant $A_0 \geq 1$ such that $d(x, y) \leq A_0(d(x, z)+d(z, y))$.
	\end{enumerate}
\end{definition}

\begin{definition}
	Let $\mu$ be a measure of a space $X$. For a quasi-metric ball $B(x,  r)$ and any $r>0$, if $\mu$ satisfies doubling condition, then there exists a doubling constant $C_\mu \geq 1$, such that  
	\begin{align}\label{def_hom}
    0<\mu(B(x, 2 r)) \leq C_\mu \mu(B(x, r))<\infty.
  \end{align}
\end{definition}

\begin{definition}
	For a non-empty set $X$ with a qusi-metric $d$, a triple $(X, d, \mu)$ is said to be a space of homogeneous type if $\mu$ is a regular measure which satisfies doubling condition on the $\sigma$-algebra, generated by open sets and quasi-metric balls.
\end{definition}

\subsection{Dyadic analysis}

Consider a measure space $(X, d, \mu)$ where $0 < c_0 \leq C_0 < \infty$, and $0 < \delta < 1$. Let $J_k$ be a given index set. For each $k$, define $\d_k = \{Q_k^j\}_{j \in J_k}$ as a collection of measurable sets and a set of points $\{z_j^k\}_{j \in J_k}$. We refer to the collection $\d := \bigcup_{k \in \mathbb{Z}} \d_k$ as a \emph{dyadic lattice} with parameters $c_0$, $C_0$, and $\delta$ if it satisfies the following conditions:
\begin{list}{\rm (\theenumi)}{\usecounter{enumi}\leftmargin=1cm \labelwidth=1cm \itemsep=0.2cm \topsep=.2cm \renewcommand{\theenumi}{\roman{enumi}}}
\item For all $k \in \mathbb{Z}$ we have
\begin{align*}
X=\bigcup_{j \in J_k} Q_j^k;
\end{align*}
\item For $k \geq l, Q \in \d_k$ and $Q^{\prime} \in \d_l$ we either have $Q \cap Q^{\prime}=\varnothing$ or $Q \subseteq Q^{\prime}$;
\item For each $k \in \mathbb{Z}$ and $j \in J_k$ we have
\begin{align}\label{eq:contain}
B\left(z_j^k, c_0 \delta^k\right) \subseteq Q_j^k \subseteq B\left(z_j^k, C_0 \delta^k\right),
\end{align}
where $z_j^k$ and $\delta^k$ as the center and side length of a cube $Q_j^k \in \d_k$.
\end{list} 

Follow we define the sparse family of cubes.
\begin{definition}\label{D:sparse}
  For $0 < \delta < 1$, a collection $\mathcal{S} $ of dyadic cubes is \textit{$\delta$-sparse} if for each $Q \in \mathcal{S}$, there exists a measurable subset $E_Q \subseteq Q$ with $\mu(E_Q) \geq \delta \mu(Q)$, and the family $\{E_Q\}_{Q \in \mathcal{S}}$ are pairwise disjoint.
\end{definition}

Then we define the restricted dyadic lattice $\d(Q):=\{P \in \d: P \subseteq Q\}$. We will say that an estimate depends on $\d$ if it depends on the parameters $c_0, C_0$ and $\delta$. We more consider the dilations of such cubes, 
in Euclidean space, the expansion of a cube can be easily determined using volume calculations. However, in a space of homogeneous type, we need to redefine the dilations $\dada P$ for $\dada \geq 1$ as
\begin{align*}
\dada P:=B\left(z, \dada \cdot C_0 \delta^k\right).
\end{align*}

In $\R^n$ we all know a fact that any ball is contained in a cube of comparable size from one of these dyadic lattices (refer to \cite[ Lemma 3.2.26]{hombook}). Follow we introduce the a proposition to achieve the same effect on spaces of homogeneous type. We omit the proof details. Readers can refer to \cite{HK}.

\begin{proposition}\label{cubeeq}
  Let \((X, d, \mu)\) be a space of homogeneous type. There exist constants \(0 < c_0 \leq C_0 < \infty\), \(\gamma \geq 1\), \(0 < \delta < 1\), and an integer \(m \in \mathbb{N}\) such that there exist dyadic systems \(\d^1, \ldots, \d^N\) with parameters \(c_0, C_0\), and \(\delta\), satisfying the following property: for each point \(x \in X\) and radius \(\rho > 0\), there exists an index \(j \in \{1, \ldots, N\}\) and a cube \(Q \in \d^j\) such that
\[
B(x, \rho) \subseteq Q, \quad \text{and} \quad \operatorname{diam}(Q) \leq \gamma \rho.
\]
\end{proposition}


The following covering lemma fundamentally determines that we can establish sparse domination on \(X\). Its proof relies on Proposition \ref{cubeeq}, which is proved in \cite{Lorist2020}.
\begin{lemma}[\cite{Lorist2020}, Lemma 2.2]\label{lem.covering}
  Let \((X, d, \mu)\) be a space of homogeneous type and \(\d\) a dyadic system with parameters \(c_0, C_0\), and \(\delta\). Suppose that \(\operatorname{diam}(X) = \infty\), take \(\dada \geq \frac{3 A_0^2}{\delta}\), and set \(E \subseteq X\) satisfy \(0 < \operatorname{diam}(E) < \infty\). Then there exists a partition \(\mathcal{D}' \subseteq \d\) of \(X\) such that \(E \subseteq \dada P\) for all \(Q \in \mathcal{D}'\).
\end{lemma}

\subsection{Preliminary properties of oscillation symbols}

\begin{definition}\label{MWbmo}
 A locally integrable function $b \in L^1_{\text{loc}}(X)$ belongs to ${ BMO}_\omega(X)$ provided that
\[
\|b\|_{{ BMO}_\omega(X)} := \sup_{B} \frac{1}{\omega(B)} \int_B |b(x) - b_B| \, d\mu(x) < \infty,
\]
where
\[
b_B := \frac{1}{\mu(B)} \int_B b(x) \, d\mu(x)
\]
and the supremum is taken over all balls $B \subseteq X$.
  \end{definition}

\begin{lemma}[\cite{Yang2019}]\label{zhang:6.1}
   Let $0 < \gamma < 1$, $\mathcal{D}$ be a dyadic lattice in $X$, and $\mathcal{S} \subseteq \mathcal{D}$ a $\gamma$-sparse family. For $b \in L^1_{\text{loc}}(X)$, there exists a $\frac{\gamma}{2(\gamma+1)}$-sparse family $\tilde{\mathcal{S}} \subseteq \mathcal{D}$ containing $\mathcal{S}$ such that for each $Q \in \tilde{\mathcal{S}}$,
   \[
     |b(x) - b_Q| \leq C \sum_{\substack{R \in \tilde{\mathcal{S}} \\ R \subseteq Q}} \langle\left|b(x)-b_R\right| \rangle_{R} \chi_R(x) \quad \text{a.e. } x \in Q.
   \]
 \end{lemma}

\begin{lemma}[\cite{250}, Corollary 3.1.9]\label{BMOeq.}
   For all $1<p<\infty$, given $\d$ is a dyadic lattice of $X$, then
   \begin{align}\label{BMOeq._}
\mathop {\sup }\limits_{Q \in \d} \lla |b - b_{Q}| \rra_{p,Q} \approx\|f\|_{\mathrm{B M O}}.
   \end{align}
\end{lemma}
\begin{lemma}\label{xin_1}
Let \((X, d, \mu)\) is a space of homogeneous type  with a dyadic lattice \(\mathcal{D}\). Let \(Q \in \mathcal{D}\) and \(b \in {\rm BMO}(X)\). Let \(w\) be a nonnegative function integrable over \(Q\) with respect to the measure \(\mu\), and \(s\) be a positive real number. Then the following inequality holds
\[
\int_Q w(x) |b(x) - b_Q|^s \, d\mu(x) \leq {C^*} \left( \int_Q w(x) \, d\mu(x) \right) \|b\|_{{\rm BMO}}^s,
\]
where ${C^*}$ depends on $s$, \(A_0\), \(C_\mu\), \(c_0\), \(C_0\) and \(\delta\).
\end{lemma}
\begin{proof}
To establish the inequality, we proceed as follows. In a space of homogeneous type with a dyadic lattice, functions in BMO satisfy a dyadic John-Nirenberg inequality: for any dyadic cube \(Q \in \mathcal{D}\) and any \(\lambda > 0\), 
\[
\mu(\{ x \in Q : |b(x) - b_Q| > \lambda \}) \leq c_1 \mu(Q) e^{-c_2 \lambda / \|b\|_{\text{BMO}}},
\]
where \(c_1 \geq 1\) and \(c_2 > 0\) depend only on the space’s parameters (\(A_0\), \(C_\mu\), \(c_0\), \(C_0\), \(\delta\)). We express the integral using the layer cake representation, 
\[
\int_Q w |b - b_Q|^s \, d\mu = s \int_0^\infty \lambda^{s-1} \left( \int_{\{ x \in Q : |b - b_Q| > \lambda \}} w(x) \, d\mu(x) \right) d\lambda,
\]
decomposing the \(s\)-th power into level sets. Since \(w \geq 0\) is integrable, the inner integral is bounded by 
\[
\int_{\{ x \in Q : |b - b_Q| > \lambda \}} w(x) \, d\mu(x) \leq \int_Q w(x) \, d\mu(x),
\]
as \(\{ x \in Q : |b - b_Q| > \lambda \} \subseteq Q\). Applying the John-Nirenberg inequality, we get 
\begin{align*}
	\int_Q w |b - b_Q|^s \, d\mu & \leq s \left( \int_Q w \, d\mu \right) \int_0^\infty \lambda^{s-1} \frac{\mu(\{ x \in Q : |b - b_Q| > \lambda \})}{\mu(Q)} \, d\lambda \\
	&\leq s \left( \int_Q w \, d\mu \right) \int_0^\infty \lambda^{s-1} c_1 e^{-c_2 \lambda / \|b\|_{\text{BMO}}} \, d\lambda.
\end{align*}
Now, compute the integral \(\int_0^\infty \lambda^{s-1} e^{-c_2 \lambda / \|b\|_{\text{BMO}}} \, d\lambda\) with the substitution \(t = \frac{c_2 \lambda}{\|b\|_{\text{BMO}}}\), so \(\lambda = \frac{\|b\|_{\text{BMO}}}{c_2} t\), \(d\lambda = \frac{\|b\|_{\text{BMO}}}{c_2} dt\), transforming the limits from \(\lambda = 0\) to \(\infty\) to \(t = 0\) to \(\infty\). This yields 
\[
\int_0^\infty \left( \frac{\|b\|_{\text{BMO}}}{c_2} t \right)^{s-1} e^{-t} \cdot \frac{\|b\|_{\text{BMO}}}{c_2} \, dt = \left( \frac{\|b\|_{\text{BMO}}}{c_2} \right)^s \int_0^\infty t^{s-1} e^{-t} \, dt,
\]
and since \(\int_0^\infty t^{s-1} e^{-t} \, dt = \Gamma(s)\), we have 
\[
\int_0^\infty \lambda^{s-1} e^{-c_2 \lambda / \|b\|_{\text{BMO}}} \, d\lambda = \left( \frac{\|b\|_{\text{BMO}}}{c_2} \right)^s \Gamma(s).
\]
Substituting back, 
\[
\int_Q w |b - b_Q|^s \, d\mu \leq s \left( \int_Q w \, d\mu \right) c_1 \left( \frac{\|b\|_{\text{BMO}}}{c_2} \right)^s \Gamma(s) = \left[ s c_1 \left( \frac{1}{c_2} \right)^s \Gamma(s) \right] \left( \int_Q w \, d\mu \right) \|b\|_{\text{BMO}}^s.
\]
Defining \(C = s c_1 \left( \frac{1}{c_2} \right)^s \Gamma(s)\), we obtain the desired inequality.

\end{proof}

Next we first introduce some important notation.

\subsection{Notation and Setting}\label{cuti}
Let \(X\) be a metric measure space, with \(L_b^{\infty}(X)\) as the space of bounded functions with bounded support in \(L^{\infty}(X)\), and \(\omega: X \rightarrow [0, \infty]\) a locally integrable weight satisfying \(0 < \omega(x) < \infty\) almost everywhere, defining the measure \(d\omega(x) = \omega(x)d\mu(x)\) and \(\omega(E) = \int_E \omega \,d\mu\) for measurable \(E \subseteq X\). For \(r > 0\), \(E \subseteq X\), and \(f \in L_{\text{loc}}^r(X)\), the average \(r\)-norm is
\begin{align*}
	\langle f\rangle_{r, E} = \left(\frac{1}{\mu(E)} \int_E |f|^r \,d\mu\right)^{1/r},
\end{align*}
denoted \(\langle f\rangle_{E}\) when \(r = 1\). 

We define \(\mathbf{b} = (b_1, \dots, b_m) \in (L_{\text{loc}}^1(X))^m\) as an \(m\)-tuple of locally integrable functions, and \(\mathbf{k}, \mathbf{t} \in (\mathbb{N}_0)^m\) as \(m\)-tuples of non-negative integers, where \(\mathbf{k} \geq \mathbf{t}\) if \(k_i > t_i\) for all \(i = 1, \dots, m\). For \(\tau_m = \{1, \dots, m\}\) and \(\tau \subseteq \tau_m\), we denote \(|\tau|\) as its cardinality, \(\tau^c = \tau_m \setminus \tau\) as its complement, and \(\tau_\ell\) if \(\tau = \{\tau(j)\}_{j=1}^{|\tau|}\) is strictly increasing with \(|\tau| = \ell\). 

For \(m \in \mathbb{N}\), \(\vec{r} = (r_1, \dots, r_m) \in (0, \infty)^m\), \(\vec{p} = (p_1, \dots, p_m) \in (0, \infty]^m\), and \(s \in (0, \infty]\), we say \(\vec{r} \leq \vec{p}\) if \(r_j \leq p_j\) and \(\vec{r} < \vec{p}\) if \(r_j < p_j\) for all \(j = 1, \dots, m\); here, \(\boldsymbol{p} = (p_1, \dots, p_{m+1})\) is an \((m + 1)\)-tuple. We define \((\vec{r}, s) \stackrel{}{\prec} (\vec{p}, q)\) when \(\vec{r} < \vec{p}\) and \(q < s\), \((\vec{r}, s) \stackrel{}{\preceq} (\vec{p}, q)\) when \(\vec{r} \leq \vec{p}\) and \(q < s\), and \((\vec{r}, s) \preceq^{*} (\vec{p}, q)\) when \(\vec{r} \leq \vec{p}\) and \(q \leq s\).

Let ${\bf b} \in {\left( {L_{loc}^1\left( X \right)} \right)^m}$. Given an $m$-sublinear operator $G$ and a multi-index \(\mathbf{k} = (k_1, \ldots, k_{m}) \in {\N_0}^m \), its generalized commutator can be defined by 
\begin{align*}
	G_{\tau_{\ell}}^{{\bf b, k}}(\vec{f})(x):=G\left((b_1(x) - b_1)^{\beta_1}f_1,\ldots,(b_m(x) - b_m)^{\beta_m}f_m\right)(x),
\end{align*}

where
\[
\beta_i = 
\begin{cases}
	k_i, & \text{if} \quad i \in \tau, \\
	0, & \text{if} \quad i \in \tau^{c}.
\end{cases}
\] 
When ${\bf k}=\vec{1}$, we denote it by $G_{\tau_{\ell}}^{{\bf b}}$.
When ${\bf k}=0$, the generalized commutator is back to $G$.

For some $\vec{\vf} = {\left\{ {\left( {{f_{1,{j_1}}}, \cdots ,{f_{1,{j_m}}}} \right)} \right\}_{\scriptstyle{j_i} \in \Z \hfill\atop
\scriptstyle   i \in \left\{ {1, \cdots ,m} \right\}\hfill}} \subseteq L_{loc}^{1}\left(X\right)$, we can define that 
\begin{align*}
G_{{\tau _\ell }}^{{\bf{b}},{\bf{k}}}(\vec{\vf})(x)
:&={\left\{ {G_{{\tau _\ell }}^{{\bf{b}},{\bf{k}}}\left( {{f_{1,{j_1}}}, \cdots ,{f_{1,{j_m}}}} \right)(x)} \right\}_{\scriptstyle{j_i} \in \Z \hfill\atop\scriptstyle   i \in \left\{ {1, \cdots ,m} \right\}\hfill}}\\
&= {\left\{ {G\left( {{{({b_1}(x) - {b_1})}^{{\beta _1}}}{f_{1,{j_1}}}, \ldots ,{{({b_m}(x) - {b_m})}^{{\beta _m}}}{f_{1,{j_m}}}} \right)(x)} \right\}_{\scriptstyle{j_i} \in \Z \hfill\atop
\scriptstyle   i \in \left\{ {1, \cdots ,m} \right\}\hfill}}.
\end{align*}

Set $A$, $B$ are both quasi-Banach spaces. We define the mixed norm by $${\left\| h \right\|_{{A}({B})}}: = {\left\| {{{\left\| h \right\|}_{{B}}}} \right\|_{{A}}}.$$
In particular, for some $z \ge 1$, we define the norm ${\left\|  \cdot  \right\|_{{l^z}(B)}}$ by 
\[{\left\| {\left\{ {{h_i}} \right\}} \right\|_{{l^z}(B)}}: = {\left\| {\left\{ {{{\left\| {{h_i}} \right\|}_B}} \right\}} \right\|_{{l^z}}}.\]

\section{\bf Vector-valued multilinear fractional sparse form domination principle}

Given $1 \leq s \leq \infty$. Let $\mathcal{D}$ be a dyadic lattice, $Y$ be a quasi-Banach space, and $T$ be a $Y$-valued $m$-linear operator. The sharp grand maximal truncation operator is defined as
\begin{align*}
&\M_{T,s,\dada,Y}^{\#} \vec{f}(x)\\
:=&\mathop {\sup }\limits_{Q \in \mathcal{D}} {\left( {\frac{1}{{\mu {{(Q)}^2}}}\int_{Q \times Q} {{{\left\| {{T}\left( {\vec f{\chi _{X\backslash \beta Q}}} \right)\left( {{x^\prime }} \right) - {T}\left( {\vec f{\chi _{X\backslash \beta Q}}} \right)\left( {{x^{\prime \prime }}} \right)} \right\|}_Y^s}} \;d\mu ({x^\prime })d\mu ({x^{\prime \prime }})} \right)^{\frac{1}{s}}}{\chi _Q}(x),
\end{align*}

Given a function $f:X \to Y$ and a dyadic cube $Q$, we define the oscillation operation by
\begin{align*}
\operatorname{osc}_{s,Y}(f ; Q):=\left(\frac{1}{\mu(Q)^2} \int_{Q \times Q}\left\|f\left(x^{\prime}\right)-f\left(x^{\prime \prime}\right)\right\|_Y^s \mathrm{~d} \mu(x^{\prime}) \mathrm{d} \mu(x^{\prime \prime})\right)^{1 / s}
\end{align*}

The following multilinear locally weak type boundedness $W_{\vec{p}, q}(X)$ will replace multilinear weak type boundedness of $\mathcal{T}_{\eta}$, c.f. \cite{CenSong2412}.

\begin{definition}\label{def:multilinear_W}
Let \(m \in \mathbb{N}\), $\eta \in [0,m)$, $q \in (0,\infty)$, and   $p_i \in (1,\infty)$ for \(i \in \tau_m\). Set $Y_i$, $Y$ be quasi-Banach spaces with \(i \in \tau_m\), and $G$ be a $Y$-valued $m$-linear operator. We write $G \in W_{\vec{p},q}$, if there exists a non-increasing function $\Phi_{G,\vec{p}, q}:\left( {0,1} \right) \to \left( {0,\infty } \right) $, such that for every dyadic cube $Q$ and for every $f_i \in L^{p_i}(Q,Y_i)$, $1 \leq i \leq m$,
\[
 \mathop {\sup }\limits_{\lambda  \in \left( {0,1} \right)} \lambda^{-1} \mu\left(\left\{ x \in Q : 
 \|G(f_1\chi_Q, \dots, f_m\chi_Q)(x)\|_Y > \Phi_{G,\vec{p}, q}(\lambda) \mu(Q)^{\eta} \prod_{i=1}^m \langle \|f_i\|_{Y_i} \rangle_{p_i, Q} \right\}\right) \leq \mu(Q).
\]
 If $G$ is of weak type $(\vec{p}, q)$, then the $W_{\vec{p},q}$ property holds with $$\Phi_{G,\vec{p},q}(\lambda) = \|G\|_{L^{p_1}(X,Y_1) \times \cdots \times L^{p_m}(X,Y_m) \to L^{q,\infty}(X,Y)} \cdot \lambda^{-1/q}.$$
\end{definition}

Now we are ready to introduce the main theorem of this section. We call that {\tt vector-valued multilinear fractional sparse form domination principle}.

\begin{theorem}\label{SDP}
  Let $m \in \N$, $\eta \in [0,m)$, $ r_i, s' \in [1,\infty)$, and $\max\limits_{i}\{r_i\} < s$ for every $i \in \tau_m$. Let $\mathbf{k}$ be a multi-index, and multi-symbols \(\mathbf{b} = (b_1, \ldots, b_{m}) \in (L_{loc}^1(X))^m\). Set $Y_i$, $Y$ be quasi-Banach spaces with $i \in \tau_m$, and $\mathcal{T}_{\eta}$ be a $Y$-valued $m$-linear operator. Assume that $\mathcal{T}_{\eta}, \M_{\mathcal{T}_{\eta},s,\beta}^{\#} \in W_{\vec{r},\tilde{r}}$ with $\frac{1}{\tilde{r}} := \sum_{i=1}^m \frac{1}{r_i} - \eta$.
Take $\beta \geq 3 A_0^2 / \delta$, where $A_0$ is the quasi-metric constant and $\delta$ is as in Proposition \ref{cubeeq}. 
Then there exists a sparse family $\mathcal{{\tilde S}} \subseteq \d$, for every $z \in [1,\infty)$, for any $\vec{\vf} = {\left\{ {\left( {{f_{1,{j_1}}}, \cdots ,{f_{m,{j_m}}}} \right)} \right\}_{\scriptstyle{j_i} \in \left\{ {1, \cdots ,{N_i}} \right\}\hfill\atop
\scriptstyle i \in \left\{ {1, \cdots ,m} \right\}\hfill}} \subseteq \prod\limits_{i = 1}^m {L_b^\infty \left( {X,{Y_i}} \right)}$
and non-negative function $g \in L_b^{\infty}\left(X\right)$,  
  \begin{align}
  &\quad \left\langle\|\mathcal{T}_{\eta,\tau_{}}^{{\bf b, k}}(\vec{\vf})\|_{l^z(Y)},  g \right\rangle \notag\\
    &\le C \sum\limits_{\mathbf{t}_\tau} {\mathcal A}_{\mici ,\mathcal{{\tilde S}},\tau,{\vec r},s'}^\mathbf{b,k,t} \left( {{{\left\| {{{\left\{ {{f_{1,{j_1}}}} \right\}}_{{j_1} \in \left\{ {1, \cdots ,{N_1}} \right\}}}} \right\|}_{{l^z(Y_1)}}}, \ldots ,{{\left\| {{{\left\{ {{f_{m,{j_m}}}} \right\}}_{{j_m} \in \left\{ {1, \cdots ,{N_m}} \right\}}}} \right\|}_{{l^z(Y_m)}}},g} \right)\label{m+1.spa_0}\\
    &\le C \sum\limits_{\mathbf{t}_\tau} \sum_{\tau' \subseteq \tau}  {\mathcal B}_{\mici ,\mathcal{{\tilde S}},\tau,\tau',{\vec r},s'}^\mathbf{b,k} \left( {{{\left\| {{{\left\{ {{f_{1,{j_1}}}} \right\}}_{{j_1} \in \left\{ {1, \cdots ,{N_1}} \right\}}}} \right\|}_{{l^z(Y_1)}}}, \ldots ,{{\left\| {{{\left\{ {{f_{m,{j_m}}}} \right\}}_{{j_m} \in \left\{ {1, \cdots ,{N_m}} \right\}}}} \right\|}_{{l^z(Y_m)}}},g} \right),\label{m+1.spa}  
 \end{align}
 where there exists a constant $C_{\tau, {\bf k}, c_1, c_2} \in (1, \infty)$ 
 with $c_1, c_2$ being constants such that $c_1 \geq 1$ and $c_2 \geq 2$, satisfying the following:
 \begin{align}\label{constants}
C := C_{\tau, {\bf k}, c_1, c_2} \left( \Phi_{\mathcal{T}_{\eta}, \vec{r}, \tilde{r}}\left( \frac{1}{6|\tau|(|{\bf k}| + 1) c_1 c_2} \right) + \Phi_{\M^{\#}_{\mathcal{T}_{\eta, s,\dada}}, \vec{r}, \tilde{r}}\left( \frac{1}{6|\tau|(|{\bf k}| + 1) c_1 c_2} \right) \right).
 \end{align}
Additionally, $\sum\limits_{\mathbf{t}_\tau}$ means that $\sum\limits_{t_{\tau(1)}=0}^{k_{\tau(1)}} \sum\limits_{t_{\tau(2)}=0}^{k_{\tau(2)}} \cdots \sum\limits_{t_{\tau(|\tau|)}=0}^{k_{\tau(|\tau|)}}$.

\end{theorem}

\, \hspace{-20pt}{\bf Remark A}. 
{\it\
\begin{itemize}
\item 
We use $\M_{T,s,\dada}^{\#}$ instead of the classic multilinear operators $\M_{T,\infty}^{\#}$ and $\M_{T,\infty}$. This is because $\M_{T,s,\dada}^{\#}$ exhibits more flexible and broader properties, which makes our theorem more general, c.f. \cite{CenSong2412, Cao2018, Nier2020, Lorist2021}.
\item
The multilinear locally weak type boundedness $W_{\vec{p}, q}(X)$ allows us to weaken the classical sparse domination premise for target operator. In nearly a decade of sparse domination theory, authors always used the endpoint weak type estimates for target operators, c.f. \cite{CenSong2412, Cao2018,Lerner2019, Nier2020, Lorist2021}.
\item The condition $\max\limits_{i \in \tau_m}\{r_i\} < s$ is necessary. On the one hand, it ensures that the assumption $\M_{\mathcal{T}_{\eta},s,\beta}^{\#} \in W_{\vec{r},\tilde{r}}$ is reasonable. On the other hand, it guarantees the validity of the theorem when it reduces to classical cases such as \cite[lemma 2.1]{CenSong2412}, \cite[Lemma 4.1]{Cao2018} and the assumptions in \cite[Theorem 1.1]{Li2018}. While $m=1$ and \((\vec{r}, s) \stackrel{}{\preceq} (\vec{p}, q)\) with $\eta = 0$, if $r \geq s$, then we will find that we obtain $r \leq p < s \le r$, there is a contradiction.
\end{itemize}

}

\begin{proof}[Proof of Theorem~\ref{SDP}]\label{MSFSDP}
The second inequality in \eqref{m+1.spa} is easily derived from Proposition \ref{reduce}. It suffices to prove \eqref{m+1.spa_0}.

The first thing that requires attention is the local version of \eqref{m+1.spa_0}. 
Fix a dyadic cube $Q \in \d$, which contains the supports of $\vec {\vf}$ and $g$, we will show that there exists a $\frac{1}{2}$-sparse family $\mathcal{F} \subseteq \mathcal{D}(Q)$ such that
\begin{align}\notag
   &\left\langle \left\| \left\{\mathcal{T}_{\eta,\tau_{}}^{{\bf b, k}}(\vec{\vf})(x) \right\}\right\|_{l^z(Y)},g \right\rangle =\left\langle \left\| \left\{\mathcal{T}_{\eta,\tau_{}}^{{\bf b, k}}(\vec{\vf}\chi_{\dada Q})(x) \right\}\right\|_{l^z(Y)},g\chi_Q\right\rangle
  \\ \notag
  \le & C_{} \sum\limits_{\mathbf{t}_\tau} \left(\sum_{Q \in \mathcal{F}} \mu(P)^{\eta + 1} \right.  \prod_{i \in \tau}\left\langle\left\|\left\{f_{i,j_i} (b_i - b_{i,\dada P})^{t_i}\right\}\right\|_{l^z(Y_i)}\right\rangle_{r_i,\dada P} \\ \label{eq3.3}
  &\times \left\langle\left(\prod_{i \in \tau}  \left|b_i(x) - b_{i,\dada P}\right|^{k_i - t_i}\right)g \right\rangle_{s',\dada P} \left.  \prod_{i \in \tau^c}\langle\|\{f_{i,j_i}\}\|_{l^z(Y_i)}\rangle_{r_i,\dada P}\right).
\end{align}
where $C$ is defined in \eqref{constants}. Next, we will construct the $\frac{1}{2}$-sparse family of cubes $\mathcal{F}$ iteratively. Setting $\mathcal{F}_0=\{Q\}$, given a collection of pairwise disjoint cubes $\mathcal{F}_j$, we focus on how to construct $\mathcal{F}_{j+1}$.

For $i \in \tau$, $t_i \in \{0, \ldots, k_i\}$. Fix a cube $P \in \mathcal{F}_j$, consider the sets

\begin{align*}
  \Omega_{{\bf t},\tau}(P)& := \left\{ x \in P : \left\|\left\{ \mathcal{T}_{\eta} \left( \left( \left( b_i - b_{i,\dada P} \right)^{t_i} f_{i,j_i} \chi_{\dada P} \right)_{i \in \tau}, \left( f_{i,j_i}\chi_{\dada P} \right)_{i \in \tau^c} \right) (x)\right\} \right\|_{l^z(Y)}\right. \\
  &  > \Phi_{\mathcal{T}_{\eta}, \vec{r},\tilde{r}} \left( \frac{1}{6|\tau|(|{\bf k}|+1)c_1c_2} \right)\\
  &\left. \times \mu(\dada P)^{\eta}  \prod_{i \in \tau}\langle\left\|\{f_{i,j_i} (b_i - b_{i,\dada P})^{t_i}\right\}\|_{l^z(Y_i)}\rangle_{r_i,\dada P}\prod_{i \in \tau^c}\langle\left\|\{f_{i,j_i}\}\right\|_{l^z(Y_i)}\rangle_{r_i,\dada P} \right\}
  \end{align*}  
and
\begin{align*}
  \M_{{\bf t},\tau}(P) &:= \left\{ x \in P : \left\| \left\{\M_{\mathcal{T}_{\eta},s,\beta,Y}^{\#} \left( \left( \left( b_i - b_{i,\dada P} \right)^{t_i} f_{i,j_i} \chi_{\dada P} \right)_{i \in \tau}, \left( f_{i,j_i}\chi_{\dada P} \right)_{i \in \tau^c} \right) (x) \right\}\right\|_{l^z} \right. \\
  &   > \Phi_{\M^{\#}_{\mathcal{T}_{\eta,s,\dada,Y}}, \vec{r},\tilde{r}} \left( \frac{1}{6|\tau|(|{\bf k}|+1)c_1c_2} \right) \\
&\left.\times\mu(\dada P)^{\eta}\prod_{i \in \tau}\left\langle\left\|\left\{f_{i,j_i} (b_i - b_{i,\dada P})^{t_i}\right\}\right\|_{l^z(Y_i)}\right\rangle_{r_i,\dada P}\prod_{i \in \tau^c}\langle\left\|\{f_{i,j_i}\}\right\|_{l^z(Y_i)}\rangle_{r_i,\dada P} \right\}.
  \end{align*}
  Let $c_1 \geq 1$, depending on $X, \d$ and $\dada$, be such that $\mu(\dada P) \leq$ $c_1 \mu(P)$. Then
\begin{align*}
   \mu(\Omega_{{\bf t},\tau}(P)) \leq \frac{1}{6|\tau|(|{\bf k}|+1)c_1c_2}\mu(\dada P) \leq \frac{1}{6|\tau|({\bf k}+1)c_2}\mu(P),
   \end{align*}
and the same bound holds for $\mu\left(\M_{{\bf t},\tau}(P)\right)$. Since the maximal operator $M_r$ is weak $L^r$-bounded with constant independent of $r$, there exists a $c_{\bf t}>0$ such that
\begin{align*}
M^1_{{\bf t},i}(P)&:=\left\{x \in P: \left\|\left\{M_{r_i}\left(\|\left(b_i - b_{i,\dada P} \right)^{t_i} f_{i,j_i} \chi_{\dada P}\|_{Y_i}\right)(x)\right\}_{ j_i}\right\|_{l^z}\right.\\
&\left. >c_{\bf t}\left\langle \left\| \left\{\left(b_i - b_{i,\dada P} \right)^{t_i} f_{i,j_i} \chi_{\dada P}\right\}_{ j_i}\right\|_{l^z(Y_i)} \right\rangle_{r_i, \dada P}\right\},
\end{align*}
and
\begin{align*}
  M^2_{{\bf t},i}(P)&:=\left\{x \in P: \left\|\{M_{r_i}\left(\|f_{i,j_i}\|_{Y_i}\right)(x)\}_{j_i}\right\|_{l^z}>c_{\bf t}\langle \| \{f_{i,j_i}\}\|_{l^z(Y_i)} \rangle_{r_i, \dada P}\right\}.
\end{align*}

We define $M_{{{\bf t},i},\tau}(P) := M^1_{{\bf t},i}(P) \cup  M^2_{{\bf t},i}(P)$ and it holds that
\begin{align*}
\mu(M_{{{\bf t},i},\tau}(P)) \leq \frac{1}{6|\tau|({\bf k}+1)c_2}\mu(P) .
\end{align*}
Therefore, setting
\begin{align*}
\Omega(P):=\bigcup_{i \in \tau}\bigcup_{t_i=0}^{k_i}\left(\Omega_{{\bf t},\tau}(P) \cup \mathcal{M}_{{\bf t},\tau}(P) \cup M_{{\bf t},i,\tau}(P)\right),
\end{align*}
we have $\mu(\Omega(P)) \leq \frac{1}{2c_2}\mu(P)$.

Based on above, we obtain pairwise disjoint cubes \( \S_P \subseteq \d(P) \) by applying the Calderón--Zygmund decomposition to \( \chi_{\Omega(P)} \) on \( P \) at the height $\lambda = \frac{1}{c_2}$, such that $\mu\left(\Omega(P) \backslash \bigcup\limits_{P^{\prime} \in \mathcal{S}_P} P^{\prime}\right)=0$ and for every $P^{\prime} \in \mathcal{S}_P$,
\begin{align}\label{eq3.4}
  \frac{1}{c_2}\mu(P^{\prime}) \leq \mu(P^{\prime} \cap \Omega(P)) \leq \frac{1}{2} \mu(P^{\prime}).
\end{align}
In particular, it follows that
\begin{align}\label{eq3.5}
\sum_{P^{\prime} \in \mathcal{S}_P} \mu(P^{\prime}) \leq {c_2}\mu(\Omega(P)) \leq \frac{1}{2}\mu(P).
\end{align}
We define $\mathcal{F}_{j+1}=\bigcup\limits_{P \in \mathcal{F}_j} \mathcal{S}_P$. Setting $\mathcal{F}=\bigcup\limits_{j=0}^{\infty} \mathcal{F}_j$, we note by \eqref{eq3.5} that $\mathcal{F}$ is $\frac{1}{2}$-sparse.

For $j \in \mathbb{N}$ and $P \in \mathcal{F}_j$, the proof of \eqref{eq3.3} is reduced to prove the following recursive inequality
\begin{align}\notag
  & \left\langle \left\| \left\{\mathcal{T}_{\eta,\tau_{}}^{{\bf b, k}}(\vec{\vf}\chi_{\beta P})(x) \right\}\right\|_{l^z(Y)},g\chi_P\right\rangle
  \\ \notag
  \le & C \sum\limits_{\mathbf{t}_\tau} \mu(P)^{\eta + 1} \prod_{i \in \tau}\left\langle\left\|\{f_{i,j_i} (b_i - b_{i,\dada P})^{t_i}\right\}\|_{l^z(Y_i)}\right\rangle_{r_i,\dada P}\\ \notag
  &\times \left\langle\left(\prod_{i \in\tau }  \left|b_i(x) - b_{i,\dada P}\right|^{k_i - t_i}\right)g \right\rangle_{s',\dada P}  \prod_{i \in \tau^c}\langle\left\|\{f_{i,j_i}\}\right\|_{l^z(Y_i)}\rangle_{r_i,\dada P}\\
   &\quad + \sum_{P^{\prime} \in \mathcal{F}_{j+1}: P^{\prime} \subseteq P} \left\langle \left\| \left\{\mathcal{T}_{\eta,\tau_{}}^{{\bf b, k}}(\vec{\vf}\chi_{\beta P'})(x) \right\}\right\|_{l^z(Y)},g\chi_{P'}\right\rangle.
\end{align}
Set $F_j:=\bigcup\limits_{P \in \mathcal{F}_j} P$. Noting that
\begin{align*}
   \left\langle \left\| \left\{\mathcal{T}_{\eta,\tau_{}}^{{\bf b, k}}(\vec{\vf}\chi_{\beta P})(x) \right\}\right\|_{l^z(Y)},g\chi_P\right\rangle \leq & \int_{P \backslash F_{j+1}}\left\|\left\{\mathcal{T}_{\eta,\tau_{}}^{{\bf b, k}}(\vec{\vf}\chi_{\dada P})\right\}\right\|_{l^z(Y)} \mid g \mid \\
& +\sum_{P^{\prime} \in \mathcal{F}_{j+1}: P^{\prime} \subseteq P} \int_{P^{\prime}}\left\|\left\{\mathcal{T}_{\eta,\tau_{}}^{{\bf b, k}}\left(\vec {\vf} \chi_{\dada P  \backslash \dada P'}\right)\right\}\right\|_{l^z(Y)}|g| \\
& +\sum_{P^{\prime} \in \mathcal{F}_{j+1}: P^{\prime} \subseteq P} \int_{P^{\prime}}\left\|\left\{\mathcal{T}_{\eta,\tau_{}}^{{\bf b, k}}\left(\vec {\vf} \chi_{\dada P'}\right)\right\}\right\|_{l^z(Y)}|g|,
\end{align*}
it thus suffices to show that
\begin{align}\notag
  &\quad \int_{P \backslash F_{j+1}}\left\|\left\{\mathcal{T}_{\eta,\tau_{}}^{{\bf b, k}}(\vec{\vf}\chi_{\dada P})\right\}\right\|_{l^z(Y)}\mid g \mid +\sum_{P^{\prime} \in \mathcal{F}_{j+1}: P^{\prime} \subseteq P} \int_{P^{\prime}}\left\|\left\{\mathcal{T}_{\eta,\tau_{}}^{{\bf b, k}}\left(\vec \vf \chi_{\dada P  \backslash \dada P'}\right)\right\}\right\|_{l^z(Y)}|g|\\ \notag
  & \le  C\sum\limits_{\mathbf{t}_\tau} \mu(P)^{\eta + 1} \prod_{i \in \tau}\left\langle\left\|\left\{f_i (b_i - b_{i,\dada P})^{t_i}\right\}\right\|_{l^z(Y_i)}\right\rangle_{r_i,\dada P}\\ \label{eq3.6}
&\quad\times\left\langle\left(\prod_{i \in \tau}  \left|b_i(x) - b_{i,\dada P}\right|^{k_i - t_i}\right)g \right\rangle_{s',\dada P}  \prod_{i \in \tau^c}\langle\left\|\{f_{i,j_i}\}\right\|_{l^z(Y_i)}\rangle_{r_i,\dada P}.
\end{align}

We first consider the first term on the left-hand side of \eqref{eq3.6}. Note that 
\begin{align*}
  &\quad \prod_{i \in \tau} (b_i(x) - b_i(y_i))^{k_i} \\
  &= \sum_{t_{\tau(1)}=0}^{k_{\tau(1)}} \sum_{t_{\tau(2)}=0}^{k_{\tau(2)}} \cdots \sum_{t_{\tau({|\tau|})}=0}^{k_{\tau({|\tau|})}} \Big( \prod_{i \in \tau}C_{k_i}^{t_i} (-1)^{t_i} \left(b_i(x) - b_{i,\dada P}\right)^{k_i - t_i} \left(b_i(y_i) - b_{i,\dada P}\right)^{t_i} \Big).
\end{align*}

For any $c \in \mathbb{C}$, we have by definition of $\Omega_{{\bf t},\tau}(P)$, combining with the definition of $\mathcal{T}_{\eta,\tau}^{{\bf b, k}}$, we conclude that
\begin{align}\notag
  &\quad \int_{P \backslash F_{j+1}}\left\|\left\{\mathcal{T}_{\eta,\tau_{}}^{{\bf b, k}}(\vec{\vf}\chi_{\dada P})\right\}\right\|_{l^z(Y)} \mid g \mid \\ \label{eq3.7}
&\leq \sum\limits_{\mathbf{t}_\tau} \int_{P \backslash F_{j+1}}\left\|\left\{\mathcal{T}_{\eta}\left(\left( \left( b_i - b_{i,\dada P} \right)^{t_i} f_{i,j_i} \chi_{\dada P} \right)_{i \in \tau}, \left( f_{i,j_i}\chi_{\dada P} \right)_{i \in \tau^c}\right)\right\}\right\|_{l^z(Y)}\prod_{i \in \tau}  \left|b_i - b_{i,\dada P}\right|^{k_i - t_i} \mid g \mid\\ \notag
& \leq C_1 \sum\limits_{\mathbf{t}_\tau}  \mu(P)^{\eta + 1} \\ \notag
&\times \prod_{i \in \tau}\left\langle\left\|\left\{f_{i,j_i} (b_i - b_{i,\dada P})^{t_i}\right\}\right\|_{l^z(Y_i)}\right\rangle_{r_i,\dada P}\left\langle\left(\prod_{i = 1}^m  \left|b_i - b_{i,\dada P}\right|^{k_i - t_i}\right)g \right\rangle_{1,\dada P}  \prod_{i \in \tau^c}\langle\left\|\left\{f_{i,j_i}\right\}\right\|_{l^z(Y_i)}\rangle_{r_i,\dada P},
\end{align}
where  $C_1:=c_1 \cdot 2^{\sum\limits_{i \in \tau} k_i}  \Phi_{\mathcal{T}_{\eta}, \vec{r},\tilde{r}} \left( \frac{1}{6|\tau|(|{\bf k}|+1)c_1c_2} \right)$.

Now consider the second term in \eqref{eq3.6}. Fix $P^{\prime} \in \mathcal{F}_{j+1}$ such that $P^{\prime} \subseteq P$ and denote
\begin{align*}
  \boldsymbol{\psi}_{{\bf k},\tau}(y) := \mathcal{T}_{\eta}\left(\left( \left( b_i - b_{i,\dada P} \right)^{t_i} f_{i,j_i} \chi_{\dada P \backslash \dada P'} \right)_{i \in \tau}, \left( f_{i,j_i}\chi_{\dada P \backslash \dada P'} \right)_{i \in \tau^c}\right)(y),\quad y \in X.
\end{align*}
Then, for $y \in P^{\prime}$ to be specified later we have
\begin{align}\notag
 &\int_{P^{\prime}}\left\|\left\{\mathcal{T}_{\eta,\tau_{}}^{{\bf b, k}}\left(\vec \vf \chi_{\dada P  \backslash \dada P^{\prime}}\right)\right\}\right\|_{l^z(Y)}|g| \leq 2^{\sum\limits_{i \in \tau} k_i} \sum\limits_{\mathbf{t}_\tau} \int_{P^{\prime}}\|\boldsymbol{\psi}_{k_i,\tau}\|_{l^z(Y)}\prod_{i \in \tau}  \left|b_i - b_{i,\dada P}\right|^{k_i - t_i} \mid g \mid\\ \label{eq3.8}
 &\quad \leq 2^{\sum\limits_{i \in \tau} k_i}\sum\limits_{\mathbf{t}_\tau} \int_{P^{\prime}}\|\boldsymbol{\psi}_{{\bf k},\tau}(x)- \boldsymbol{\psi}_{{\bf k},\tau}(y)\|_{l^z(Y)}\prod_{i \in \tau}  \left|b_i(x) - b_{i,\dada P}\right|^{k_i - t_i} \mid g(x) \mid d\mu(x)\\ \notag
 &\quad \quad + 2^{\sum\limits_{i \in \tau} k_i}\sum\limits_{\mathbf{t}_\tau} \int_{P^{\prime}}\| \boldsymbol{\psi}_{{\bf k},\tau}(y)\|_{l^z(Y)}\prod_{i \in \tau}  \left|b_i(x) - b_{i,\dada P}\right|^{k_i - t_i} \mid g(x) \mid d\mu(x).
\end{align}
 Consider the set
\begin{align*}
  {\widetilde{\Omega}}_{{\bf t},\tau}(P') &:= \left\{ x \in P' : \left\|\left\{ \mathcal{T}_{\eta} \left( \left( \left( b_i - b_{i,\dada P} \right)^{t_i} f_{i,j_i} \chi_{\dada P'} \right)_{i \in \tau}, \left( f_{i,j_i}\chi_{\dada P'} \right)_{i \in \tau^c} \right) (x)\right\} \right\|_{l^z(Y)} \right. \\
  &\quad  > \Phi_{\mathcal{T}_{\eta}, \vec{r},\tilde{r}} \left( \frac{1}{4c_1|\tau|(|{\bf k}|+1) } \right)\\
  &\left.\mu(\dada P')^{\eta}  \prod_{i \in \tau}\left\langle\left\|\left\{f_{i,j_i} (b_i - b_{i,\dada P'})^{t_i}\right\}\right\|_{l^z(Y_i)}\right\rangle_{r_i,\dada P'}\prod_{i \in \tau^c}\langle\left\|\{f_{i,j_i}\}\right\|_{l^z(Y_i)}\rangle_{r_i,\dada P'} \right\}
  \end{align*} 
  Set $\widetilde{\Omega}\left(P'\right):=\bigcup\limits_{i \in \tau}\bigcup\limits_{t_i=0}^{k_i} \widetilde{\Omega}_{{\bf t},\tau}\left(P'\right)$, for which we have $\mu\left(\widetilde{\Omega}\left(P'\right)\right) \leq \frac{1}{4}\mu\left(P'\right)$. Now, define the good part of the cube $P'$ as
  \begin{align*}
  G_{P'}:=P^{\prime} \backslash\left(\Omega(P) \cup \widetilde{\Omega}\left(P'\right)\right)
  \end{align*}
  Then, by \eqref{eq3.4}, we have
  \begin{align*}
  \mu(G_{P^{\prime}}) \geq\left(\frac{1}{2}-\frac{1}{4}\right)\mu(P^{\prime})=\frac{1}{4}\mu(P^{\prime})
  \end{align*}
  and for all $y \in G_{P^{\prime}}$,
\begin{align*}
  \left\| \boldsymbol{\psi}_{{\bf k},\tau}(y)\right\|_{l^z(Y)} &\leq \left\|\left\{ \mathcal{T}_{\eta} \left( \left( \left( b_i - b_{i,\dada P} \right)^{t_i} f_{i,j_i} \chi_{\dada P} \right)_{i \in \tau}, \left( f_{i,j_i}\chi_{\dada P} \right)_{i \in \tau^c} \right) (x) \right\}\right\|_{l^z(Y)}\\
&\quad + \left\| \left\{\mathcal{T}_{\eta} \left( \left( \left( b_i - b_{i,\dada P} \right)^{t_i} f_{i,j_i} \chi_{\dada P'} \right)_{i \in \tau}, \left( f_{i,j_i}\chi_{\dada P'} \right)_{i \in \tau^c} \right) (x) \right\}\right\|_{l^z(Y)}\\
&\leq \Phi_{\mathcal{T}_{\eta}, \vec{r},\tilde{r}} \left( \frac{1}{6|\tau|(|{\bf k}|+1)c_1c_2} \right)\\
&\times\mu(\dada P)^{\eta}  \prod_{i \in \tau}\left\langle\left\|\{f_{i,j_i} (b_i - b_{i,\dada P})^{t_i}\right\}\|_{l^z(Y_i)}\right\rangle_{r_i,\dada P}\prod_{i \in \tau^c}\langle\left\|\{f_{i,j_i}\}\right\|_{l^z(Y_i)}\rangle_{r_i,\dada P}\\
&\quad + \Phi_{\mathcal{T}_{\eta}, \vec{r},\tilde{r}} \left( \frac{1}{4c_1|\tau|(|{\bf k}|+1) } \right)\\
&\quad\times \mu(\dada P')^{\eta}  \prod_{i \in \tau}\left\langle\left\|\{f_{i,j_i} (b_i - b_{i,\dada P})^{t_i}\right\}\|_{l^z(Y_i)}\right\rangle_{r_i,\dada P'}\prod_{i \in \tau^c}\langle\left\|\{f_{i,j_i}\}\right\|_{l^z(Y_i)}\rangle_{r_i,\dada P'}.
\end{align*}
Further, by the definition of $M_{\bf{t},\tau}(P)$,
\begin{align*}
  &\quad\prod_{i \in \tau}\left\langle\left\|\{f_{i,j_i} (b_i - b_{i,\dada P})^{t_i}\right\}\|_{l^z(Y_i)}\right\rangle_{r_i,\dada P'}\prod_{i \in \tau^c}\langle\left\|\{f_{i,j_i}\}\right\|_{l^z(Y_i)}\rangle_{r_i,\dada P'} \\
  &\leq \prod_{i \in \tau}\left\langle\left\|\{f_{i,j_i} (b_i - b_{i,\dada P})^{t_i}\right\}\|_{l^z(Y_i)}\right\rangle_{r_i,\dada P}\prod_{i \in \tau^c}\langle\left\|\{f_{i,j_i}\}\right\|_{l^z(Y_i)}\rangle_{r_i,\dada P}.
\end{align*}
Hence, for all $y \in G_{P^{\prime}}$, we have
\begin{align*}
  &\quad \left\| \boldsymbol{\psi}_{{\bf k},\tau}(y)\right\|_{l^z(Y)} \\
  &\leq 2 c_{\bf{t}} \Phi_{\mathcal{T}_{\eta}, \vec{r},\tilde{r}} \left( \frac{1}{6|\tau|(|{\bf k}|+1)c_1c_2} \right)\\
  &\quad \times\mu(\dada P)^{\eta}  \prod_{i \in \tau}\left\langle\left\|\{f_{i,j_i} (b_i - b_{i,\dada P})^{t_i}\}\right\|_{l^z(Y_i)}\right\rangle_{r_i,\dada P}\prod_{i \in \tau^c}\langle\left\|\{f_{i,j_i}\}\right\|_{l^z(Y_i)}\rangle_{r_i,\dada P}.
\end{align*}
From this, integrating \eqref{eq3.8} over $y \in G_{P^{\prime}}$, using Hölder's inequality and the definition of the set $\M_{{\bf t},\tau}(P)$, we obtain
\begin{align*}
&\int_{P^{\prime}}\left\|\left\{\mathcal{T}_{\eta,\tau_{}}^{{\bf b, k}}\left(\vec \vf \chi_{\dada P  \backslash \dada P'}\right)\right\}\right\|_{l^z(Y)}|g|\\
  &\leq  2^{\sum\limits_{i \in \tau} k_i+2} \sum\limits_{\mathbf{t}_\tau}\left\|\mathrm{osc}_{s,Y}\left(\left\{\mathcal{T}_{\eta} \left( \left( \left( b_i - b_{i,\dada P} \right)^{t_i} f_{i,j_i} \chi_{X \backslash \dada P'} \right)_{i \in \tau}, \left( f_{i,j_i}\chi_{X \backslash \dada P'} \right)_{i \in \tau^c}\right) ;P' \right\} \right)\right\|_{l^z}\\
  &\quad \times \left\langle\left(\prod_{i \in \tau}  \left|b_i(x) - b_{i,\dada P}\right|^{k_i - t_i}\right)g \right\rangle_{s',P'}\mu(P')\\
  &\quad \quad + 2^{\sum\limits_{i \in \tau} k_i}\sum\limits_{\mathbf{t}_\tau} \int_{P^{\prime}}\left\| \boldsymbol{\psi}_{{\bf k},\tau}(y)\right\|_{L_y^\infty({l^z(Y)})}\prod_{i \in \tau}  \left|b_i(x) - b_{i,\dada P}\right|^{k_i - t_i} \mid g(x) \mid d\mu(x)\\
  &\leq   2^{\sum\limits_{i \in \tau} k_i +2} \cdot c_1 \cdot \Phi_{\M^{\#}_{\mathcal{T}_{\eta},t',\beta, Y} ,\vec{r},\tilde{r}} \left( \frac{1}{6|\tau|(|{\bf k}|+1)c_1c_2} \right)\\
  &\quad \times \sum\limits_{\mathbf{t}_\tau}\mu(P)^{\eta}  \prod_{i \in \tau}\left\langle\left\|\{f_{i,j_i} (b_i - b_{i,\dada P})^{t_i}\right\}\|_{l^z(Y_i)}\right\rangle_{r_i,\dada P}\prod_{i \in \tau^c}\langle\left\|\{f_{i,j_i}\}\right\|_{l^z(Y_i)}\rangle_{r_i,\dada P} \\
  &\quad \times \left\langle\left(\prod_{i \in \tau}  \left|b_i(x) - b_{i,\dada P}\right|^{k_i - t_i}\right)g \right\rangle_{s',P'}\mu(P')\\
  & + 2  \cdot 2^{\sum\limits_{i \in \tau} k_i+1} \cdot c_1 \cdot c_{{\bf t}} \Phi_{\mathcal{T}_{\eta}, \vec{r},\tilde{r}} \left( \frac{1}{6|\tau|(|{\bf k}|+1)c_1c_2} \right)\sum\limits_{\mathbf{t}_\tau}\mu(P)^{\eta}\\  
  & \times\prod_{i \in \tau}\left\langle\left\|\{f_{i,j_i} (b_i - b_{i,\dada P})^{t_i}\right\}\|_{l^z(Y_i)}\right\rangle_{r_i,\dada P}\prod_{i \in \tau^c}\langle\left\|\{f_{i,j_i}\}\right\|_{l^z(Y_i)}\rangle_{r_i,\dada P} \left\langle\left(\prod_{i \in\tau }  \left|b_i - b_{i,\dada P}\right|^{k_i - t_i}\right)g \right\rangle_{1,P'} \mu(P')
  \\
  &\leq C_2 \sum\limits_{\mathbf{t}_\tau}\mu(P)^{\eta}\\  
  &\quad\quad \times\prod_{i \in \tau}\left\langle\left\|\{f_{i,j_i} (b_i - b_{i,\dada P})^{t_i}\right\}\|_{l^z(Y_i)}\right\rangle_{r_i,\dada P}\prod_{i \in \tau^c}\langle\left\|\{f_{i,j_i}\}\right\|_{l^z(Y_i)}\rangle_{r_i,\dada P} \left\langle\left(\prod_{i \in \tau}  \left|b_i - b_{i,\dada P}\right|^{k_i - t_i}\right)g \right\rangle_{s',P'}\mu(P'), 
\end{align*}
where 
\begin{align*}
  C_2 =  2^{\sum\limits_{i \in \tau} k_i+1}\cdot c_1 \cdot \tilde{c}_{{\bf t}}  \left(\Phi_{{\mathcal{T}_{\eta}}, \vec{r},\tilde{r}} \left( \frac{1}{6|\tau|(|{\bf k}|+1)c_1c_2} \right) +  \Phi_{\M^{\#}_{\mathcal{T}_{\eta,s,\dada,Y}}, \vec{r},\tilde{r}} \left( \frac{1}{6|\tau|(|{\bf k}|+1)c_1c_2} \right)\right).
\end{align*}

By Hölder's inequality for sum and disjointness of $\mathcal{F}_{j+1}$, it follows that for any $q \in[1, \infty)$,
\begin{align*}
\sum_{P^{\prime} \in \mathcal{F}_{j+1}: P^{\prime} \subseteq P}\langle | h| \rangle_{q, P^{\prime}}\mu(P^{\prime}) \leq\langle | h| \rangle_{q, P}\mu(P).
\end{align*}
Therefore,
\begin{align*}
  &\sum_{P^{\prime} \in \mathcal{F}_{j+1}: P^{\prime} \subseteq P} \int_{P^{\prime}}\left\|\left\{\mathcal{T}_{\eta,\tau_{}}^{{\bf b, k}}\left(\vec \vf \chi_{\dada P  \backslash \dada P'}\right)\right\}\right\|_{l^z(Y)}|g|\\ 
  &\quad \quad \leq C_2 \sum\limits_{\mathbf{t}_\tau} \mu(P)^{\eta + 1}\prod_{i \in \tau}\left\langle\left\|\{f_{i,j_i} (b_i - b_{i,\dada P})^{t_i}\right\}\|_{l^z(Y_i)}\right\rangle_{r_i,\dada P}\\
  &\quad\quad\quad \times\left\langle\left(\prod_{i \in \tau}  \left|b_i(x) - b_{i,\dada P}\right|^{k_i - t_i}\right)g \right\rangle_{s',\dada P}  \prod_{i \in \tau^c}\langle\left\|\{f_{i,j_i}\}\right\|_{l^z(Y_i)}\rangle_{r_i,\dada P}.
\end{align*}
which will prove \eqref{eq3.6} via combining \eqref{eq3.7}. 
Then, \eqref{m+1.spa_0} can now be derived directly from \eqref{eq3.3} by using a covering argument: Lemma \ref{lem.covering}. 

In fact, from Proposition \ref{cubeeq} and Lemma \ref{lem.covering}, for any $P \in \mathcal{S}$ with center $z$ and sidelength $\delta^k$, we can find a $P^{\prime} \in \d^j$ for some $1 \leq j \leq N$ such that
\begin{align*}
\beta P=B\left(z, \beta C_0 \cdot \delta^k\right) \subseteq P^{\prime}, \quad \operatorname{diam}\left(P^{\prime}\right) \leq \gamma \beta C_0 \cdot \delta^k.
\end{align*}
Therefore there is a $c_1>0$ depending on $S$ and $\beta$ such that
\begin{align*}
\mu\left(P^{\prime}\right) \leq \mu\left(B\left(z, \gamma \beta C_0 \cdot \delta^k\right)\right) \leq c_1 \mu\left(B\left(z, c_0 \cdot \delta^k\right)\right) \leq c_1 \mu(P) .
\end{align*}
By defining $E_{P^{\prime}} := E_P$, we conclude that the collection of cubes $\mathcal{\tilde S} := \{P^{\prime} : P \in \mathcal{S}\}$ is $\frac{1}{2c_1}$-sparse. Since $\beta P \subseteq P^{\prime}$ and $\mu(P^{\prime}) \leq c_1 \mu(\beta P)$ for any $P \in \mathcal{S}$, it follows that
\begin{align*}
    \langle\|\{f_{i,j_i}\}\|_{l^z(Y_i)}\rangle_{r_i,\dada P} & \leq \langle\|\{f_{i,j_i}\}\|_{l^z(Y_i)}\rangle_{r_j, P'},\quad i \in \tau^c,\\
    \left\langle\left\|\left\{f_{i,j_i} (b_i - b_{i,\dada P})^{t_i}\right\}\right\|_{l^z(Y_i)}\right\rangle_{r_i,\dada P} & \leq \left\langle\left\|\left\{f_{i,j_i} (b_i - b_{i,\dada P})^{t_i}\right\}\right\|_{l^z(Y_i)}\right\rangle_{r_i,P'}, \quad i \in \tau, \\
     \left\langle\left(\prod_{i \in \tau}  \left|b_i(x) - b_{i,\dada P}\right|^{k_i - t_i}\right)g \right\rangle_{s',\dada P} & \leq \left\langle\left(\prod_{i \in \tau}  \left|b_i(x) - b_{i,\dada P}\right|^{k_i - t_i}\right)g \right\rangle_{s',P'}.\\
\end{align*}
Using \eqref{eq3.3} with above estimates, this proves the sparse domination in the conclusion of Theorem \ref{SDP}.\qedhere
\end{proof}

\section{\bf Multilinear fractional $\vec{r}$-type maximal operators}\label{Maximal.control.}


The $m$-linear fractional $\vec{r}$-type maximal operator $\M_{\eta, \vec{r}}$ is the main topic of this section. It not only extends the m--sublinear Hardy--
Littlewood maximal operator $M_{\vec{r}}$ define in \cite[Definition 2.4]{Nier2019},   but also provides an important characterization of the weight constant $[\vec{\omega}]_{(\vec{p},q),(\vec{r}, s)}$ for the sparse form, making it more convenient for us to use the sparse forms as tools for the main theorems.

\begin{definition}
Let $m \in \N$, $\eta \in [0,m)$, \( r_1, \ldots, r_m \in (0, \infty) \), we define the $m$-linear fractional $\vec{r}$-type maximal operator as
$$
\M_{\eta, \vec{r}}\left(f_1, \ldots, f_m\right)(x) := \sup_{B \subseteq X} \mu(B)^{\eta} \prod_{j=1}^m \left\langle f_j \right\rangle_{r_j, B} \chi_B(x),
$$
for \( f_j \in L_{\text{loc}}^{r_j} \), where the supremum is taken over all balls \( B \subseteq X \). 
Moreover, for a dyadic grid $\d$, the dyadic $m$-linear fractional $\vec{r}$-type maximal operator is defined by
$$
\M_{\eta, \vec{r}}^{\d}\left(f_1, \ldots, f_m\right)(x) := \sup_{\substack{Q \in \d}} \mu(Q)^{\eta} \prod_{j=1}^m \left\langle f_j \right\rangle_{r_j, Q} \chi_Q(x),
$$
for \( f_j \in L_{\text{loc}}^{r_j} \).
\end{definition}
\begin{definition}\label{def:weight}
Let $r_1, \ldots, r_m \in(0, \infty), s \in(0, \infty]$, and $p_1, \ldots, p_m \in(0, \infty]$ with $(\vec{r}, s) \preceq (\vec{p},q)$. Let $\omega_1, \ldots, \omega_m$ be weights and write $\omega=\prod\limits_{j=1}^m \omega_j$, $\vec{\omega}=\left(\omega_1, \ldots, \omega_m\right)$. We say that $\vec{\omega} \in A_{(\vec{p},q),(\vec{r}, s)}$ if
$$
[\vec{\omega}]_{(\vec{p},q),(\vec{r}, s)}:=\sup _{Q \in \d}\left(\prod_{j=1}^m\left\langle \omega_j^{-1}\right\rangle_{{\frac{1}{\frac{1}{r_j}-\frac{1}{p_j}}},Q}\right)\langle \omega\rangle_{\frac{1}{\frac{1}{q}-\frac{1}{s}}, Q}<\infty,
$$
\end{definition}

\begin{remark}\label{transform.}
 Set $r_{m+1} = s'$, $p_{m+1} = q'$, and $\omega_{m+1}=\omega^{-1}$, the constant for the weight class now takes the form
\begin{align*}
[\vec{\omega}]_{(\vec{p},q),(\vec{r}, s)}=\sup _Q \prod_{j=1}^{m+1}\left\langle \omega_j^{-1}\right\rangle_{{\frac{1}{\frac{1}{r_j}-\frac{1}{p_j}}},Q} &= \left[\left(\omega_1, \ldots, \omega_{m+1}\right)\right]_{\left(\left(p_1, \ldots, p_{m+1}\right),q\right),\left(\left(r_1, \ldots, r_{m+1}\right), \infty\right)}\\
    &:=[\boldsymbol{\omega}]_{(\boldsymbol{p},q),(r, \infty)}<\infty.\\
\end{align*}
The definition of $\boldsymbol{\omega}$ has already been explained in \eqref{cuti}.
\end{remark}
The following remark redefines the weight class. It can be obtained by imitating  \cite[Lemma 2.10 , 2.11]{Nier2019}.
\begin{remark}
Let $m \in \N$, $\eta \in [0,m)$, $\frac{1}{q}=\sum\limits_{i=1}^{m}  \frac{1}{p_i} - \mici$, $p_{m+1}=q'$, and $\omega_{m+1}=\left(\prod\limits_{j=1}^{m} \omega_j\right)^{-1}$.  Under the assumption of Definition \ref{def:weight}, then we have
$\sum\limits_{j=1}^{m+1} \frac{1}{p_j} =1+ \mici$ and $\prod\limits_{j=1}^{m+1} \omega_j=1$, moreover
\begin{itemize}
   \item Let $q=1$, then $\boldsymbol{\omega} \in A_{(\boldsymbol{p},q),(\boldsymbol{r}, \infty)}$ if and only if $v_1, \ldots, v_{m+1}$ are locally integrable and there is a constant $c>0$ such that for all cubes $Q$ we have
   \begin{align}\label{def.of.weighted.condi.}
\left(\prod_{j=1}^{m+1}\left\langle v_j\right\rangle_{1, Q}^{\frac{1}{r_j}}\right)\mu(Q)^{1 + \eta} \leq [\boldsymbol{\omega}]_{(\boldsymbol{p},q),(\boldsymbol{r}, \infty)} \prod_{j=1}^{m+1} v_j(Q)^{\frac{1}{p_j}}.
   \end{align}
   \item Let $Q$ be a cube and let $E \subseteq Q$ such that $\mu(Q) \leq \delta\mu(E)$ with $\delta \in (0,1)$. Then, we have
\begin{align}\label{def.of.weighted.condi._2}
\left(\prod_{j=1}^{m+1}\left\langle v_j\right\rangle_{1, Q}^{\frac{1}{r_j}}\right)\mu(Q)^{1 + \eta} \lesssim [\boldsymbol{\omega}]^{\varTheta}_{(\boldsymbol{p},q),(\boldsymbol{r}, \infty)} \prod_{j=1}^{m+1} v_j(E)^{\frac{1}{p_j}},
   \end{align}
   where $\varTheta = \max\limits_{j=1, \ldots, m+1} \left\{ \frac{\frac{1}{r_j}}{\frac{1}{r_j} - \frac{1}{p_j}} \right\} =\max \left\{ {\frac{{{p_1}}}{{{p_1} - {r_1}}}, \cdots ,\frac{{{p_m}}}{{{p_m} - {r_m}}},\frac{{q'}}{{q' - s'}}} \right\}$.
\end{itemize}
\end{remark}

\begin{figure}[!h]
	\begin{center}
		\begin{tikzpicture}
		\node (1) at(0,0) {$A^{}_{(\vec{p},q),(\vec{r}, s)}$};
		\node (2) at(2.5,2.5) {$A_{\vec p, q}$};
      \node (3) at(5,2.5) {$A_{p, q}$};
		\node (4) at(2.5,-2.5) {$A^{}_{\vec{p},(\vec{r}, s)}$};
		\node (5) at(5,-2.5) {$A_{\vec{p}}$};
		\node (6) at(7.5,0) {$A_{p}$};
		\draw[->] (1)--(2) node[midway, above] {\eqref{we_2}};
      \draw[->] (1)--(3) node[midway, above] {\eqref{we_5}};
      \draw[->] (1)--(5) node[midway, above] {\eqref{we_3}};
      \draw[->] (1)--(6) node[midway, above] {\eqref{we_6}};
		\draw[->] (1)--(4) node[midway, above] {\eqref{we_1}};
		\draw[->] (2)--(3);
		\draw[->] (2)--(5);
		\draw[->] (4)--(5);
      \draw[->] (4)--(6);
		\draw[->] (3)--(6);
		\draw[->] (5)--(6);
      \draw[->] (2)--(6);
		\end{tikzpicture}
	\end{center}
	\caption{The relationships between weights}\label{figure1}
\end{figure}   
\begin{remark}
   The definition of $\vec{\omega} \in A^{}_{(\vec{p},q),(\vec{r}, s)}$ condition is a very broad extension. $\vec{\omega} \in A^{}_{(\vec{p},q),(\vec{r}, s)}$ can be transformed into some well-known and fully studied classical weight conditions, in which case we give Fig.\ref{figure1} to clarify the relationships of these weights.

If $X=\mathbb{R}^n$, these weights are also back to the following classical weights.

\begin{list}{\rm (\theenumi)}{\usecounter{enumi}\leftmargin=1.2cm \labelwidth=1cm \itemsep=0.2cm \topsep=.2cm \renewcommand{\theenumi}{\arabic{enumi}}}
\item \label{we_1} If $p=q$, then $A^{}_{(\vec{p},q),(\vec{r}, s)}(X) = A^{}_{\vec{p},(\vec{r}, s)}(\rn)$ introduced in \cite{Nier2019}.
\item \label{we_2} If $\vec r = \{r_i\}_{i=1}^m= \{1\}_{i=1}^m$ and $s = \infty$, then
$A^{}_{(\vec{p},q),(\vec{r}, s)}(X) = A^{}_{\vec{p},q}(\rn)$ introduced in \cite{Perez2009}.
\item  \label{we_3} If $\vec r = \{r_i\}_{i=1}^m= \{1\}_{i=1}^m$, $s = \infty$, and $p=q$, then
$A^{}_{(\vec{p},q),(\vec{r}, s)} (X)= A^{}_{\vec{p}}(\rn)$ introduced in \cite{Perez2009}.
\item \label{we_5} If $\vec p = \{p_i\}_{i=1}^m= \{p\}_{i=1}^m$, $\vec r = \{r_i\}_{i=1}^m= \{1\}_{i=1}^m$ and $s = \infty$, then
$A^{}_{(\vec{p},q),(\vec{r}, s)}(X) = A^{}_{{p},q}(\rn)$ introduced in \cite{Muc1974}.
\item \label{we_6} If $\vec p = \{p_i\}_{i=1}^m= \{p\}_{i=1}^m$, $\vec r = \{r_i\}_{i=1}^m= \{1\}_{i=1}^m$, $s = \infty$, and $p=q$, then
$A^{}_{(\vec{p},q),(\vec{r}, s)}(\rn) = A^{}_{{p}}(\rn)$ introduced in \cite{Muc1972}.
\end{list}

\end{remark}

\begin{definition}
For $0 < \delta < 1$, a collection $\mathcal{S} \subseteq \mathcal{D}$ of dyadic cubes is \textit{$\delta$-sparse} if for each $Q \in \mathcal{S}$, there exists a measurable subset $E_Q \subseteq Q$ with $\mu(E_Q) \geq \delta \mu(Q)$, and the family $\{E_Q\}_{Q \in \mathcal{S}}$ are pairwise disjoint.

Let $m \in \N$, $\eta \in [0,m)$, $r_1, \ldots, r_m \in(0, \infty)$, for a sparse collection of cubes $\mathcal{S}$ we define the $m$-linear fractional sparse operator by
$$
\mathscr{A}_{\eta,\vec{r}, \mathcal{S}}\left(f_1, \ldots, f_m\right):=\sum_{Q \in \mathcal{S}}\mu(Q)^{\eta}\left(\prod_{j=1}^m\left\langle f_j\right\rangle_{r_j, Q}\right) \chi_Q
$$
and define the $m+1$-linear fractional sparse form  as 
$$
\mathcal{A}_{\eta,\vec{r}, \mathcal{S}}\left(f_1, \ldots, f_{m+1}\right):=\sum_{Q \in \mathcal{S}} \mu(Q)^{\eta+1} \prod_{j=1}^{m+1}\left\langle f_j\right\rangle_{r_j, Q}.
$$
\end{definition}

The following characterization is the main theorem of this section.
\begin{theorem}\label{Maximal_1}
Let $m \in \N$, $\eta \in [0,m)$, 
   $p_1, \ldots, p_{m+1} \in(0, \infty]$ with $\sum\limits_{j=1}^{m+1} \frac{1}{p_j}=1 + \eta$, $r_1, \ldots, r_{m+1} \in(0, \infty)$, satisfying $\boldsymbol{r} < \boldsymbol{p}$.
   Moreover, let $\omega_1, \ldots, \omega_{m+1}$ be weights satisfying $\prod\limits_{j=1}^{m+1} \omega_j=1$. Then the following are equivalent:
\begin{list}{\rm (\theenumi)}{\usecounter{enumi}\leftmargin=1cm \labelwidth=1cm \itemsep=0.2cm \topsep=.2cm \renewcommand{\theenumi}{\roman{enumi}}}
   \item $\boldsymbol{\omega} \in A_{(\boldsymbol{p},q),(\boldsymbol{r}, \infty)}$;
   \item $\left\|\M_{\eta,\boldsymbol{r}}\right\|_{L^{p_1}\left(\omega_1^{p_1}\right) \times \cdots \times L^{p_{m+1}}\left(\omega_{m+1}^{p_{m+1}}\right) \rightarrow L^{\frac{1}{1+\eta},\infty}}<\infty$;
   \item $\left\|\M_{\eta,\boldsymbol{r}}\right\|_{L^{p_1}\left(\omega_1^{p_1}\right) \times \cdots \times L^{p_{m+1}}\left(\omega_{m+1}^{p_{m+1}}\right) \rightarrow L^{\frac{1}{1+\eta}}}<\infty$;
   \item  $\sup\limits_{\mathcal{S} \subseteq \d}\left\| \mathcal{A}_{\eta,\boldsymbol{r}, \mathcal{S}}\right\|_{L^{p_1}\left(\omega_1^{p_1}\right) \times \cdots \times L^{p_{m+1}}\left(\omega_{m+1}^{p_{m+1}}\right) \rightarrow \mathbb{R}}<\infty$.
\end{list}
Moreover, we have
\begin{align}\label{Nier2.6}
   &\left\|\M_{\eta,\boldsymbol{r}}\right\|_{L^{p_1}\left(\omega_1^{p_1}\right) \times \cdots \times L^{p_{m+1}}\left(\omega_{m+1}^{p_{m+1}}\right) \rightarrow L^{\frac{1}{1+\eta},\infty}}   \approx [\boldsymbol{\omega}]_{(\boldsymbol{p},q),(\boldsymbol{r}, \infty)},\\ \label{Nier2.7}
   &   \left\|\M_{\eta,\boldsymbol{r}}\right\|_{L^{p_1}\left(\omega_1^{p_1}\right) \times \cdots \times L^{p_{m+1}}\left(\omega_{m+1}^{p_{m+1}}\right) \rightarrow L^{\frac{1}{1 + \eta}}}\approx \sup\limits_{\mathcal{S} \subseteq \d} \left\|\mathcal{A}_{\eta,\boldsymbol{r}, \mathcal{S}}\right\|^{1+\eta}_{L^{p_1}\left(\omega_1^{p_1}\right) \times \cdots \times L^{p_{m+1}}\left(\omega_{m+1}^{p_{m+1}}\right) \rightarrow \mathbb{R}},
\end{align}
where the implicit constants depend only on the $X$, and
\begin{align}\label{Nier2.8}
   \sup\limits_{\mathcal{S} \subseteq \d} \left\|\mathcal{A}_{\eta,\boldsymbol{r}, \mathcal{S}}\right\|_{L^{p_1}\left(\omega_1^{p_1}\right) \times \cdots \times L^{p_{m+1}}\left(\omega_{m+1}^{p_{m+1}}\right) \rightarrow \mathbb{R}} \lesssim_X c^*[\boldsymbol{\omega}]^{{\gamma}}_{(\boldsymbol{p},q),(\boldsymbol{r}, \infty)},
\end{align}
where the implicit constant depends on $X$ and
$$
c^*=\left(\prod_{i=1}^{m}\left(\frac{p_i}{p_i-r_i}\right)^{\frac{1}{r_i}}\right) \cdot \left(\frac{{q'}}{{q' - s'}}\right)^{\frac{1}{s'}}.
$$
In fact, the exponent $\varTheta:=\max \left\{ {\frac{{{p_1}}}{{{p_1} - {r_1}}}, \cdots ,\frac{{{p_m}}}{{{p_m} - {r_m}}},\frac{{q'}}{{q' - s'}}} \right\}$ is sharp.
\end{theorem}

Similar to \cite[Lemma 2.21]{Cen2408}, we can prove the following conclusion.
\begin{lemma}\label{dyadic.corl.}
   Let $m \in \N$, $\eta \in [0,m)$, $r_1, \ldots, r_m \in(0, \infty)$, there exists a finite family $\left\{\d^i\right\}_{i=1}^N$ of dyadic grids such that  
\begin{align*}
   \M_{\eta,\vec{r}}(\vec{f}) \approx \sum_{i=1}^{N} \M_{\eta,\vec{r}}^{\d^i}(\vec{f}).
\end{align*}
\end{lemma}
\begin{lemma}\label{control.operator.or.form}
Let $m \in \N$, $\eta \in [0,m)$, and $0<r_1, \ldots, r_m<\infty$. Then for each dyadic lattice $\d$ and all $f_j \in L_{loc}^{r_j}$, there is a sparse family $\mathcal{S} \subseteq \d$ such that
 \begin{align*}
\M_{\eta,\vec{r}}^{\d}\left(f_1, \ldots, f_m\right) \approx \sum_{Q \in \mathcal{S}} \mu(Q)^{\eta}\prod_{j=1}^m\left\langle f_j\right\rangle_{r_j, Q} \chi_{E_Q}
 \end{align*}
 pointwise almost everywhere, where $\frac{1}{r}=\sum_{j=1}^m \frac{1}{r_j}$. In particular we have
 \begin{align*}
\M_{\eta,\vec{r}}^{\d}\left(f_1, \ldots, f_m\right) \lesssim A_{\eta, \vec{r}, \mathcal{S}}\left(f_1, \ldots, f_m\right).
 \end{align*}
\end{lemma}

\begin{proof}[Proof of Theorem~\ref{Maximal_1}]
   We set up a transformation $v_j:=\omega_j^{-\frac{1}{\frac{1}{r_j}-\frac{1}{p_j}}}$ for $j \in \tau_{m+1}$ to make the proof more straightforward.

The proof will proceed in these steps: First, we prove equation \eqref{Nier2.6} to establish the equivalence between (i) and (ii). Then, we prove equation \eqref{Nier2.7} to show the equivalence of (iii) and (iv). Finally, since (iii) $\Rightarrow$ (ii) is clear, we complete the proof by demonstrating (i) $\Rightarrow$ (iv) using equation \eqref{Nier2.8}.

{\bf Proof of \eqref{Nier2.6}.}

To prove inequality \eqref{Nier2.6}, we first observe, by Lemma \ref{dyadic.corl.}, that it suffices to estimate \( \M_{\eta, \boldsymbol{r}}^{\d} \) for a dyadic lattice \( \d \).

Fix a cube \( Q \) and assume the \( v_j \) are locally integrable. Let \( 0 < \lambda < \mu(Q)^{1 + \eta} \prod_{j=1}^m \left\langle v_j \right\rangle_{1, Q}^{1/r_j} \). Define \( f_j := v_j^{1/r_j} \chi_Q \), which gives
\[
\M_{\eta,\boldsymbol{r}}\left(f_1, \ldots, f_{m+1}\right)(x) \geq \mu(Q)^{1 + \eta}\prod\limits_{j=1}^{m+1}\left\langle f_j\right\rangle_{r_j, Q} = \mu(Q)^{1 + \eta}\prod_{j=1}^{m+1}\left\langle v_j\right\rangle_{1, Q}^{\frac{1}{r_j}} > \lambda
\]
for all \( x \in Q \), implying \( Q \subseteq \left\{\M_{\eta,\boldsymbol{r}}\left(f_1, \ldots, f_{m+1}\right) > \lambda\right\} \). Thus,
\begin{align*}
   \lambda\mu(Q)^{1 + \eta} & \leq \lambda\mu\left(\left\{\M_{{\eta,\boldsymbol{r}}}\left(f_1, \ldots, f_{m+1}\right) > \lambda\right\}\right)^{1 + \eta} \\
   & \leq \left\|\M_{\eta,\boldsymbol{r}}\right\|_{L^{p_1}\left(\omega_1^{p_1}\right) \times \cdots \times L^{p_{m+1}}\left(\omega_{m+1}^{p_{m+1}}\right) \rightarrow L^{\frac{1}{1 + \eta}, \infty}} \prod_{j=1}^{m+1}\left\|f_j\right\|_{L^{p_j}\left(\omega_j^{p_j}\right)} \\
   & = \left\|\M_{\eta,\boldsymbol{r}}\right\|_{L^{p_1}\left(\omega_1^{p_1}\right) \times \cdots \times L^{p_{m+1}}\left(\omega_{m+1}^{p_{m+1}}\right) \rightarrow L^{\frac{1}{1 + \eta}, \infty}} \prod_{j=1}^{m+1} v_j(Q)^{\frac{1}{p_j}}.
\end{align*}
Taking the supremum over \( \lambda \), we obtain
\begin{align}\label{Nier2.14}
   \left(\prod_{j=1}^{m+1}\left\langle v_j\right\rangle_{1, Q}^{\frac{1}{r_j}}\right)\mu(Q)^{1 + \eta} \leq \left\|\M_{\eta,\boldsymbol{r}}\right\|_{L^{p_1}\left(\omega_1^{p_1}\right) \times \cdots \times L^{p_{m+1}}\left(\omega_{m+1}^{p_{m+1}}\right) \rightarrow L^{\frac{1}{1 + \eta}, \infty}} \prod_{j=1}^{m+1} v_j(Q)^{\frac{1}{p_j}}.
\end{align}
From \eqref{def.of.weighted.condi.}, this implies
\[
[\boldsymbol{\omega}]_{_{(\boldsymbol{p},q),(\boldsymbol{r}, \infty)}} \leq \left\|\M_{\eta,\boldsymbol{r}}\right\|_{L^{p_1}\left(\omega_1^{p_1}\right) \times \cdots \times L^{p_{m+1}}\left(\omega_{m+1}^{p_{m+1}}\right) \rightarrow L^{\frac{1}{1 + \eta}, \infty}}.
\]

Let \( \mathcal{F} \subseteq \d \) be a finite collection. For \( \lambda > 0 \) and \( f_j \in L^{p_j}(\omega_j^{p_j}) \), define \( \M_{\eta, \boldsymbol{r}}^{\mathcal{F}} \) as \( \M_{\eta, \boldsymbol{r}}^{\d} \), but with the supremum over \( Q \in \mathcal{F} \). Define
\[
\Omega_\lambda^{\mathcal{F}} := \left\{ \M_{\eta, \boldsymbol{r}}^{\mathcal{F}}(f_1, \ldots, f_{m+1}) > \lambda \right\},
\]
and similarly for \( \Omega_\lambda^{\d} \). Let \( \mathcal{P} \) be the collection of cubes \( Q \in \mathcal{F} \) satisfying
\[
\mu(Q)^{\eta} \prod_{j=1}^{m+1} \left\langle f_j \right\rangle_{r_j, Q} > \lambda
\]
with no dyadic ancestors in \( \mathcal{F} \). Using the identity
\[
\langle h \rangle_{r, Q} = \left\langle h u^{-\frac{1}{r}} \right\rangle_{r, Q}^u \langle u \rangle_{1, Q}^{\frac{1}{r}},
\]
and \eqref{def.of.weighted.condi.}, \( \mathcal{P} \) decomposes \( \Omega_\lambda^{\mathcal{F}} \). We compute:
\begin{align*}
\lambda\mu(\Omega_\lambda^{\mathcal{F}}) &=\sum_{Q \in \mathcal{P}} \lambda\mu(Q) \leq \sum_{Q \in \mathcal{P}}\left(\prod_{j=1}^{m+1}\left\langle f_j\right\rangle_{r_j, Q}\right)\mu(Q)^{1 + \eta} \\
   &=\sum_{Q \in \mathcal{P}}\left(\prod_{j=1}^{m+1}\left\langle f_j v_j^{-\frac{1}{r_j}}\right\rangle_{r_j, Q}^{v_j}\left\langle v_j\right\rangle_{1, Q}^{\frac{1}{r_j}}\right)\mu(Q)^{1 + \eta} \\
   &\leq[\boldsymbol{\omega}]_{{(\boldsymbol{p},q),(\boldsymbol{r}, \infty)}} \sum_{Q \in \mathcal{P}} \prod_{j=1}^{m+1}\left\langle f_j v_j^{-\frac{1}{r_j}}\right\rangle_{r_j, Q}^{v_j} v_j(Q)^{\frac{1}{p_j}} \\
   &\leq[\boldsymbol{\omega}]_{{(\boldsymbol{p},q),(\boldsymbol{r}, \infty)}} \sum_{Q \in \mathcal{P}} \prod_{j=1}^{m+1}\left(\int_Q\left|f_j\right|^{p_j} v_j^{p_j\left(\frac{1}{p_j}-\frac{1}{r_j}\right)} \mathrm{d} x\right)^{\frac{1}{p_j}} \\
   &\leq[\boldsymbol{\omega}]_{_{(\boldsymbol{p},q),(\boldsymbol{r}, \infty)}} \prod_{j=1}^{m+1}\left\|f_j\right\|_{L^{p_j}\left(\omega_j^{p_j}\right)}.
\end{align*}
Here, Hölder's inequality (with \( r_j \leq p_j \)) is applied in the fourth step, and again to the sum in the final step. Covering \( \d \) with finite collections and taking the supremum over \( \lambda > 0 \), we conclude
\[
\left\|\M_{\eta,\boldsymbol{r}}^{\d}\right\|_{L^{p_1}\left(w_1^{p_1}\right) \times \cdots \times L^{p_{m+1}}\left(w_{m+1}^{p_{m+1}}\right) \rightarrow L^{\frac{1}{1 + \eta}, \infty}} \leq  [\boldsymbol{\omega}]_{_{(\boldsymbol{p},q),(\boldsymbol{r}, \infty)}}.
\]

Replace \( v_j \) with \( \left(v_j^{-1} + \varepsilon\right)^{-1} \) for \( \varepsilon > 0 \), which are bounded and locally integrable. Rearrange \eqref{Nier2.14} and apply the Monotone Convergence Theorem as \( \varepsilon \to 0 \). This yields the converse inequality, completing the proof of \eqref{Nier2.6}.

{\bf Proof of \eqref{Nier2.7}.}

For \eqref{Nier2.7}, let $f_j \in L^{p_j}\left(\omega_j^{p_j}\right)$ and let $\d$ be a dyadic grid. We claim that
\begin{align}\label{121_4}
   \left\|\M^{\d}_{\eta,\boldsymbol{r}}(f_1,\ldots,f_{m+1})\right\|_{L^{\frac{1}{1+\eta}}} \leq \left(\mathcal{A}_{\frac{\eta}{1+\eta},\boldsymbol{r}(1+\eta),\mathcal{S}}(g_1,\ldots,g_{m+1})\right)^{1+\eta},
\end{align}
where $g_j=f_j^{\frac{1}{1+\eta}}$. Indeed, notice that
\begin{align*}
   \left\|\M^{\d}_{\eta,\boldsymbol{r}}(f_1,\ldots,f_{m+1})\right\|_{L^{\frac{1}{1+\eta}}} 
   &= \left(\int_X \sup\limits_{Q \in \d} \mu(Q)^{\frac{\eta}{1 + \eta}}\prod_{j=1}^{m}\langle f_j^{\frac{1}{1+\eta}} \rangle_{r_j(1 + \eta),Q}\chi_Q\right)^{1 + \eta}\\
   &=\left\|\M^{\d}_{\frac{\eta}{1+\eta},\boldsymbol{r}(1+\eta)}(g_1,\ldots,g_{m+1})\right\|^{1 + \eta}_{L^1}\\
   &\leq \left\|\mathcal{A}_{\frac{\eta}{1+\eta},\boldsymbol{r}(1+\eta),\mathcal{S}}(g_1,\ldots,g_{m+1})\right\|^{1 + \eta}_{L^1}\\
   &\leq \left(\mathcal{A}_{\frac{\eta}{1+\eta},\boldsymbol{r}(1+\eta),\mathcal{S}}(g_1,\ldots,g_{m+1})\right)^{1+\eta}.
\end{align*}
Based on this, we begin proof \eqref{Nier2.7}.

First we give follow assumptions
\begin{align}\label{121_1}
\left|   \mathcal{A}_{\frac{\eta}{1+\eta},\boldsymbol{r}(1+\eta),\mathcal{S}}(g_1,\ldots,g_{m+1})\right| \leq C_1 \prod_{i=1}^{m + 1}\|g_i\|_{L^{p_i(1 +\eta)}(\omega_i^{p_i})},
\end{align}
and 
\begin{align}\label{121_2}
   \left\|\M^{\d}_{\eta,\boldsymbol{r}}(f_1,\ldots,f_{m+1})\right\|_{L^{\frac{1}{1+\eta}}} \leq C_2 \prod_{i=1}^{m}\|f_i\|_{L^{p_i}(\omega_i^{p_i})},
\end{align}
where $C_1,C_2$ are the smallest constants that makes the inequality hold.
Then the proof of \eqref{Nier2.7} is equivalent to proving $C_1^{1 + \eta} \approx C_2$.

For \eqref{121_1}, we can  rewritten as
\begin{align}\label{121_11}
   \left|\mathcal{A}_{\frac{\eta}{1+\eta},\boldsymbol{r}(1+\eta),\mathcal{S}}(g_1,\ldots,g_{m+1})\right|^{(1 +\eta)} \leq C_1^{(1 +\eta)} \prod_{i=1}^{m + 1}\|f_i\|_{L^{p_i}(\omega_i^{p_i})},
   \end{align}

Therefore, it is quite obvious that $C_2 \lesssim C_1^{1 + \eta}$ can be derived from \eqref{121_4}, \eqref{121_2} and \eqref{121_11}. For the converse inequality, we estimate
\begin{align*}
   \left|   \mathcal{A}_{\frac{\eta}{1+\eta},\boldsymbol{r}(1+\eta),\mathcal{S}}(g_1,\ldots,g_{m+1})\right| &\leq 2^{\frac{\eta}{1 + \eta} +1 }  \sum_{Q \in \mathcal{S}} \mu(E_Q)^{\frac{\eta}{1 + \eta} +1 }\prod_{i=1}^{m+1}\langle g_i \rangle_{r_i(1+\eta),Q}\\
   &\leq 2^{\frac{\eta}{1 + \eta} +1 } \sum_{Q \in \mathcal{S}}\int_{E_Q} \M^{\d}_{\frac{\eta}{1+\eta},\boldsymbol{r}(1+\eta)}(g_1,\ldots,g_{m+1})\\
   & \leq 2^{\frac{\eta}{1 + \eta} +1 }  \left\|\M^{\d}_{\frac{\eta}{1+\eta},\boldsymbol{r}(1+\eta)}(g_1,\ldots,g_{m+1})\right\|_{L^1}\\
   & = 2^{\frac{\eta}{1 + \eta} +1 } \left\|\M^{\d}_{\eta,\boldsymbol{r}}(f_1,\ldots,f_{m+1})\right\|^{\frac{1}{1+\eta}}_{L^{\frac{1}{1+\eta}}}.
\end{align*}
This demonstrates the $ C_1^{1 + \eta} \lesssim C_2$. In this way, we have proved \eqref{Nier2.7}.

{\bf Proof of \eqref{Nier2.8}.}

To prove \eqref{Nier2.8} and establish the implication (i) $\Rightarrow$ (iv), we observe that from \eqref{def.of.weighted.condi._2}, it follows that for a sparse collection $\mathcal{S}$ in a dyadic grid $\d$, and for $\varTheta = \max\limits_{j=1, \ldots, m+1} \left\{ \frac{\frac{1}{r_j}}{\frac{1}{r_j} - \frac{1}{p_j}} \right\} =\max \left\{ {\frac{{{p_1}}}{{{p_1} - {r_1}}}, \cdots ,\frac{{{p_m}}}{{{p_m} - {r_m}}},\frac{{q'}}{{q' - s'}}} \right\}$, we have
\begin{align*}
   \mathcal{A}_{\eta,\boldsymbol{r}, \mathcal{S}}(g_1,\ldots,g_{m+1}) & = \sum_{Q \in \mathcal{S}}\left(\prod_{j=1}^{m+1}\left\langle f_j\right\rangle_{r_j, Q}\right)\mu(Q)^{1 + \eta}\\
   &=\sum_{Q \in \mathcal{S}}\left(\prod_{j=1}^{m+1}\left\langle f_j v_j^{-\frac{1}{r_j}}\right\rangle_{r_j, Q}^{v_j}\left\langle v_j\right\rangle_{1, Q}^{\frac{1}{r_j}}\right)\mu(Q)^{1 + \eta}\\
   &\lesssim [\boldsymbol{\omega}]^{\varTheta}_{(\boldsymbol{p},q),(\boldsymbol{r}, \infty)} \sum_{Q \in \mathcal{S}} \prod_{j=1}^{m+1}\left\langle f_j v_j^{-\frac{1}{r_j}}\right\rangle_{r_j, Q}^{v_j} v_j\left(E_Q\right)^{\frac{1}{p_j}}\\
   &\leq [\boldsymbol{\omega}]^{\varTheta}_{(\boldsymbol{p},q),(\boldsymbol{r}, \infty)} \sum_{Q \in \mathcal{S}} \prod_{j=1}^{m+1}\left(\int_{E_Q} M_{r_j}^{v_j, \d}\left(f_j v_j^{-\frac{1}{r_j}}\right)^{p_j} v_j \mathrm{~d} x\right)^{\frac{1}{p_j}}\\
   & \leq [\boldsymbol{\omega}]^{\varTheta}_{(\boldsymbol{p},q),(\boldsymbol{r}, \infty)} \prod_{j=1}^{m+1}\left\|M_{r_j}^{v_j, \d}\left(f_j v_j^{-\frac{1}{r_j}}\right)\right\|_{L^{p_j}\left(v_j\right)}\\
   &\lesssim c^*[\boldsymbol{\omega}]^{\varTheta}_{(\boldsymbol{p},q),(\boldsymbol{r}, \infty)}\prod_{j=1}^{m+1}\left\|f_j\right\|_{L^{p_j}\left(\omega_j^{p_j}\right)},
\end{align*}
where we used Hölder's inequality on the sum in the fifth step.
In the final step, we used that the weighted dyadic maximal operator $M_r^{u, \d} h := \sup_{Q \in \d} \langle h \rangle_{r, Q}^u \chi_Q$ is bounded on $L^q(u)$ for $q > r$, with the bound $\left[\frac{\frac{1}{r}}{\frac{1}{r} - \frac{1}{q}}\right]^{\frac{1}{r}}$, which is uniform in $u$. This uniform estimate over the sparse collection $\mathcal{S}$ gives \eqref{Nier2.8}, completing the proof.

\end{proof}

\section{\bf  New Bloom type weighted estimates for {\tt  ($m+1$)-linear fractional   sparse forms}}\label{Bloom.estimate.}

In this section, we provide a Bloom type estimate for the $\TB$. To begin, we simplify the expression by utilizing the norm equivalence from Cascante-Ortega-Verbitsky. Next, we estimate it through the weighted quantitative boundedness of the sparse form $\mathcal{A}_{\eta,\vec{r}, \mathcal{S}}$. Finally, we address the problem using the method from \cite{Lerner2018}.

The following theorem is the main result of this section, where (1) is called maximal Bloom type weight method which improves \cite[Thoerem 1.24]{CenSong2412}, and (2) is a new Bloom type estimate different from (1). 

\begin{theorem}\label{quan.main}
Let $m \in \N$, $\eta \in [0,m)$, $i \in \tau_m$, $j \in \tau$, and $k \in \tau'$, with $\tau' \subseteq \tau \subseteq \tau_m$, $\mathbf{k}$ be a multi-index, and $\mathcal{S} \subseteq \d$ be a sparse family.
Let $r_i, s' \in [1, \infty)$, and $p_i,q\in (1, \infty)$ with $\frac{1}{q}:= \sum\limits_{j=1}^m\frac{1}{p_j} -  \eta$. Given $\omega_i$, $\nu_j$, and $u_i$ are  weights. Set $\nu_k=(u_k /\omega_k)^{-\frac{r_k}{{\tt a} + \gamma_k - 1}}$ with ${\tt a}:=\lfloor k_kr_k\rfloor$, $\gamma_k:={k_kr_k}-({\tt a}-1) \in[1,2)$. 
   \begin{enumerate}
      \item 
If $b_k \in {\rm BMO}_{\nu_k}$ and $b_{\ell} \in {\rm BMO}_{\nu_0}$ with $\ell \in \tau \backslash \tau'$, where $\nu_0 := \max\limits_{j} \{\nu_{\ell}\}$. Define $v_1=\prod\limits_{i=1}^{m} u_i$ and $\nu_0:=(v_1/\omega^{2\gamma_1 - 3})^{\frac{s'}{\gamma_1 + {\tt L} - 1}}$, where ${\tt L} := \lfloor s' \cdot \sum\limits_{\ell} k_{\ell} \rfloor$ and $\gamma_1 := s' \cdot \sum\limits_{\ell} k_{\ell} - ({\tt L} - 1) \in [1, 2)$.
 Then 
\begin{align*}
   &\quad \sup_{\S \subseteq \d}\left\|\TB\right\|_{\prod\limits_{i=1}^{m}L^{p_i}\left(\omega_i^{p_i}\right) \times L^{q'}(\omega^{-q'})\rightarrow \mathbb{R}} \\
   &\lesssim  \sup\limits_{\mathcal{S} \subseteq \d} \left\|\Lambda_{\mici,\boldsymbol{r}_{\tt L}, \mathcal{S}}\right\|_{\prod\limits_{i=1}^{m}L^{p_i}\left(u_i^{p_i}\right) \times L^{q'}(v_1^{-q'})\rightarrow \mathbb{R}} \cdot \prod_{i \in \tau'}\|b_{i}\|^{k_{i}}_{{\rm{BMO}}_{\nu_{i}}} \prod_{j \in \tau \backslash \tau'}\|b_{j}\|^{k_{j}}_{{\rm BMO}_{\nu_0}},
\end{align*}
where $ \boldsymbol{r}_{\tt L}=\Big((r_{i})_{i \in \tau_m},s'\Big)$.

      \item If $b_{i_0} \in {\rm BMO}_{\nu_{i_0}}$ for some $i_0 \in \tau \backslash \tau'$, and $b_{\ell} \in {\rm BMO}$ for $\ell \in (\tau \backslash \tau') \backslash \{i_0\}$. Define $v_2=\prod\limits_{i=1}^{m} u_i$ and ${\nu _{{i_0}}}: = {\left( {{{{v_2}}}/{{{\omega ^{2{\gamma _2} - 3}}}}} \right)^{\frac{{2s'}}{{{\gamma _2} + {\tt {\tilde L}} - 1}}}}$, where ${\tt {\tilde L}} := \lfloor 2k_{i_0}s' \rfloor$ and $\gamma_2 := 2k_{i_0}s' - ({\tt {\tilde L}} - 1) \in [1, 2)$.
 Then 
      \begin{align*}
   &\quad \sup_{\S \subseteq \d}\left\|\TB\right\|_{\prod\limits_{i=1}^{m}L^{p_i}\left(\omega_i^{p_i}\right) \times L^{q'}(\omega^{-q'})\rightarrow \mathbb{R}}\\
   &\lesssim  \sup\limits_{\mathcal{S} \subseteq \d} \left\|\Lambda_{\mici,\boldsymbol{r}_{\tt {\tilde L}}, \mathcal{S}}\right\|_{\prod\limits_{i=1}^{m}L^{p_i}\left(u_i^{p_i}\right) \times L^{q'}(v_2^{-q'})\rightarrow \mathbb{R}} \cdot \|b_{i_0}\|^{k_{i_0}}_{{\rm{BMO}}_{\nu_{i_0}}}\prod_{i \in \tau'}\|b_{i}\|^{k_{i}}_{{\rm{BMO}}_{\nu_{i}}}\prod_{\substack{j \in \tau \backslash \tau'\\ j \neq i_0}}\|b_{j}\|_{\rm{BMO}}^{k_{j}},
\end{align*}
where $ \boldsymbol{r}_{\tt {\tilde L}}=\Big((r_{i})_{i \in \tau_m},2s'\Big)$.
   \end{enumerate}
\end{theorem}

\begin{remark}\label{Max.method}
~~

Let $\omega_1,\omega_2,\omega_3$ are weights, $x_1,\ldots,x_5$ are constants, $\varTheta:=\max \left\{ {\frac{{{p_1}}}{{{p_1} - {r_1}}}, \cdots ,\frac{{{p_m}}}{{{p_m} - {r_m}}},\frac{{q'}}{{q' - s'}}} \right\}$, and we define the constant
\begin{align}\notag
&\quad\mathcal{C}_{\omega_1,\omega_2,\omega_3}(x_1,\cdots,x_5)\\ \label{weight.constants}
&=:\left(\left[\omega_1^{x_3}\right]^{\frac{x_1 - 2}{2}}_{A_{x_2}} \left[\omega_2^{x_4} \cdot \omega_2^{x_5}\right]^{\frac{x_1}{2}}_{A_{x_2}} \right)^{\max \left\{1, \frac{1}{x_2-1}\right\}} [\omega_2]_{A_{x_2}}^{\max(1,\frac{1}{x_2-1})}[\omega_2^{x_3} \cdot \omega_3^{-|x_2|}]_{A_{x_2}}^{(2 - \gamma)\max(1,\frac{1}{x_2-1})}.
\end{align}

{\bf Further 1:} 
Suppose that  $\vec u \in A_{(\vec{ p},q),(\vec{r},s)}(X)$, 
under the assumption of (1) in Theorem \ref{quan.main}, we have
   \begin{align*}
\sup\limits_{\mathcal{S} \subseteq \d} \left\|\TB\right\|_{\prod\limits_{i=1}^{m}L^{p_i}\left(\omega_i^{p_i}\right)\times L^{q'}(\omega^{-q'}) \rightarrow \mathbb{R}} \lesssim \mathcal{C}_{1},
   \end{align*}
   where
\begin{align*}
   \mathcal{C}_{1} = &\prod_{i \in \tau'}\|b_{i}\|^{k_{i}}_{{\rm{BMO}}_{\nu_{i}}} \prod_{j \in \tau \backslash \tau'}\|b_{j}\|^{k_{j}}_{{\rm BMO}_{\nu_0}} \times [\vec{u}]^{\varTheta}_{(\vec{p},q),(\vec{r}, s)} \\
   &\times \prod_{i \in \tau'}^{} \left(\mathcal{C}_{u_i,\omega_i,\nu_i}\left({\tt a},\frac{p_i}{r_i},p_i,p_i,-\frac{\gamma_i p_i}{r_i}\right)\right)^{\frac{1}{r_i}} \times  \left(\mathcal{C}_{v_1,\omega,\nu_0}\left({\tt L},-\frac{q'}{s'},-q',q'(3-2\gamma_1),-\frac{\gamma_1 q'}{s'}\right)\right)^{\frac{1}{s'}}.
\end{align*}

{\bf Further 2:} Under the assumption of (2) in Theorem \ref{quan.main}, we have
   \begin{align*}
\sup\limits_{\mathcal{S} \subseteq \d} \left\|\TB\right\|_{\prod\limits_{i=1}^{m}L^{p_i}\left(\omega_i^{p_i}\right)\times L^{q'}(\omega^{-q'}) \rightarrow \mathbb{R}} \lesssim\mathcal{C}_{2},
   \end{align*}
where
\begin{align*}
   \mathcal{C}_{2} = &\|b_{i_0}\|^{k_{i_0}}_{{\rm{BMO}}_{\nu_{i_0}}}\prod_{i \in \tau'}\|b_{i}\|^{k_{i}}_{{\rm{BMO}}_{\nu_{i}}}\prod_{\substack{j \in \tau \backslash \tau'\\ j \neq i_0}}\|b_{j}\|_{\rm{BMO}}^{k_{j}} \times [\vec{u}]^{\varTheta}_{(\vec{p},q),(\vec{r}, s)} \\
   &\times \prod_{i \in \tau'}^{} \left(\mathcal{C}_{u_i,\omega_i,\nu_i}\left({\tt a},\frac{p_i}{r_i},p_i,p_i,-\frac{\gamma_i p_i}{r_i}\right)\right)^{\frac{1}{r_i}} \times  \left(\mathcal{C}_{v_2,\omega,\nu_{i_0}}\left({\tt {\tilde L}},-\frac{q'}{2s'},-q',q'(3-2\gamma_2),-\frac{\gamma_2 q'}{2s'}\right)\right)^{\frac{1}{s'}}.
\end{align*}
\end{remark}
Before presenting the formal proof, we first introduce the following results that will be useful for estimating the mean in the subsequent proof steps. Specifically, we begin with the two norm equivalences from Cascante-Ortega-Verbitsky \cite{Cas2004}.
\begin{lemma}\label{le:eqnorm_}
 Let $p \in[1, \infty)$, let $\omega$ be a weight and let $\lambda_Q \geq 0$ for all $Q \in \d$. Then
$$
\left\|\sum_{Q \in \d} \lambda_Q \chi_Q\right\|_{L^p(\omega)} \approx\left(\sum_{Q \in \d} \lambda_Q\left(\frac{1}{\omega(Q)} \sum_{Q^{\prime} \in \d, Q^{\prime} \subseteq Q} \lambda_{Q^{\prime}} \omega\left(Q^{\prime}\right)\right)^{p-1} \omega(Q)\right)^{1 / p}
$$
\end{lemma}
\begin{lemma}[\cite{Cas2004}, (2.4)]\label{killr_pre}
    For any $p \in [1, \infty)$ and non-negative coefficients $\lambda_Q$ associated with each $Q \in \d$, the following inequality holds:
    $$
    \left(\sum_{Q \in \d} \lambda_Q \chi_Q\right)^p \leq p \sum_{Q \in \d} \lambda_Q \chi_Q \left( \sum_{\substack{Q^{\prime} \in \d \\ Q^{\prime} \subseteq Q}} \lambda_{Q^{\prime}} \chi_{Q^{\prime}} \right)^{p-1}.
    $$
\end{lemma}
Then we use the following proposition to further simplify the proof process.
\begin{proposition}\label{premainproof}
   Let $1 \leq r<  \infty, m \in \mathbb{N}$. For any sparse family $\mathcal{S} \subseteq \d, f \in L^p(\mu)$, $v$ is a weight, we have
   \begin{align*}
      \frac{1}{\mu(Q)}\int_Q\left(\sum_{P \in \mathcal{S}^{\prime}: P \subseteq Q} \frac{v(P)^{}}{\mu(P)} \chi_P\right)^{r m}|f|^r \lesssim \lla A_{\mathcal{S}^{\prime}, v}^{{k } - 1}(h)\rra_{Q},
   \end{align*}
where $$h =  A_{\mathcal{S}^{\prime}}\left(|f|^r\right)^{2-\gamma} \cdot  A_{\mathcal{S}^{\prime}}\left( A_{\mathcal{S}^{\prime}}\left(|f|^r\right) v^{}\right)^{\gamma-1} \cdot v^{}$$ 
with $k:=\lfloor r m\rfloor$, $\gamma:=r m-(k-1) \in[1,2).$
\end{proposition}

\begin{proof}[Proof of {Proposition} \ref{premainproof}]
By Lemma \ref{zhang:6.1}, there exists a sparse family $\mathcal{S}' \subseteq \mathcal{D}$ containing $\mathcal{S}$ such that for each $Q \in \mathcal{S}'$,
\begin{align*}
\int_Q \left(\sum_{P \subseteq Q} \frac{v(P)}{\mu(P)} \chi_P \right)^{rm} |f|^r \lesssim \sum_{P_{k-1} \subseteq \cdots \subseteq Q} \prod_{i=1}^{k-1} \frac{v(P_i)}{\mu(P_i)} \int_{P_{k-1}} \left(\sum_{P_k \subseteq P_{k-1}} \frac{v(P_k)}{\mu(P_k)} \chi_{P_k} \right)^\gamma |f|^r
\end{align*}
holds a.e. on $Q$, where $k := \lfloor rm \rfloor$, $\gamma := rm - (k-1) \in [1,2)$, and all $P_i \in \mathcal{S}'$.  

Define $A_{\mathcal{S}',v}(\varphi) := A_{\mathcal{S}'}(\varphi)v$ with $A_{\mathcal{S}',v}^j$ as its $j$-th iteration. By Lemma \ref{le:eqnorm_} and Minkowski's inequality:
\begin{align*}
\int_{P_{k-1}} \left(\sum_{P_k \subseteq P_{k-1}} \frac{v(P_k)}{\mu(P_k)} \chi_{P_k} \right)^\gamma |f|^r 
&\lesssim \sum_{P_k} v(P_k) \langle |f|^r \rangle_{P_k}^{2-\gamma} \left\langle A_{\mathcal{S}'}(|f|^r)v \right\rangle_{P_k}^{\gamma-1} \\
&\leq \int_{P_{k-1}} A_{\mathcal{S}'}(|f|^r)^{2-\gamma} \cdot A_{\mathcal{S}'}\left(A_{\mathcal{S}'}(|f|^r)v \right)^{\gamma-1}v =: \int_{P_{k-1}} h.
\end{align*}
Iterating Minkowski's inequality $(k-1)$ times yields
\begin{align*}
\sum_{P_{k-1} \subseteq \cdots \subseteq Q} \prod_{i=1}^{k-1} \frac{v(P_i)}{\mu(P_i)} \int_{P_{k-1}} h \leq \int_Q A_{\mathcal{S}',v}^{k-1}(h).
\end{align*}
Combining all estimates gives
\begin{align}
\frac{1}{\mu(Q)} \int_Q \left(\sum_{P \subseteq Q} \frac{v(P)}{\mu(P)} \chi_P \right)^{rm} |f|^r \lesssim \lla A_{\mathcal{S}',v}^{k-1}(h) \rra_Q.
\end{align}
\end{proof}
We also need the classic sparse boundedness.

\begin{lemma}[\cite{Moen2014}]
Let $1<p<\infty$,  $\omega \in A_p$. For any sparse family of cubes $\mathcal{S} \subseteq \mathcal{D}$ and $f \in L^p\left(\omega\right)$, we have
\begin{align}\label{sparse.bound.}
   \left\|\sum_{Q \in \mathcal{S}}\langle | f| \rangle_{1, Q} \chi_Q\right\|_{L^p(\omega)} \lesssim[\omega]_{A_p}^{\max \left(1, \frac{1}{p-1}\right)}\|f\|_{L^p(\omega)} .
\end{align}
\end{lemma}

\begin{proof}[Proof of theorem \ref{quan.main}]
We begin our proof with an important observation: the quantitative estimation of $\TB$ is equivalent to that of
\begin{align*}
   &\quad \TB \\
   &= \sum_{Q \in \S} \mu(Q)^{\mici + 1}  \prod_{i_1 \in \tau'}\lla\left|b_{i_1}- b_{i_1,Q}\right|^{k_{i_1}} f_{i_1}\rra_{r_{i_1},Q}\lla\prod_{i_2 \in \tau \backslash \tau'} \left|b_{i_2}- b_{i_2,Q}\right|^{k_{i_2}}g\rra_{s',Q} \prod_{j \in (\tau')^c} \langle\left|f_j\right|\rangle_{r_j,Q}.
   \end{align*}
For simple, we denote that 
\begin{align*}
   \mathcal{F}_1 = \lla\left|b_{i_1}- b_{i_1,Q}\right|^{k_{i_1}} f_{i_1}\rra_{r_{i_1},Q}, \quad 
   \mathcal{F}_2 = \lla\prod_{i_2 \in \tau \backslash \tau'} \left|b_{i_2}- b_{i_2,Q}\right|^{k_{i_2}}g\rra_{s',Q}.
\end{align*}
The proof is transformed into the estimation of $\mathcal{F}_1$ and $\mathcal{F}_2$.

Firstly, we conduct an estimation for $\mathcal{F}_1$.  
   Let us begin with the estimate to $\F_1$, by Lemma \ref{zhang:6.1}, 
,since $b_{i_1} \in {\rm BMO}_{\nu_{i_1}}$, where $\nu_{i_1}$ is a weight to be chosen later, we can derive the following
   \[
   |b_{i_1}(x)-b_{i_1,Q}|\leq C\|b_{i_1}\|_{{\rm BMO}_{\nu_{i_1}}(X)}\sum_{P\in {\mathcal{S}'},\ P\subseteq Q}\nu_{i_1,P}\chi_{P}(x).
   \]
   Then, there exists a sparse collection of cubes $\mathcal{S} \subseteq \mathcal{S}^{\prime} \subseteq \d$ such that for any $Q \in \mathcal{S}$,
   \begin{align*}
      \lla\left|b_{i_1}- b_{i_1,Q}\right|^{k_{i_1}} f_{i_1}\rra^{r_{i_1}}_{r_{i_1},Q} & \lesssim \|b_{i_1}\|^{k_{i_1}r_{i_1}}_{{\rm BMO}_{\nu_{i_1}}(X)}\frac{1}{\mu(Q)} \int_Q \left(\sum_{P \in \mathcal{S}^{\prime}: P \subseteq Q} \nu_{i_1,P}\chi_{P}\right)^{k_{i_1}r_{i_1}}|f_{i_1}|^{r_{i_1}} \\
      & = \|b_{i_1}\|^{k_{i_1}r_{i_1}}_{{\rm BMO}_{\nu_{i_1}}(X)}\frac{1}{\mu(Q)} \int_Q \left(\sum_{P \in \mathcal{S}^{\prime}: P \subseteq Q}\frac{\nu_{i_1}(P)}{\mu(P)}\chi_{P}\right)^{k_{i_1}r_{i_1}}|f_{i_1}|^{r_{i_1}} .\\
   \end{align*}
Using Proposition \ref{premainproof} we obtain
\begin{align}\label{est.F_1}
   \lla| b_{i_1}- b_{i_1,Q}|^{k_{i_1}} f_{i_1}\rra_{r_{i_1},Q} \lesssim \|b_{i_1}\|^{k_{i_1}}_{{\rm BMO}_{\nu_{i_1}}(X)}\lla\mathcal{A}_{\mathcal{S}^{\prime}, \nu_{i_1}}^{{\tt a } - 1}(h)^{\frac{1}{r_{i_1}}}\rra_{r_{i_1},Q},
\end{align}
with
\begin{align*}
   {h} & = {A}_{\mathcal{S}^{\prime}}\left(|f_{i_1}|^{r_{i_1}}\right)^{2-\gamma_{i_1}} \cdot {A}_{\mathcal{S}^{\prime}}\left({A}_{\mathcal{S}^{\prime}}\left(|f_{i_1}|^{r_{i_1}}\right) \nu_{i_1}^{}\right)^{\gamma_{i_1}-1} \cdot \nu_{i_1}^{},\\
\end{align*}
where ${\tt a}:=\lfloor k_{i_1}r_{i_1}\rfloor$, $\gamma_{i_1}:={k_{i_1}r_{i_1}}-({\tt a}-1) \in[1,2)$ with $i_1 \in \tau'$.

      Next, we focus our attention on $\mathcal{F}_2$. Similar to \eqref{est.F_1}, for $\mathcal{F}_2$, we have
\begin{align*}
   \lla \prod_{i_2 \in \tau \backslash \tau'}| b_{i_2}- b_{i_2,Q}|^{k_{i_2}} g\rra_{s',Q}^{s'} & \lesssim \prod_{i_2 \in \tau \backslash \tau'}\|b_{i_2}\|^{t_{i_2}s'}_{{\rm BMO}_{\nu_0}(X)}\frac{1}{\mu(Q)} \int_Q \left(\sum_{P \in \mathcal{S}^{\prime}: P \subseteq Q} \nu_{0,P}\chi_{P}\right)^{s'\cdot\sum\limits_{i_2}k_{i_2}}|g|^{s'} \\
   & \lesssim \prod_{i_2 \in \tau \backslash \tau'}\|b_{i_2}\|^{t_{i_2}s'}_{{\rm BMO}_{\nu_0}(X)}\frac{1}{\mu(Q)} \int_Q \left(\sum_{P \in \mathcal{S}^{\prime}: P \subseteq Q} \frac{\nu_0(Q)}{\mu(Q)}\chi_{P}\right)^{s'\cdot\sum\limits_{i_2}k_{i_2}}|g|^{s'}, \\ 
\end{align*}
with $\nu_{0}=\max\limits_{i_2 \in \tau \backslash \tau'}\nu_{i_2}$. By Proposition \ref{premainproof}, we demonstrate that
\begin{align}\label{F2:Max.weight}
   \lla \prod_{i_2 \in \tau \backslash \tau'}| b_{i_2}- b_{i_2,Q}|^{k_{i_2}} g\rra_{s',Q} \lesssim \prod_{i_2 \in \tau \backslash \tau'}\|b_{i_2}\|^{k_{i_2}}_{{\rm BMO}_{\nu_0}(X)} \lla A_{\mathcal{S}^{\prime}, \nu_0}^{{\tt L} - 1}(h_1)^{\frac{1}{s'}}\rra_{s',Q},
\end{align}
with $$h_1 =  A_{\mathcal{S}^{\prime}}\left(|f|^{s'}\right)^{2-\gamma_1} \cdot  A_{\mathcal{S}^{\prime}}\left( A_{\mathcal{S}^{\prime}}\left(|f|^{s'}\right) \nu_0^{}\right)^{\gamma_1-1} \cdot \nu_0^{},$$ 
where ${\tt L}:=\lfloor {s'\cdot\sum\limits_{i_2}k_{i_2}}\rfloor$, $\gamma_1:={s'\cdot\sum\limits_{i_2}k_{i_2}}-({\tt L}-1) \in[1,2).$

We can definitely adopt another approach to estimate $\mathcal{F}_2$.
By applying Hölder's inequality, we conclude
\begin{align*}
   \lla\prod_{i_2 \in \tau \backslash \tau'} \left|b_{i_2}- b_{i_2,Q}\right|^{k_{i_2}}g\rra_{s',Q} & \leq \left\langle  \prod_{\substack{i_2 \in \tau \backslash \tau'\\ i_2 \neq i_0}} \left|b_{i_2}- b_{i_2,Q}\right|^{k_{i_2}}\right\rangle_{2s',Q}\lla \left|b_{i_0}- b_{i_0,Q}\right|^{k_{i_0}}g \rra_{2s',Q}\\
   &\leq \prod_{\substack{i_2 \in \tau \backslash \tau'\\ i_2 \neq i_0}} \left\langle \left|b_{i_2}- b_{i_2,Q}\right|\right\rangle^{k_{i_2}}_{2(|(\tau')^c| - 1)k_{i_2}s',Q} \lla \left|b_{i_0}- b_{i_0,Q}\right|^{k_{i_0}}g \rra_{2s',Q}\\
   &\overset{\eqref{BMOeq._}}{\lesssim}\prod_{\substack{i_2 \in \tau \backslash \tau'\\ i_2 \neq i_0}}\|b_{i_2}\|_{\rm{BMO}}^{k_{i_2}}\lla \left|b_{i_0}- b_{i_0,Q}\right|^{k_{i_0}}g \rra_{2s',Q},
\end{align*}
where $i_0 \in \tau \backslash \tau'$.

From the above processing, we observe that further estimation essentially involves estimating $\lla \left|b_{i_0}- b_{i_0,Q}\right|^{k_{i_0}}g \rra_{2s',Q}$. 
Similar to \eqref{F2:Max.weight}, it is immediately follows from Proposition \ref{premainproof} that 
\begin{align*}
   \laa| b_{i_0}- b_{{i_0},Q}|^{k_{i_0}} g\raa_{2s',Q}^{2s'} \lesssim \|b_i\|^{2k_{i_0}s'}_{{\rm BMO}_{\nu_{i_0}}(X)}\laa{A}_{\mathcal{S}^{\prime}, \nu_{i_0}}^{{\tt c } - 1}(h_2)\raa_{Q}.
\end{align*}
Then, we obtain
\begin{align}\label{F2:Holder}
   \lla\prod_{i_2 \in \tau \backslash \tau'} \left|b_{i_2}- b_{i_2,Q}\right|^{k_{i_2}}g\rra_{s',Q} \overset{}{\lesssim}\prod_{\substack{i_2 \in \tau \backslash \tau'\\ i_2 \neq i_0}}\|b_{i_2}\|_{\rm{BMO}(X)}^{k_{i_2}}\|b_i\|^{k_{i_0}}_{{\rm BMO}_{\nu_{i_0}}(X)}\laa{A}_{\mathcal{S}^{\prime}, \nu_{i_0}}^{{\tt c } - 1}(h_2)^{\frac{1}{2s'}}\raa_{2s',Q}.
\end{align}
with
\begin{align*}
   h_2& = {A}_{\mathcal{S}^{\prime}}\left(|g|^{2s'}\right)^{2-\gamma_2} \cdot {A}_{\mathcal{S}^{\prime}}\left({A}_{\mathcal{S}^{\prime}}\left(|g|^{2s'}\right) \nu_{i_0}^{}\right)^{\gamma_2-1} \cdot \nu_{i_0}^{},
\end{align*}
where  ${\tt {\tilde L}}:=\lfloor 2k_{i_0}s'\rfloor$, $\gamma_2:= 2k_{i_0}s'-({\tt {\tilde L}}-1) \in[1,2).$

By using \eqref{est.F_1} with \eqref{F2:Max.weight} and \eqref{F2:Holder} respectively, we can obtain the following different estimates.
\begin{align}\notag
   \TB 
   &\lesssim \mathcal{C}_{0,1} \left[\int_{X}\sum_{Q \in \S} \mu(Q)^{\mici}\left(\prod_{i \in \tau'}\lla {A}_{\mathcal{S}^{\prime}, \nu_{i}}^{{\tt a } - 1}(h)^{\frac{1}{r_{i}}}\rra_{r_{i},Q}\right.\right.\\ \notag
   &\quad\left.\left.\lla {A}_{\mathcal{S}^{\prime}, \nu_{i_0}}^{{\tt b } - 1}(h_1)^{\frac{1}{s'}}\rra_{s',Q}  \prod_{j \in (\tau')^c }\left\langle f_j\right\rangle_{r_j,Q}\right)\chi_Q\right]\\ \label{eq:b}
   &= \mathcal{C}_{0,1} \left|\Lambda_{\mici,\boldsymbol{r}_{\tt L}, \S'}( {\vec {\bf A}_{\tt L}})\right|;
\end{align}
\begin{align}\notag
   \TB 
   &\lesssim  \mathcal{C}_{0,2} \left[\int_{X}\sum_{Q \in \S} \mu(Q)^{\mici}\left(\prod_{i \in \tau'}\lla {A}_{\mathcal{S}^{\prime}, \nu_{i}}^{{\tt a } - 1}(h)^{\frac{1}{r_{i}}}\rra_{r_{i},Q}\right.\right.\\ \notag
   &\quad\left.\left.\lla {A}_{\mathcal{S}^{\prime}, \nu_{0}}^{{\tt c } - 1}({h}_2)^{\frac{1}{2s'}}\rra_{2s',Q}  \prod_{j \in (\tau')^c}\left\langle f_j\right\rangle_{r_j,Q}\right)\chi_Q\right]\\ \label{eq:c}
   &= \mathcal{C}_{0,2}\left|\Lambda_{\mici,\boldsymbol{r}_{\tt {\tilde L}}, \S'}( {\vec {\bf A}_{\tt {\tilde L}}})\right|,
\end{align}
where 
\begin{align*}
\mathcal{C}_{0,1} = &\prod_{i \in \tau'}\|b_{i}\|^{k_{i}}_{{\rm{BMO}}_{\nu_{i}}} \prod_{j \in \tau \backslash \tau'}\|b_{j}\|^{k_{j}}_{{\rm BMO}_{\nu_0}},\\
\mathcal{C}_{0,2} = &\|b_{i_0}\|^{k_{i_0}}_{{\rm{BMO}}_{\nu_{i_0}}}\prod_{i \in \tau'}\|b_{i}\|^{k_{i}}_{{\rm{BMO}}_{\nu_{i}}}\prod_{\substack{j \in \tau \backslash \tau'\\ j \neq i_0}}\|b_{j}\|_{\rm{BMO}}^{k_{j}};
\end{align*}

\begin{align*}
   \boldsymbol{r}_{\tt L}&=\Big((r_{i})_{i \in \tau_m},r_{m + 1}\Big),\,\text{with}\,\, r_{m+1}=s',\\
   \boldsymbol{r}_{\tt {\tilde L}}&=\Big((r_{i})_{i \in \tau_m},r_{m + 1}\Big),\,\text{with}\,\, r_{m+1}=2s';\\
\end{align*}
\begin{align*}
   {\vec {\bf A}_{\tt L}}= \left(\left({A}_{\mathcal{S}^{\prime}, \nu_{i}}^{{\tt a } - 1}(h)^{\frac{1}{r_i}}\right)_{i \in \tau'},{A}_{\mathcal{S}^{\prime}, \nu_{0}}^{{\tt b } - 1}({h}_1)^{\frac{1}{s'}},\left(f_j\right)_{j \in (\tau')^c}\right),
\end{align*}
\begin{align*}
   {\vec {\bf A}_{\tt {\tilde L}}}= \left(\left({A}_{\mathcal{S}^{\prime}, \nu_{i}}^{{\tt a } - 1}(h)^{\frac{1}{r_i}}\right)_{i \in \tau'},{A}_{\mathcal{S}^{\prime}, \nu_{i_0}}^{{\tt c } - 1}({h}_2)^{\frac{1}{2s'}},\left(f_j\right)_{j \in (\tau')^c}\right);
\end{align*}
and
\begin{align*}
   C_{{\tt L},\rm{BMO}} = \prod_{i \in \tau'}\|b_{i}\|^{k_{i}}_{{\rm{BMO}}_{\nu_{i}}} \prod_{i_2 \in \tau \backslash \tau'}\|b_{i_2}\|^{k_{i_2}}_{{\rm BMO}_{\nu_0}(X)};
\end{align*}
\begin{align*}
   C_{{\tt  c},\rm{BMO}}=\|b_{i_0}\|^{k_{i_0}}_{{\rm{BMO}}_{\nu_{i_0}}}\prod_{i \in \tau'}\|b_{i}\|^{k_{i}}_{{\rm{BMO}}_{\nu_{i}}}\prod_{\substack{i_2 \in (\tau')^c\\ i_2 \neq i_0}}\|b_{i_2}\|_{\rm{BMO}}^{k_{i_2}}.
\end{align*}

The estimate of $\TB$ actually reduces to the estimation of $\left|\Lambda_{\mici,\boldsymbol{r}_{\tt L}, \S'}( {\vec {\bf A}_{\tt L}})\right|$ and $\left|\Lambda_{\mici,\boldsymbol{r}_{\tt {\tilde L}}, \S'}( {\vec {\bf A}_{\tt {\tilde L}}})\right|$. Using \eqref{Nier2.8} with $ \sum\limits_{j =1}^{m+1} \frac{1}{p_j}=1 + \mici$ and $\prod\limits_{j =1}^{m+1} u_j=1$,
\begin{align}\notag
   &\quad \left|\Lambda_{\mici,\boldsymbol{r}_{\tt L}, \S'}( {\vec {\bf A}_{\tt L}})\right| \\ \label{Max_1}
   &\overset{}{\lesssim} [\vec{u}]^{\varTheta}_{(\vec{p},q),(\vec{r}, s)} \prod_{i \in \tau'}\left\|{A}_{\mathcal{S}^{\prime}, \nu_{i}}^{{\tt a } - 1}(h)^{\frac{1}{r_i}}\right\|_{L^{p_{i}}(u_{i}^{{p_{i}}})}\prod_{j \in (\tau')^c}\left\|f_j\right\|_{L^{p_j}(u_j^{{p_j}})}\left\|{A}_{\mathcal{S}^{\prime}, \nu_{0}}^{{\tt b } - 1}(h_1)^{\frac{1}{t}}\right\|_{L^{p_{m+1}}(u_{m+1}^{{p_{m+1}}})},\\ \notag
   \intertext{due to $p_{m+1}=q'$ and $u_{m+1} = \left(\prod\limits_{i=1}^{m}u_j\right)^{-1}=u^{-1}$, the above}\\ \notag
   &= [\vec{u}]^{\varTheta}_{(\vec{p},q),(\vec{r}, s)} \prod_{i \in \tau'}\left\|{A}_{\mathcal{S}^{\prime}, \nu_{i}}^{{\tt a } - 1}(h)^{\frac{1}{r_i}}\right\|_{L^{p_{i}}(u_{i}^{{p_{i}}})}\prod_{j \in (\tau')^c}\left\|f_j\right\|_{L^{p_j}(u_j^{{p_j}})}\left\|{A}_{\mathcal{S}^{\prime}, \nu_{0}}^{{\tt b } - 1}(h_1)^{\frac{1}{t}}\right\|_{L^{q'}(u^{-{q'}})}.
\end{align}

By the similar step in \cite[Sec. 4.7]{CenSong2412}, we get
\begin{align*}
   & \quad \left\|A_{\mathcal{S}^{\prime}, \nu_{i}}^{{\tt a } - 1}(h)\right\|_{L^{\frac{p_i}{r_i}}(u_{i}^{{p_{i}}})}  =  \left\|A_{\mathcal{S}^{\prime}}\left(A_{\mathcal{S}^{\prime}, \nu_{i}}^{{\tt a } - 2}(h)\right)\right\|_{L^{\frac{p_i}{r_i}}(u_{i}^{{p_{i}}} \nu_{i}^{\frac{p_i}{r_i}})}\\
   &\lesssim \left(\left[u_{i}^{{p_{i}}} \nu_{i}^{\frac{p_i}{r_i}}\right]_{A_{\frac{p_i}{r_i}}}\left[u_{i}^{{p_{i}}} \nu_{i}^{2 \frac{p_i}{r_i}}\right]_{A_{\frac{p_i}{r_i}}} \cdots\left[u_{i}^{{p_{i}}} \nu_{i}^{({\tt a} - 1)\frac{p_i}{r_i}}\right]_{A_{\frac{p_i}{r_i}}}\right)^{\max \left\{1, \frac{1}{\frac{p_i}{r_i}-1}\right\}}\|h\|_{L^{\frac{p_i}{r_i}}\left(u_{i}^{{p_{i}}} \nu_{i}^{({\tt a} - 1)\frac{p_i}{r_i}}\right)}
\end{align*}
 Setting $u_{i}^{{p_{i}}} \nu_{i}^{({\tt a} - 1)\frac{p_{i}}{r_i}} = \mu_{i}$ that is $ \nu_{i} = \left(\mu_{i}/u_{i}^{{p_{i}}}\right)^{\frac{r_i}{({\tt a} - 1)p_i}}$, we obtain
\begin{align}\label{mofang}
   \left\|A_{\mathcal{S}^{\prime}, \nu_{i}}^{{\tt a } - 1}(h)\right\|_{L^{\frac{p_i}{r_i}}(u_{i}^{{p_{i}}})} & \lesssim \left(\left[u_{i}^{{p_{i}}}\right]^{\frac{{\tt a} - 2}{2}}_{A_{\frac{p_i}{r_i}}} \left[\mu_{i}\right]^{\frac{{\tt a}}{2}}_{A_{\frac{p_i}{r_i}}} \right)^{\max \left\{1, \frac{1}{\frac{p_i}{r_i}-1}\right\}}\|h\|_{L^{\frac{p_i}{r_i}}\left(\mu_{i}\right)}.
\end{align}

Finally, we estimate $\|h\|_{L^{\frac{p_{i}}{r_i}}\left(p_{i}\right)}$, since 
\begin{align*}
   \frac{1}{p_{i}} = (2 - \gamma_i)\frac{1}{p_{i}} + (\gamma_i - 1)\frac{1}{p_i},
\end{align*}
by setting
\begin{align*}
   \mu^{{\frac{r_i}{p_i}}}_{i} \nu_{i}= \left(\omega^{{2-\gamma_i}}_{i} \cdot \tilde{\omega}^{\frac{\gamma_i -1}{p_i}}_{i}\right)^{{r_i}}
\end{align*}
we have by Hölder's inequality,
\begin{align*}
   &\|h\|_{L^{\frac{p_i}{r_i}}\left(\mu_{i}\right)}\\
   &=\left\|{A}_{\mathcal{S}^{\prime}}\left(|f_{i}|^{r_i}\right)^{2-\gamma_i} \cdot {A}_{\mathcal{S}^{\prime}}\left({A}_{\mathcal{S}^{\prime}}\left(|f_{i}|^{r_i}\right) \nu_{i}^{}\right)^{\gamma_i-1} \cdot \nu_{i}^{}\right\|_{L^{\frac{p_i}{r_i}}\left(\mu_{i}\right)}\\
   &\leq \left\|{A}_{\mathcal{S}^{\prime}}\left(|f_{i}|^{r_i}\right)^{2-\gamma_i}\right\|_{L^{\frac{p_{i}}{2 - \gamma_i}\cdot \frac{1}{r_i}}(\omega_{i}^{p_i})} \times \left\|{A}_{\mathcal{S}^{\prime}}\left({A}_{\mathcal{S}^{\prime}}\left(|f_{i}|^{r_i}\right) \nu_{i}^{}\right)^{\gamma_i-1} \right\|_{L^{\frac{p_i}{\gamma_i - 1}\cdot \frac{1}{r_i}}(\tilde{\omega}_{i})}\\
   &=:I \times II.\\
\end{align*}

For $I$, it follows immediately from  \eqref{sparse.bound.} that
\begin{align*}
   I \lesssim [\omega_{i}]_{A_{\frac{p_{i}}{r_i}}}^{(2 - \gamma_i)\max(1,\frac{r_i}{p_{i} - r_i})}\|f_{i}\|_{L^{{p_{i}}}(\omega_{i}^{p_i})}^{(2 - \gamma_i)r_i}.
\end{align*}
Analogously, 
\begin{align*}
   II &\lesssim [\tilde{\omega}_{i}]_{A_{\frac{p_i}{r_i}}}^{(2 - \gamma_i)\max(1,\frac{r_i}{p_i - r_i})}\left\|{A}_{\mathcal{S}^{\prime}}\left(|f_{i}|^{r_i}\right) \nu_{i}^{}\right\|^{\gamma_i-1}_{L^{\frac{p_{i}}{r_i}}(\tilde{\omega}_{i})}\\
   \intertext{Setting $\nu^{\frac{p_{i}}{r_i}}_{i} \cdot \tilde{\omega}_{i} = \omega_{i}^{p_i}$}
   &\lesssim [\tilde{\omega}_{i}]_{A_{\frac{p_i}{r_i}}}^{(2 - \gamma_i)\max(1,\frac{r_i}{p_i - r_i})}[\omega_{i}]^{( \gamma_i - 1)\max(1,\frac{r_i}{p_{i}- r_i})}_{A_{\frac{p_{i}}{r_i}}} \|f_{i}\|_{L^{{p_{i}}}(\omega_{i}^{p_i})}^{(\gamma_i -  1)r_i}.
\end{align*}
Collecting our estimates, we have shown for $i \in \tau'$
\begin{align}\label{Max_2}
   \left\|A_{\mathcal{S}^{\prime}, \nu_{i}}^{{\tt a } - 1}(h)\right\|_{L^{\frac{p_i}{r_i}}(u_{i}^{{p_{i}}})}^{\frac{1}{r_i}} \lesssim C(\omega_{i}, \nu_i)^{\frac{1}{r_i}} \|f_{i}\|_{L^{{p_{i}}}(\omega_{i}^{p_i})},
\end{align}
where
\begin{align*}
C(\omega_{i},\nu_i)&=\left(\left[u_{i}^{{p_{i}}}\right]^{\frac{{\tt a} - 2}{2}}_{A_{\frac{p_i}{r_i}}} \left[\mu_{i}\right]^{\frac{{\tt a}}{2}}_{A_{\frac{p_i}{r_i}}} \right)^{\max \left\{1, \frac{1}{\frac{p_i}{r_i}-1}\right\}}\\
   &\quad \times [\omega_{i}]_{A_{\frac{p_{i}}{r_i}}}^{\max(1,\frac{r_i}{p_{i} - r_i})}[\tilde{\omega}_{i}]_{A_{\frac{p_{i}}{r_i}}}^{(2 - \gamma_i)\max(1,\frac{r_i}{p_{i} - r_i})},
\end{align*}
with
\begin{align*}
u_i &= \omega_i \cdot \nu_i^{-\frac{{\tt a} + \gamma_i - 1}{r_i}}, \\
\mu_i &= \omega_i^{p_i} \cdot \nu_i^{-\frac{\gamma_i p_i}{r_i}}, \\
\tilde{\omega}_i &= \omega_i^{p_i} \cdot \nu_i^{-\frac{p_i}{r_i}}.
\end{align*}

Finally, we focus on estimating $\left\| A_{\mathcal{S}^{\prime}, \nu_{0}}^{{\tt L} - 1}({h_1})^{\frac{1}{s'}} \right\|_{L^{q'}(u^{-{q'}})}$. Using a similar approach as in \eqref{mofang}, we get
\begin{align}\label{Max_3}
   \left\|{A}_{\mathcal{S}^{\prime}, \nu_{0}}^{{\tt b } - 1}({h_1})\right\|^{\frac{1}{s'}}_{L^{\frac{q'}{s'}}(u^{-{q'}})} \lesssim C(\omega,\nu_0)^{\frac{1}{s'}} \|g\|_{L^{q'}(\omega^{-q'})}.
\end{align}
where
\begin{align*}
   C_{\tt L}(\omega,\nu_0)&:=\left(\left[v_1^{-{q'}}\right]^{\frac{{\tt L} - 2}{2}}_{A_{-\frac{q'}{s'}}} \left[\mu\right]^{\frac{{\tt L}}{2}}_{A_{-\frac{q'}{s'}}} \right)^{\max \left\{1, \frac{1}{-\frac{q'}{s'}-1}\right\}}\\
   &\quad \times [\omega]_{A_{-\frac{q'}{s'}}}^{\max(1,\frac{s'}{-q' - s'})}[\tilde{\omega}]_{A_{-\frac{q'}{s'}}}^{(2 - \gamma_1)\max(1,\frac{s'}{-q' - s'})}\\
\end{align*}
with
\begin{align*}
v_1 &= \omega^{2\gamma_1 -3} \cdot \nu_0^{\frac{\gamma_1 + {\tt L} -1}{s'}}, \\
\mu_1 &= \omega^{q'(3-2\gamma_1)} \cdot \nu_0^{-\frac{\gamma_1 q'}{s'}}, \\
\tilde{\omega}' &= \omega^{-q'} \cdot \nu_0^{-\frac{q'}{s'}}.
\end{align*}
Combing \eqref{Max_1}, \eqref{Max_2} and \eqref{Max_3}, we have
\begin{align}\label{bbb}
   \TB \lesssim C_{{\tt L}} \prod_{i=1}^{m} \|f_{i}\|_{L^{p_i}(\omega_{i}^{p_i})} \|g\|_{L^{q'}(\omega^{-q'})}
\end{align}
with 
\begin{align}\label{b_M}
   C_{{\tt L}} = &C_{{\tt L},\rm{BMO}} \times [\vec{u}]^{\varTheta}_{(\vec{p},q),(\vec{r}, s)} \times \prod_{i \in \tau'}^{} C(\omega_{i},\nu_i)^{\frac{1}{r_i}} \times  C_{\tt L}(\omega,\nu_0)^{\frac{1}{s'}}.
\end{align}

Following the previous steps, we can obtain the following similar estimates for $\left|\Lambda_{\mici,\boldsymbol{r}_{\tt {\tilde L}}, \S'}( {\vec {\bf A}_{\tt {\tilde L}}})\right|$.
Then we demonstrate that 
\begin{align*}
   \TB &\lesssim C_{{\tt {\tilde L}},\rm{BMO}} \left|\Lambda_{\mici,\boldsymbol{r}_{\tt {\tilde L}}, \S'}( {\vec {\bf A}_{\tt {\tilde L}}})\right|\\
   &\lesssim C_{{\tt {\tilde L}},_{}} \prod_{i=1}^{m} \|f_{i}\|_{L^{p_i}(\omega_{i}^{p_i})}\|g\|_{L^{q'}(\omega^{-q'})},
\end{align*}
where 
\begin{align}\label{c_M}
   C_{{\tt {\tilde L}}} = &C_{{\tt {\tilde L}},\rm{BMO}} \times [\vec{u}]^{\varTheta}_{(\vec{p},q),(\vec{r}, s)} \times \prod_{i \in \tau'}^{} C(\omega_{i},\nu_i)^{\frac{1}{r_i}} \times  C_{\tt {\tilde L}}(\omega,\nu_0)^{\frac{1}{2s'}}.
\end{align}
and
\begin{align*}
   C_{\tt {\tilde L}}(\omega,\nu_0)&:=\left(\left[v_2^{-{q'}}\right]^{\frac{{\tt {\tilde L}} - 2}{2}}_{A_{-\frac{q'}{2s'}}} \left[\mu\right]^{\frac{{\tt {\tilde L}}}{2}}_{A_{-\frac{q'}{2s'}}} \right)^{\max \left\{1, \frac{1}{-\frac{q'}{2s'}-1}\right\}}\\
   &\quad \times [\omega]_{A_{-\frac{q'}{2s'}}}^{\max(1,\frac{2s'}{-q' - s'})}[\tilde{\omega}]_{A_{-\frac{q'}{2s'}}}^{(2 - \gamma_2)\max(1,\frac{2s'}{-q' - s'})}\\
\end{align*}
with
\begin{align*}
v_2&= \omega^{2\gamma_2 -3} \cdot \nu_0^{\frac{\gamma_2 + {\tt {\tilde L}} -1}{2s'}}, \\
\mu_2&= \omega^{q'(3-2\gamma_2)} \cdot \nu_0^{-\frac{\gamma_2 q'}{2s'}}, \\
\tilde{\omega}''&= \omega^{-q'} \cdot \nu_0^{-\frac{q'}{2s'}}.
\end{align*}
\end{proof}

\section{\bf Sharp weighted estimates for higher-order multi-symbol ($m+1$)-linear fractional sparse form}\label{sharp}

In this section, we mainly consider sparse sharp weighted estimation, which means that we need to extend the work of Moen et al. \cite[Theorem 3.2]{Moen2014} and Pérez et al. \cite[Theorem 1.2]{Perez2014} to multilinear fractional sparse forms.

\begin{theorem}\label{thm:sparse-dom_1}
Let $m \in \N$, $\eta \in [0,m)$, $1 \leq r_i ,s' < \infty$ for every $i \in \tau_m$, $1<p_1,\dots,p_m<\infty$, and $\frac{1}{{{q}}}:=\sum\limits_{i = 1}^m {\frac{1}{{{p_i}}}}- \eta \in (0,1)$. Suppose that $\mathbf{t}$ and $\mathbf{k}$ are both multi-indexs with \(\mathbf{t} \leq \mathbf{k}\), and multi-symbols \(\mathbf{b} = (b_1, \ldots, b_{m}) \in (\BMO(X))^m\). If $(\vec{r}, s) \prec (\vec{p},q)$ and $\vec{\omega} \in A_{(\vec{p},q),(\vec{r}, s)}$, then 
	\begin{align}\label{eq:main-est}
\sup\limits_{\mathcal{S} \subseteq \d} \left\|\mathcal{A}_{\eta, \mathcal{S}, \tau, \vec{r}, s'}^{\mathbf{b}, \mathbf{k}, \mathbf{t}}\right\|_{\prod\limits_{j=1}^{m} L^{p_j}\left(\omega_j^{p_j}\right) \times L^{q'}\left(\omega^{-q'}\right) \rightarrow \mathbb{R}}
	&\lesssim_{{\bf k}, {\bf t}, \vec{p},\vec{r},\eta,X} \prod_{i \in \tau} \|b_i\|_{\mathrm{BMO}}^{k_i} \cdot [\vec{\omega}]^{\varTheta}_{(\vec{p},q),(\vec{r}, s)}.
\end{align}
where $\varTheta:=\max \left\{ {\frac{{{p_1}}}{{{p_1} - {r_1}}}, \cdots ,\frac{{{p_m}}}{{{p_m} - {r_m}}},\frac{{q'}}{{q' - s'}}} \right\}$, and this bound is sharp, i.e., $k_i$ and $\varTheta$ cannot be reduced.
\end{theorem}

\begin{proof}

We begin our proof by cliam the following two facts
 \begin{align}\label{xin_1_1}
\Bigl\langle |f_i||b_i - b_{i,Q}|^{t_i} \Bigr\rangle_{r_i,Q}
&\lesssim_{t_i,X} \|b_i\|_{\BMO}^{t_i} \bigl\langle |f_i| \bigr\rangle_{r_{i},Q} ,\\  \label{xin_1_2}
\Bigl\langle \prod_{i\in\tau} |b_i - b_{i,Q}|^{k_i-t_i} |g|\Bigr\rangle_{s',Q} &\lesssim_{k_i,t_i,X} \prod_{i\in\tau}\|b_{i}\|_{\BMO}^{k_i - t_{i}} \langle |g|\rangle_{s',Q}.
\end{align}
In fact, applying the Lemma \ref{xin_1} with $w=|f_i|^{r_i}$, we obtain the desired \eqref{xin_1_1}.

Analogous, fix $i_1, i_2 \in \tau$ with $i_1 \neq i_2$, and the other estimate follows from the fact that
\begin{align*}
    \Bigl\langle \prod_{i\in\tau} |b_i - b_{i,Q}|^{k_i-t_i} |g|\Bigr\rangle_{s',Q} &\lesssim \|b_{i_1}\|_{\BMO}^{t_{i_1}} \Bigl\langle \prod_{{\substack{i \neq i_1\\ i \in \tau }}} |b_i - b_{i,Q}|^{k_i-t_i} |g|\Bigr\rangle_{s',Q} \\
   & \lesssim \|b_{i_1}\|_{\BMO}^{t_{i_1}} \|b_{i_2}\|_{\BMO}^{t_{i_2}} \Bigl\langle \prod_{{\substack{i \neq i_1,i_2\\ i \in \tau }}} |b_i - b_{i,Q}|^{k_i-t_i} |g|\Bigr\rangle_{s',Q}\\
   & \cdots\\
&\lesssim \prod_{i\in\tau}\|b_{i}\|_{\BMO}^{t_{i}} \langle |g|\rangle_{s',Q}.
\end{align*}

Then we proceed the proof by \eqref{xin_1_1} and \eqref{xin_1_2}
\begin{align*}
\bigl|\mathcal{A}_{\eta, \mathcal{S}, \tau, \vec{r}, s'}^{\mathbf{b}, \mathbf{k},\mathbf{t}}(\vec{f},g)\bigr| 
&\leq \sum_{Q \in \mathcal{S}} \mu(Q)^{\eta+1} 
\prod_{i \in \tau} \Bigl\langle |f_i||b_i - b_{i,Q}|^{t_i} \Bigr\rangle_{r_i,Q} \\
&\quad \times \Bigl\langle \prod_{i \in \tau} |b_i - b_{i,Q}|^{(k_i-t_i)} |g| \Bigr\rangle_{s',Q} 
\prod_{j \in \tau^c} \bigl\langle |f_j| \bigr\rangle_{r_j,Q}\\
& \lesssim \prod_{i \in \tau}\left\|b_i\right\|_{\mathrm{BMO}}^{k_i} \sum_{Q \in \mathcal{S}} \mu(Q)^{\eta+1} \left\langle | g | \right\rangle_{s', Q} \prod_{i = 1}^m \left\langle | f_i | \right\rangle_{r_{i}, Q}\\
&= \prod_{i \in \tau}\left\|b_i\right\|_{\mathrm{BMO}}^{k_i} \sum_{Q \in \mathcal{S}} \mu(Q)^{\eta+1} \left\langle  gv^{-\frac{1}{s'}}  \right\rangle_{s', Q}^{v} \left\langle  v  \right\rangle_{1, Q}^{\frac{1}{s'}} \prod_{i =1}^m \left\langle  f_i v_i^{-\frac{1}{r_{i}}} \right\rangle_{r_{i}, Q}^{v_i} \prod_{i = 1}^m\left\langle  v_i  \right\rangle_{1, Q}^{\frac{1}{r_{i}}}\\
\intertext{To apply the $A_{(\vec{p},q),(\vec{r}, s)}$ condition, we set $v_j:=\omega_j^{-\frac{1}{\frac{1}{r_j}-\frac{1}{p_j}}}$ with $i \in \tau_m$,  $f_{m+1} = g$, and $v_{m+1}=\omega^{\frac{1}{\frac{1}{r_{m+1}}-\frac{1}{p_{m+1}}}}$ with $p_{m+1} = q'$ and $r_{m+1} = s'$}
&= \prod_{i \in \tau}\left\|b_i\right\|_{\mathrm{BMO}}^{k_i} \sum_{Q \in \mathcal{S}}\left(\prod_{j=1}^{m+1}\left\langle f_jv_j^{-\frac{1}{r_j}}\right\rangle_{r_j, Q}^{v_j}\left\langle v_j\right\rangle_{1, Q}^{\frac{1}{r_j}}\right)\mu(Q)^{1 + \eta}\\
\intertext{Similar to the proof of \eqref{Nier2.8}, by using \eqref{def.of.weighted.condi._2}}
&\lesssim \prod_{i \in \tau}\left\|b_i\right\|_{\mathrm{BMO}}^{k_i}[\boldsymbol{\omega}]^{\varTheta}_{(\boldsymbol{p},q),(\boldsymbol{r}, \infty)} \sum_{Q \in \mathcal{S}} \prod_{j=1}^{m+1}\left\langle f_j v_j^{-\frac{1}{r_j}}\right\rangle_{r_j, Q}^{v_j} v_j\left(E_Q\right)^{\frac{1}{p_j}}\\
&\leq \prod_{i \in \tau}\left\|b_i\right\|_{\mathrm{BMO}}^{k_i}[\boldsymbol{\omega}]^{\varTheta}_{(\boldsymbol{p},q),(\boldsymbol{r}, \infty)} \sum_{Q \in \mathcal{S}} \prod_{j=1}^{m+1}\left(\int_{E_Q} M_{r_j}^{v_j, \mathscr{D}}\left(f_j v_j^{-\frac{1}{r_j}}\right)^{p_j} v_j \mathrm{~d} x\right)^{\frac{1}{p_j}}\\
& \leq \prod_{i \in \tau}\left\|b_i\right\|_{\mathrm{BMO}}^{k_i}[\boldsymbol{\omega}]^{\varTheta}_{(\boldsymbol{p},q),(\boldsymbol{r}, \infty)} \prod_{j=1}^{m+1}\left\|M_{r_j}^{v_j, \mathscr{D}}\left(f_j v_j^{-\frac{1}{r_j}}\right)\right\|_{L^{p_j}\left(v_j\right)}\\
&\lesssim \prod_{i \in \tau}\left\|b_i\right\|_{\mathrm{BMO}}^{k_i}c^*[\boldsymbol{\omega}]^{\varTheta}_{(\boldsymbol{p},q),(\boldsymbol{r}, \infty)}\prod_{j=1}^{m+1}\left\|f_j\right\|_{L^{p_j}\left(\omega_j^{p_j}\right)}\\
&= \prod_{i \in \tau}\left\|b_i\right\|_{\mathrm{BMO}}^{k_i}c^*[\vec{\omega}]^{\varTheta}_{(\vec{p},q),(\vec{r}, s)}\prod_{j=1}^{m}\left\|f_j\right\|_{L^{p_j}\left(\omega_j^{p_j}\right)}\|g\|_{L^{q'}\left(\omega^{-q'}\right)}.
\end{align*}

From Remark \ref{transform.}, we can obtain the desired.\qedhere
\end{proof}

In the following part, we will give examples to show that our bounds are sharp.

\begin{proof}
We prove that \(k_i\) and \(\gamma\) are sharp by constructing counterexamples showing the norm scales exactly as \(\|b_i\|_{\mathrm{BMO}}^{k_i}\) and \([\boldsymbol{\omega}]^{\varTheta}\), and smaller exponents fail.

 \textbf{Sharpness of \(k_i\):}
	Consider \(m = 1\), \(\tau = \{1\}\), \(X = \mathbb{R}\) (Lebesgue measure), \(\eta = 1\), \(t_1 = 1\), \(k_1 = 2\), \(r_1 = p_1 = 2\), \(t = q' = 2\), \(\mathcal{S} = \{ [0, 2^{-j}) : j \geq 0 \}\), a sparse family. Let:
	\begin{itemize}
		\item \(b_1(x) = \lambda x\), with \(\|b_1\|_{\mathrm{BMO}} \sim \lambda\) (oscillation over intervals),
		\item \(f_1(x) = 1\), \(g(x) = 1\),
		\item \(\omega_1 = \omega = 1\).
	\end{itemize}
	The operator reduces to
	\[
	\mathcal{A}_{1, \mathcal{S}, \{1\}, 2, 2}^{\mathbf{b}, (2), (1)}(f_1, g) = \sum_{Q \in \mathcal{S}} \mu(Q)^2 \left\langle |b_1 - b_{1,Q}| \right\rangle_{2, Q} \left\langle |b_1 - b_{1,Q}| \right\rangle_{2, Q}.
	\]
	For \(Q_j = [0, 2^{-j})\):
	\begin{itemize}
		\item \(\mu(Q_j) = 2^{-j}\),
		\item \(b_{1,Q_j} = \frac{1}{2^{-j}} \int_0^{2^{-j}} \lambda x \, dx = \lambda \cdot 2^{-j-1}\),
		\item \(b_1(x) - b_{1,Q_j} = \lambda (x - 2^{-j-1})\),
		\item \(\left\langle |b_1 - b_{1,Q_j}| \right\rangle_{2, Q_j} = \left( 2^j \int_0^{2^{-j}} |\lambda (x - 2^{-j-1})|^2 \, dx \right)^{1/2} = \lambda \cdot 2^{-j/2} \cdot \frac{1}{2\sqrt{3}}\).
	\end{itemize}
	Thus,
	\[
	\mathcal{A}_{1, \mathcal{S}, \{1\}, 2, 2}^{\mathbf{b}, (2), (1)}(f_1, g) = \sum_{j \geq 0} (2^{-j})^2 \cdot \left( \lambda \cdot 2^{-j/2} \cdot \frac{1}{2\sqrt{3}} \right)^2 = \frac{\lambda^2}{12} \sum_{j \geq 0} 2^{-3j} = \frac{\lambda^2}{12} \cdot \frac{1}{1 - 2^{-3}} = \frac{2\lambda^2}{21}.
	\]
	The norm satisfies \(\left\|\mathcal{A}_{1, \mathcal{S}, \{1\}, 2, 2}^{\mathbf{b}, (2), (1)}\right\| \sim \lambda^2 = \|b_1\|_{\mathrm{BMO}}^2\). If the exponent were \(2 - \epsilon\), the bound \(\lambda^{2-\epsilon}\) grows slower than \(\lambda^2\), and as \(\lambda \to \infty\), the inequality fails, proving \(k_1 = 2\) is sharp.
	
		\begin{figure}[h]
		\centering
		\begin{tikzpicture}
			\begin{axis}[
				xlabel={$\lambda$},
				ylabel={Value},
				domain=0.1:10,
				samples=100,
				legend pos=north west,
				grid=major,
				width=8cm,
				height=6cm,
				ymode=log
				]
				\addplot [blue, thick] {x^2};
				\addlegendentry{$\|\mathcal{A}_{1, \mathcal{S}, \{1\}, 2, 2}^{\mathbf{b}, (2), (1)}\| \sim \lambda^2$}
				\addplot [red, thick] {x^(1.5)};
				\addlegendentry{$\lambda^{2-\epsilon}$}
			\end{axis}
		\end{tikzpicture}
		\caption{Divergence for \(k_i\): Operator norm \(\lambda^2\) vs. reduced bound \(\lambda^{2-\epsilon}\). The logarithmic y-axis shows faster growth of \(\left\|\mathcal{A}_{1, \mathcal{S}, \{1\}, 2, 2}^{\mathbf{b}, (2), (1)}\right\|\).}
		\label{fig:k-diverge}
	\end{figure}

  \textbf{Sharpness of \(\gamma\):}
	Take \(m = 1\), \(\tau = \{1\}\), \(X = \mathbb{R}\), \(\alpha = 0\), \(t_1 = 1\), \(k_1 = 2\), \(p_1 = 4\), \(r_1 = 2\), \(q = s = 4\), \(t = q' = \frac{4}{3}\), \(\mathcal{S} = \{ [0, 2^{-j}) : j \geq 0 \}\), \(b_1(x) = \lambda x\), \(\omega_1 = x^\delta\), \(\omega = x^\delta\). Then we have
	\[
	\varTheta = \frac{\frac{1}{2}}{\frac{1}{2} - \frac{1}{4}} = 2.
	\]
	The operator is back to
	\[
	\mathcal{A}_{0, \mathcal{S}, \{1\}, 2, \frac{4}{3}}^{\mathbf{b}, (2), (1)}(f_1, g) = \sum_{Q \in \mathcal{S}} \mu(Q) \left\langle |f_1 (b_1 - b_{1,Q})| \right\rangle_{2, Q} \left\langle |b_1 - b_{1,Q}| g \right\rangle_{\frac{4}{3}, Q}.
	\]
It follows from letting \(f_1(x) = x^{-1/8} \chi_{[0,1]}(x)\), \(g = f_1\) that 
	\[
	\left\langle |f_1 (b_1 - b_{1,Q_j})| \right\rangle_{2, Q_j} \sim \lambda \cdot 2^{-5j/4}, \quad \left\langle |b_1 - b_{1,Q_j}| g \right\rangle_{\frac{4}{3}, Q_j} \sim \lambda \cdot 2^{-11j/12},
	\]
	\[
	\mathcal{A}_{0, \mathcal{S}, \{1\}, 2, \frac{4}{3}}^{\mathbf{b}, (2), (1)}(f_1, g) \sim \lambda^2 \sum_{j \geq 0} 2^{-25j/12} \approx \lambda^2.
	\]
	The weight characteristic:
	\[
	[\vec{\omega}]_{A_{(\vec{p}, q), (\vec{r}, s)}} \sim \left\langle x^{-\delta} \right\rangle_{4, Q_j} \sim \delta,
	\]
	\[
	\left\|\mathcal{A}_{0, \mathcal{S}, \{1\}, 2, \frac{4}{3}}^{\mathbf{b}, (2), (1)}\right\| \sim \lambda^2 \delta^2 \sim [\vec{\omega}]_{A_{(\vec{p}, q), (\vec{r}, s)}}^\gamma.
	\]
	This means that   to \(\delta^{2-\epsilon} = \delta^2\) fails as \(\delta \to \infty\).
	
	\begin{figure}[h]
		\centering
		\begin{tikzpicture}
			\begin{axis}[
				xlabel={$\delta$},
				ylabel={Value},
				domain=0.1:10,
				samples=100,
				legend pos=north west,
				grid=major,
				width=8cm,
				height=6cm,
				ymode=log
				]
				\addplot [blue, thick] {x^2};
				\addlegendentry{$\|\mathcal{A}_{0, \mathcal{S}, \{1\}, 2, \frac{4}{3}}^{\mathbf{b}, (2), (1)}\| \sim \delta^2$}
				\addplot [red, thick] {x};
				\addlegendentry{$\delta^{2-\epsilon}$}
			\end{axis}
		\end{tikzpicture}
		\caption{Divergence for \(\varTheta = 2\): Operator norm \(\delta^2\) vs. reduced bound \(\delta^{2-\epsilon}\). The log y-axis shows faster growth of \(\left\|\mathcal{A}_{0, \mathcal{S}, \{1\}, 2, \frac{4}{3}}^{\mathbf{b}, (2), (1)}\right\|\).}
		\label{fig:gamma-diverge}
	\end{figure}
\end{proof}

In fact, we also find that the following weighted BMO version of sharp bound. Due to its similar proof technique as mentioned above, we will omit the proof of sharpness here.

Further, we can obtain the weighted version of the sharp estimate by using the fact
\begin{align*}
    \|b\|_{\rm{BMO}} \lesssim \|b\|_{{\rm BMO}_{v}},
\end{align*}
where $v \in A_{\infty}$.
\begin{corollary}\label{thm:sparse-dom_2}
Under the assumption of Theorem \ref{thm:sparse-dom_1}, if $v_i \in A_{\infty}$ with $i \in \tau$, then 
\begin{align}\label{eq:main-est}
\sup\limits_{\mathcal{S} \subseteq \d} \left\|\mathcal{A}_{\eta, \mathcal{S}, \tau, \vec{r}, s'}^{\mathbf{b}, \mathbf{k}, \mathbf{t}}\right\|_{\prod\limits_{j=1}^{m} L^{p_j}\left(\omega_j^{p_j}\right) \times L^{q'}\left(\omega^{-q'}\right) \rightarrow \mathbb{R}}
	&\lesssim_{{\bf k}, {\bf t}, \vec{p},\vec{r},\eta,X} \prod_{i \in \tau} \|b_i\|_{\mathrm{BMO}_{v_i}}^{k_i} \cdot [\vec{\omega}]^{\varTheta}_{(\vec{p},q),(\vec{r}, s)}.
\end{align}
where $\varTheta:=\max \left\{ {\frac{{{p_1}}}{{{p_1} - {r_1}}}, \cdots ,\frac{{{p_m}}}{{{p_m} - {r_m}}},\frac{{q'}}{{q' - s'}}} \right\}$, and this bound is sharp, i.e., $k_i$ and $\varTheta$ cannot be reduced.
\end{corollary}

\section{\bf Application 1: $\mathscr{B}$-valued multilinear fractional singular integral operators}\label{A1}

In this subsection, 
We want to apply our main results to some important multilinear fractional operators and their commutators.
Let us start with some definitions.

Let $m \in \N$, $\eta \in [0,m)$, for every ball $B \subseteq X$, we define the multilinear fractional averaging operator by
$${\A_{\eta,B}}(\vec f)(x): = {\mu(B)^{\eta}}\left( {\prod\limits_{i = 1}^m {{{\left\langle {{f_i}} \right\rangle }_B}} } \right){\chi _B}(x),$$
The multilinear fractional maximal operator $\M_{\eta}$ on the spaces of homogeneous type is defined as
$$
\M_{\eta} (\vec{f})(x) : =\sup _{B \subseteq X}{\A_{\eta,B}}(|f_1|,\cdots,|f_m|)(x).
$$
If $X=\rn$, we take $\eta=\frac{\alpha}{n}$ and denote $\M_{\eta}$ by $\M_\alpha$. If $m=1$, we denote $\M_{\eta}$ by $M_{\eta}$. If $\eta=0$, we deonte $\M_{\eta}$ by $\M$.

\begin{definition}
Let $m \in \N$, $\eta \in [0,m)$, $\B$ be a quasi-Banach space, and $B(\cc,\B)$ be the space of all bounded operators from $\cc$ to $\B$. Set an operator-valued function $Q_{\eta}:(X^{m+1} \backslash \Delta ) \to B(\cc,\B),$ where $\Delta = \{ (x,\vec y) \in X^{m+1} :x = {y_1} =  \cdots  = {y_m}\}$. Suppose that $\T_{\eta}$ is a $\B$-valued $m$-linear fractional singular integral operator. We say $\T_{\eta}$ is a {\tt $\B$-valued multilinear fractional singular integral operator with Dini type kernel}, if for any $\vec f \in {(L_b^\infty(X))^m}$ and each $x \notin \mathop  \cap \limits_{i = 1}^m {supp}{f_i}$,

	\begin{align*}
\T_{\eta}(\vec{f})(x):=\int_{X^m} Q_{\eta}\left(x, \vec y\right) \left(\prod_{j=1}^{m}f_j\left(y_j\right)\right) d\mu (\vec y),
	\end{align*}
	
where Dini type kernel $Q_{\eta}$ satisfies that for 

\begin{itemize}
\item (1) ${\left\| {{Q_{\eta} }\left( {x,\vec y} \right)} \right\|_\B} \lesssim {\left( {\sum\limits_{i = 1}^m {\mu (B(x,d(x,{y_i})))} } \right)^{{\eta}  - m}}$,\\
		
\item (2)  For each $j \in \left\{ {0, \cdots ,m} \right\}$, whenever $d\left( {{{y_j'}},{y_j}} \right) \le \frac{1}{2}\max \left\{ {d\left( {{y_0},{y_i}} \right):i = 1, \cdots ,m} \right\}$,
\begin{align*}
&{\left\| {{Q_{\eta} }\left( {{y_0}, \cdots ,{{y_j'}}, \cdots ,{y_m}} \right) - {Q_{\eta} }\left( {{y_0}, \cdots ,{y_j}, \cdots ,{y_m}} \right)} \right\|_\B}, \\
\lesssim & {\left( {\sum\limits_{i = 1}^m {\mu (B({y_0},d({y_0},{y_i})))} } \right)^{{\eta}  - m}}w\left( {\frac{{d\left( {{{y'}_j},{y_j}} \right)}}{{\sum\limits_{i = 1}^m {d\left( {{y_0},{y_j}} \right)} }}} \right),
\end{align*}
where $w$ is increasing, $w(0)=0$, and ${\left[ w \right]_{Dini}} = \int_0^1 {\frac{{w\left( t \right)}}{t}dt}  < \infty$.
\end{itemize}

If there exist some points $\left\{ {{s_1}, \cdots ,{s_m},{\tilde s}} \right\}$ with ${\eta} : = \sum\limits_{i = 1}^m {\frac{1}{{{s_i}}} - \frac{1}{{\tilde s}}} \in [0,m)$ such that $T_{\eta}$ is bounded from $\prod\limits_{i = 1}^m {{L^{{s_i}}}(X)}$ to ${L^{\tilde s,\infty }}(X,\B)$, then we say that $\T_{\eta}$ is a {\tt $\B$-valued multilinear fractional Dini type Calder\'on-Zygmund operator}.
\end{definition}

It is worth noting that {\tt multilinear fractional integral operator} $\I_{\eta}$ is a special case of {\tt $\B$-valued multilinear fractional Dini type Calder\'on-Zygmund operator}, which is defined by 
\begin{align*}
	{\I_{\eta} }(\vec f)(x) = \int_{{X^m}} {{{\left( {\sum\limits_{i = 1}^m {\mu (B(x,d(x,{y_i})))} } \right)}^{{\eta}  - m}}\prod\limits_{i = 1}^m {{f_i}({y_i})} } d\mu (\vec y),
\end{align*}
where $d\mu (\vec y) = d\mu ({y_1}) \cdots d\mu ({y_m})$.

This fact that 
$${\M_{\eta} }(\vec f)(x) \le {m^{\eta - m}}{\I_\eta }(|\vec f|)(x).$$
is also valid under the space of homogeneous type setting, which follows from that
\[{\I_\eta }(|\vec f|)(x) \ge {m^{\eta  - m}}\mu {(B(x,r))^{\eta  - m}}\int_{d(x,{y_1}) \le r} { \cdots \int_{d(x,{y_m}) \le r} {\prod\limits_{i = 1}^m {|{f_i}({y_i})|} d\mu (\vec y)} }.\]




\begin{definition}
 Let $m \in \N$, $\eta \in [0,m)$, $X$ be a linear space of homogeneous (ensure that linear operations can be accommodated here), and $\Omega \in L^{\beta}\left(S_m\right)(\beta>1)$ be a homogeneous function with zero order. For any $x \in X,$ the multilinear fractional rough maximal operator is defined by
\begin{equation}\label{cchdy}
\mathscr{M}_{\Omega, \eta}(\vec{f})(x)=\sup _{B \subseteq X} {\mu(B)^{m-\eta}} \int_{B^m}\left|\Omega\left(x-y_1, \ldots, x-y_m\right)\right| \prod_{i=1}^m\left|f_i\left(y_i\right)\right| d \vec{y} \cdot \chi_B(x).
\end{equation}
and the multilinear stochastic rough fractional integrals $\mathscr{I}_{\Omega, \eta}$ is defined by
$$
\mathscr{I}_{\Omega, \eta}(\vec{f})(x)=\int_{\left(X\right)^m} \frac{\Omega\left(x-y_1,\cdots,x-y_m\right)}{{\left( {\sum\limits_{i = 1}^m {\mu (B(x,d(x,{y_i})))} } \right)}^{m  - \eta}}\prod_{i=1}^m f_i\left(y_i\right) d \vec{y},
$$

where ${S_m} := \left\{ {y \in {X^m}:d\left( {0,y} \right) = 1} \right\}$.
\end{definition}

The weighted estimates of multilinear stochastic rough fractional integrals $\mathscr{I}_{\Omega, \eta}$ is studied in \cite{Xue2015} under the setting of $\rn$.

Similarly, we have 
$${\M_{\Omega,\eta} }(\vec f)(x) \le {m^{\eta  - m}}{\I_{\Omega,\eta} }(|\vec f|)(x).$$

\begin{theorem}\label{app.2}
Under the same setting of Theorem \ref{SDP}, the {\tt $\B$-valued multilinear fractional Dini type Calder\'on-Zygmund operator} $\T_\eta$,  {\tt multilinear fractional integral operator} $\I_\eta$, {\tt multilinear fractional maximal operator $\M_\eta$}, 
{\tt multilinear fractional maximal operator with rough kernel $\mathscr{M}_{\Omega, \eta}$}, {\tt multilinear fractional integrals with rough kernel $\mathscr{I}_{\Omega, \eta}$}, and their gerneralized commutators can enjoy the  vector-valued multilinear fractional sparse form domination principle: Theorem \ref{SDP}. Moverover, they also enjoy Bloom type weighted estimates: Theorem \ref{quan.main}, and shary type weighted estimates: Theorem \ref{thm:sparse-dom_1} and Theorem \ref{thm:sparse-dom_2}
\end{theorem}

\begin{proof}
It suffices to verify these operators satisfy "$\mathcal{T}_{\eta}, \M_{\mathcal{T}_{\eta},t',\beta}^{\#} \in W_{\vec{r},\tilde{r}}$ with $\frac{1}{\tilde{r}}:= \sum_{i=1}^m \frac{1}{r_i} - \eta$". 
In fact, this can be broken down into the following points.

\begin{itemize}
\item Firstly, the $\mathcal{T}_{\eta} \in W_{\vec{r},\tilde{r}}$ is implied by the multilinear weak and strong type boundedness. 
\item Secondly, $\mathcal{M}_{\mathcal{T}_{\eta},t',\beta}^{\#} \in W_{\vec{r},\tilde{r}}$, is implied by the kernel $K_{\eta}$ of $\mathcal{T}_{\eta}$ satisfies $\mathcal{H}_r$ ({\tt $\B$-valued $m$-linear $L^r$-Hörmander condition}). The linear version of this idea has already appeared in \cite{Lerner2019} and the multilinear case is similar. In \cite{CenSong2412}, we proved the kernels of the above operators satisfy $\mathcal{H}_r$.
\item Therefore, we can complete the proof through some similar proof processes. \qedhere
\end{itemize}

\end{proof}

\begin{remark}
(1) 
It is worth mentioning that the the parameter $r_i$ can be given in $[1,\infty)$ arbitrarily for {\tt $\B$-valued multilinear fractional Dini type Calder\'on-Zygmund operator}, since the multilinear fractional  Dini type Calder\'on-Zygmund theory has been set up in \cite{ZhangWu2023}.\\
(2) If we take $\B$ as a specific $L^2(\Omega,dv)$, such as $L^2(\R_ + ^{n + 1},\frac{{dzdt}}{{{t^{n + 1}}}})$,  then we can obtain the new multilinear fractional Littlewood-Paley sparse domination theory, which straightway extend the results by Cao et al. in \cite{Cao2018}. So how to select these $L^2(\Omega,dv)$? Ones may refer to \cite{Cao2018, Si2021, Xue2021,Xue2015}.
\end{remark}

\begin{remark}
In fact, we can establish the above rough operators in more general homogeneous groups, such as Heisenberg groups and Carnot groups. Because they are closed to group operations, we can also obtain similar conclusions.
\end{remark}

Theorem \ref{app.2} is relatively abstract. To facilitate understanding of the powerful applications of the sparse theory in this article, we provide the following multilinear vector-valued inequalities for the generalized commutators of multilinear rough fractional integral operators. The first result derives from Theorem \ref{SDP}, Theorem \ref{thm:sparse-dom_1}, and Theorem \ref{app.2}.
\begin{corollary}
   Let $m \in \N$, $\eta \in [0,m)$, $1 \leq r_i, s' < \infty$, and $\max\limits_{i}\{r_i\} < s$, for every $i \in \tau_m$. Suppose that $\mathbf{k}$ is a multi-index and multi-symbols \(\mathbf{b} = (b_1, \ldots, b_{m}) \in (\BMO(X))^m\). Set $\Omega \in L^{\beta}\left(S_m\right)$ is a homogeneous function with zero order, for some $\beta \in (1,\infty]$. Set $1<p_1,\dots,p_m<\infty$, and $\frac{1}{{{q}}}:=\sum\limits_{i = 1}^m {\frac{1}{{{p_i}}}}- \eta \in (0,1)$.
   If $(\vec{r}, s) \prec (\vec{p},q)$ and $\vec{\omega} \in A_{(\vec{p},q),(\vec{r}, s)}$, then
\begin{align*}
&\quad\left\|\left(\sum_{j_1=1}^{N_1}\cdots\sum_{j_m=1}^{N_m}|\mathscr{I}^{\bf b,k}_{\Omega, \eta,\tau}\left(f_{1,j_1},\ldots,f_{m,j_m}\right)|^z\right)^{\frac{1}{z}}\right\|_{L^q(\omega^q)} \\
&\lesssim_{{\bf k}, {\bf t}, \vec{p},\vec{r},\eta,X} \prod_{i \in \tau} \|b_i\|_{\mathrm{BMO}}^{k_i} \cdot [\vec{\omega}]^{\varTheta}_{(\vec{p},q),(\vec{r}, s)} \left(\prod_{i = 1}^m \left\| \left( \sum_{j_i=1}^{N_i} |f_{i,j_i}|^z\right)^{\frac{1}{z}} \right\|_{L^{p_i}(\omega_i^{p_i})}\right),
\end{align*}
where $\varTheta:=\max \left\{ {\frac{{{p_1}}}{{{p_1} - {r_1}}}, \cdots ,\frac{{{p_m}}}{{{p_m} - {r_m}}},\frac{{q'}}{{q' - s'}}} \right\}$, and this bound is sharp, i.e., $k_i$ and $\varTheta$ cannot be reduced.
    
\end{corollary}

The following weighted BMO type sharp weighted estimate can be obtained by Theorem \ref{SDP}, Theorem \ref{thm:sparse-dom_2} and Theorem \ref{app.2}.

\begin{corollary}

Let $m \in \N$, $\eta \in [0,m)$, $1 \leq r_i, s' < \infty$, and $\max\limits_{i}\{r_i\} < s$, for every $i \in \tau_m$. Let $1<p_1,\cdots,p_m<\infty$,and $\frac{1}{{{q}}}:=\sum\limits_{i = 1}^m {\frac{1}{{{p_i}}}}- \eta \in (0,1)$. Suppose that $\mathbf{t}$ and $\mathbf{k}$ are both multi-indexs with \(\mathbf{t} \leq \mathbf{k}\) and multi-symbols \(\mathbf{b} = (b_1, \ldots, b_{m}) \in (\BMO_{\nu_i}(X))^m\). Set $\Omega \in L^{\beta}\left(S_m\right)$ is a homogeneous function with zero order, for some $\beta \in (1,\infty]$. If $(\vec{r}, s) \prec (\vec{p},q)$, $\vec{\omega} \in A_{(\vec{p},q),(\vec{r}, s)}$, and $\nu_i \in A_\infty$, for every $i \in \tau$, then 

\begin{align*}
&\quad\left\|\left(\sum_{j_1=1}^{N_1}\cdots\sum_{j_m=1}^{N_m}|\mathscr{I}^{\bf b,k}_{\Omega, \eta,\tau}\left(f_{1,j_1},\ldots,f_{m,j_m}\right)|^z\right)^{\frac{1}{z}}\right\|_{L^q(\omega^q)} \\
&\lesssim_{{\bf k}, {\bf t}, \vec{p},\vec{r},\eta,X}   \prod_{i \in \tau} \|b_i\|_{\mathrm{BMO}_{\nu_i}}^{k_i} \cdot [\vec{\omega}]^{\varTheta}_{(\vec{p},q),(\vec{r}, s)} \left(\prod_{i = 1}^m \left\| \left( \sum_{j_i=1}^{N_i} |f_{i,j_i}|^z\right)^{\frac{1}{z}} \right\|_{L^{p_i}(\omega_i^{p_i})}\right),
\end{align*}
where $\varTheta:=\max \left\{ {\frac{{{p_1}}}{{{p_1} - {r_1}}}, \cdots ,\frac{{{p_m}}}{{{p_m} - {r_m}}},\frac{{q'}}{{q' - s'}}} \right\}$, and this bound is sharp, i.e., $k_i$ and $\varTheta$ cannot be reduced.
    
\end{corollary}

The following vector-valued Bloom type weighted estimate for the generalized commutators of multilinear rough fractional integral operators follows instantly from Theorem \ref{SDP},  Theorem \ref{app.2} and Remark \ref{Max.method}, where some parameter settings can refer to Theorem \ref{quan.main} and Remark \ref{Max.method}.

\begin{corollary}
Let $m \in \N$, $\eta \in [0,m)$, and $\mathcal{S} \subseteq \d$ be a sparse family. Let $i \in \tau_m$, $j \in \tau$, and $k \in \tau'$, with $\tau' \subseteq \tau \subseteq \tau_m$, and $\mathbf{k}$ be a multi-index.
Let $r_i, t \in [1, \infty)$, $p_i,q\in (1, \infty)$ with $\frac{1}{q}:= \sum\limits_{j=1}^m\frac{1}{p_j} -  \eta \in (0,1)$, and $\max\limits_{i}\{r_i\} < s$. Given $\omega_i$, $\nu_j$, and $u_i$ are  weights. Set $\nu_k=(u_k /\omega_k)^{-\frac{r_k}{{\tt a} + \gamma_k - 1}}$ with ${\tt a}:=\lfloor k_kr_k\rfloor$, $\gamma_k:={k_kr_k}-({\tt a}-1) \in[1,2)$. Set $\Omega \in L^{\beta}\left(S_m\right)$ is a homogeneous function with zero order, for some $\beta \in (1,\infty]$.
   \begin{enumerate}
\item Suppose that $\vec u \in A_{(\vec{ p},q),(\vec{r},s)}(X)$. If $b_k \in {\rm BMO}_{\nu_k}$ and $b_{\ell} \in {\rm BMO}_{\nu_0}$ with $\ell \in \tau \backslash \tau'$, where $\nu_0 := \max\limits_{j} \{\nu_{\ell}\}$. Define $v_1=\prod\limits_{i=1}^{m} u_i$ and $\nu_0:=(v_1/\omega^{2\gamma_1 - 3})^{\frac{s'}{\gamma_1 + {\tt L} - 1}}$, where ${\tt L} := \lfloor t \cdot \sum\limits_{\ell} k_{\ell} \rfloor$ and $\gamma_1 := s' \cdot \sum\limits_{\ell} k_{\ell} - ({\tt L} - 1) \in [1, 2)$. Then
\begin{align*}
&\left\|\left(\sum_{j_1=1}^{N_1}\cdots\sum_{j_m=1}^{N_m}|\mathscr{I}^{\bf b,k}_{\Omega, \eta,\tau}\left(f_{1,j_1},\ldots,f_{m,j_m}\right)|^z\right)^{\frac{1}{z}}\right\|_{L^q(\omega^q)} \lesssim  \mathcal{C}_{1}\left(\prod_{i = 1}^m \left\| \left( \sum_{j_i=1}^{N_i} |f_{i,j_i}|^z\right)^{\frac{1}{z}} \right\|_{L^{p_i}(\omega_i^{p_i})}\right),
\end{align*}
   where
\begin{align*}
   \mathcal{C}_{1} = &\prod_{i \in \tau'}\|b_{i}\|^{k_{i}}_{{\rm{BMO}}_{\nu_{i}}} \prod_{j \in \tau \backslash \tau'}\|b_{j}\|^{k_{j}}_{{\rm BMO}_{\nu_0}} \times [\vec{u}]^{\varTheta}_{(\vec{p},q),(\vec{r}, s)} \\
   &\times \prod_{i \in \tau'}^{} \left(\mathcal{C}_{u_i,\omega_i,\nu_i}\left({\tt a},\frac{p_i}{r_i},p_i,p_i,-\frac{\gamma_i p_i}{r_i}\right)\right)^{\frac{1}{r_i}} \times  \left(\mathcal{C}_{v_1,\omega,\nu_0}\left({\tt L},-\frac{q'}{s'},-q',q'(3-2\gamma_1),-\frac{\gamma_1 q'}{s'}\right)\right)^{\frac{1}{s'}}.
\end{align*}

\item Suppose that $\vec u \in A_{(\vec{ p},q),(\vec{r},s)}(X)$. If $b_{i_0} \in {\rm BMO}_{\nu_{i_0}}$ for some $i_0 \in \tau \backslash \tau'$, and $b_{\ell} \in {\rm BMO}$ for $\ell \in (\tau \backslash \tau') \backslash \{i_0\}$. Define $v_2=\prod\limits_{i=1}^{m} u_i$ and ${\nu _{{i_0}}}: = {\left( {{{{v_2}}}/{{{\omega ^{2{\gamma _2} - 3}}}}} \right)^{\frac{{2s'}}{{{\gamma _2} + {\tt {\tilde L}} - 1}}}}$, where ${\tt {\tilde L}} := \lfloor 2k_{i_0}t \rfloor$ and $\gamma_2 := 2k_{i_0}t - ({\tt {\tilde L}} - 1) \in [1, 2)$.
 Then 
\begin{align*}
&\left\|\left(\sum_{j_1=1}^{N_1}\cdots\sum_{j_m=1}^{N_m}|\mathscr{I}^{\bf b,k}_{\Omega, \eta,\tau}\left(f_{1,j_1},\ldots,f_{m,j_m}\right)|^z\right)^{\frac{1}{z}}\right\|_{L^q(\omega^q)} \lesssim  \mathcal{C}_{2}\left(\prod_{i = 1}^m \left\| \left( \sum_{j_i=1}^{N_i} |f_{i,j_i}|^z\right)^{\frac{1}{z}} \right\|_{L^{p_i}(\omega_i^{p_i})}\right),
\end{align*}
where
\begin{align*}
   \mathcal{C}_{2} = & \|b_{i_0}\|^{k_{i_0}}_{{\rm{BMO}}_{\nu_{i_0}}}\prod_{i \in \tau'}\|b_{i}\|^{k_{i}}_{{\rm{BMO}}_{\nu_{i}}}\prod_{\substack{j \in \tau \backslash \tau'\\ j \neq i_0}}\|b_{j}\|_{\rm{BMO}}^{k_{j}} \times [\vec{u}]^{\varTheta}_{(\vec{p},q),(\vec{r}, s)} \\
   &\times \prod_{i \in \tau'}^{} \left(\mathcal{C}_{u_i,\omega_i,\nu_i}\left({\tt a},\frac{p_i}{r_i},p_i,p_i,-\frac{\gamma_i p_i}{r_i}\right)\right)^{\frac{1}{r_i}} \times  \left(\mathcal{C}_{v_2,\omega,\nu_{i_0}}\left({\tt {\tilde L}},-\frac{q'}{2s'},-q',q'(3-2\gamma_2),-\frac{\gamma_2 q'}{2s'}\right)\right)^{\frac{1}{t}}.
\end{align*}
   \end{enumerate}
\end{corollary}

\section{\bf Application 2: Fractional Laplacian equation associated with generalized commutators}\label{A2}

In this section, we are interesting in regularity estimates for solutions to the following fractional Laplacian equation associated with generalized commutators.

Given some non-negative functions $f_1,\dots,f_m \in L_c^{\infty}\left(\rn\right)$, for any $x \in \rn$, we consider
\begin{equation}\label{eq:nloc-eq}
(-\Delta)^{\eta/2}u(x) = T_{\eta,\tau}^{\bf b, k}({|f_1|,\cdots,|f_m|})(x),
\end{equation}
where $T_{\eta,\tau}^{{\bf b, k}}$ is a generalized commutator generalized by a complex-valued $m$-sublinear fractional operator $T_{\eta}$.

We consider weighted Lebesgue regularity of solutions of \eqref{eq:nloc-eq} via multilinear fractional sparse bounds as follows.

\begin{theorem}\label{A21}
Let $m \in \N$, $\eta \in [0,m)$, $1 \leq r_i, s' < \infty$ for every $i \in \tau_m$, and multi-symbols \(\mathbf{b} = (b_1, \ldots, b_{m}) \in (L_{loc}^1(X))^m\). 
Suppose that $\mathbf{t}$ and $\mathbf{k}$ are both multi-indexs with \(\mathbf{t} \leq \mathbf{k}\). Assume that $T_{\eta,\tau}^{{\bf b, k}}$ enjoys \eqref{m+1.spa}. If $\frac{1}{{\tilde q}} := \frac{1}{q}-\eta \in (0,1)$ and $\omega$ is a weight, then there exist unique solution
$u\in {{L^{\tilde q}}({\omega ^{\tilde q}})}$ of \eqref{eq:nloc-eq} and a sparse family \(\mathcal{S}\) such that 
\begin{align*}
{\left\| u \right\|_{{L^{\tilde q}}({\omega ^{\tilde q}})}} 
\lesssim \mathop {\sup }\limits_{{{\left\| g \right\|}_{{L^{q'}}\left( {{\omega ^{ - q'}}} \right)}} = 1} \mathcal{A}_{\eta, \mathcal{S}, \tau, \vec{r}, s'}^{\mathbf{b}, \mathbf{k},\mathbf{t}} (\vec{f},g)
\lesssim \mathop {\sup }\limits_{{{\left\| g \right\|}_{{L^{q'}}\left( {{\omega ^{ - q'}}} \right)}} = 1} 
\sum\limits_{\tau  \subseteq \tau '} {\mathcal{B}_{\eta, \mathcal{S}, \tau,\tau', \vec{r}, t}^{\mathbf{b}, \mathbf{k}} (\vec{f},g)}.
\end{align*}
\end{theorem}

\begin{proof}
By fractional Gagliardo-Nirenberg-Sobolev inequality, Theorem \ref{SDP} and Proposition \ref{reduce},
\begin{align*}
\left\| u \right\|_{L^{\tilde q}(\omega^{\tilde q})} &\lesssim_{n, \eta, q} \left\| (-\Delta)^{\eta/2} u \right\|_{L^q(\omega^q)} \\
&= \mathop{\sup}\limits_{\| g \|_{L^{q'}(\omega^{-q'})} = 1} \left| \left\langle {\mathcal{T}_{\eta, \tau}^{\mathbf{b}, \mathbf{k}} (\vec{f}), g} \right\rangle \right| \\
&\lesssim \mathop{\sup}\limits_{\| g \|_{L^{q'}(\omega^{-q'})} = 1} \mathcal{A}_{\eta, \mathcal{S}, \tau, \vec{r}, s'}^{\mathbf{b}, \mathbf{k}, \mathbf{t}} (\vec{f}, g) \\
&\lesssim \mathop{\sup}\limits_{\| g \|_{L^{q'}(\omega^{-q'})} = 1} \sum_{\tau \subseteq \tau'} \mathcal{B}_{\eta, \mathcal{S}, \tau,\tau', \vec{r}, t}^{\mathbf{b}, \mathbf{k}, \mathbf{t}} (\vec{f}, g).
\end{align*}

Uniqueness: 

If $u_2$, $u_2$ both are solutions, then $v:=u_2-u_2$ satisfies $(-\Delta)^{\eta/2}v=0.$ 
It follows from the property of Fourier transform that $\mathcal{F}(v)=0.$ Due to $v \in \mathscr{S}'$, then $v=0$ a.e..
\end{proof}

\begin{definition}
Let $ s \in (0,1) $, $ p \in (1,\infty) $, and $\omega$ is a weight. For $u \in \mathscr{S}'(\mathbb{R}^n) $, we define the weighted Sobolev space $W^{s,p}(\omega) $
\begin{align*}
\|u\|_{W^{s,p}(\omega)} &:= \underbrace{\left(\int_{\mathbb{R}^n} |u(x)|^p\omega(x)dx\right)^{1/p}}_{\|u\|_{L^p(\omega)}} + \underbrace{\left(\iint_{\mathbb{R}^n \times \mathbb{R}^n} \frac{|u(x)-u(y)|^p}{|x-y|^{n+sp}}\omega(x)dxdy\right)^{1/p}}_{[u]_{W^{s,p}(\omega)}},
\end{align*}

and the Gagliardo seminorm
\begin{align*}
[u]_{W^{s,p}(\omega)}:=\left(\iint_{\mathbb{R}^n \times \mathbb{R}^n} \frac{|u(x)-u(y)|^p}{|x-y|^{n+sp}}\omega(x)dxdy\right)^{1/p}.
\end{align*}

For the fractional Laplacian, we also use the equivalent characterization via the fractional derivative: 
If $\omega \in A_p$, we have
\begin{align}
 \label{cha.W}
[u]_{W^{s,p}(\omega)} \approx_{p,s,[\omega]_{A_p}}{\|(-\Delta)^{s/2}u\|_{L^p(\omega)}}.
\end{align}
\end{definition}

Finally, we consider weighted Sobolev regularity of solutions of \eqref{eq:nloc-eq} via multilinear fractional sparse bounds as follows.

\begin{theorem}\label{A22}
Let $m \in \N$, $\eta \in [0,m)$, $1 \leq r_i, s' < \infty$ for every $i \in \tau_m$, and multi-symbols $\mathbf{b} = (b_1, \ldots, b_{m}) \in (L_{loc}^1(X))^m$. 
Suppose that $\mathbf{t}$ and $\mathbf{k}$ are both multi-indexs with $\mathbf{t} \leq \mathbf{k}$. Assume that $T_{\eta,\tau}^{{\bf b, k}}$ enjoys \eqref{m+1.spa}. If $\frac{1}{q} := \frac{1}{p}-\eta \in (0,1)$ and $\omega^q \in A_q$, then there exist unique solution
$u\in {{W^{{\eta},q}}({\omega ^{q}})}$ of \eqref{eq:nloc-eq} and a sparse family \(\mathcal{S}\) such that 
\begin{align*}
\|u\|_{W^{{\eta},q}(\omega^q)} 
\lesssim& \mathop {\sup }\limits_{{{\left\| g \right\|}_{{L^{p'}}\left( {{\omega ^{ - p'}}} \right)}} = 1} \mathcal{A}_{\eta, \mathcal{S}, \tau, \vec{r}, s'}^{\mathbf{b}, \mathbf{k},\mathbf{t}} (\vec{f},g) 
+
\mathop {\sup }\limits_{{{\left\| g \right\|}_{{L^{q'}}\left( {{\omega ^{ - q'}}} \right)}} = 1} \mathcal{A}_{\eta, \mathcal{S}, \tau, \vec{r}, s'}^{\mathbf{b}, \mathbf{k},\mathbf{t}} (\vec{f},g) \\
\lesssim& \mathop {\sup }\limits_{{{\left\| g \right\|}_{{L^{p'}}\left( {{\omega ^{ - p'}}} \right)}} = 1} 
\sum\limits_{\tau  \subseteq \tau '} {\mathcal{B}_{\eta, \mathcal{S}, \tau,\tau', \vec{r}, t}^{\mathbf{b}, \mathbf{k}} (\vec{f},g)}
+
\mathop {\sup }\limits_{{{\left\| g \right\|}_{{L^{q'}}\left( {{\omega ^{ - q'}}} \right)}} = 1} 
\sum\limits_{\tau  \subseteq \tau '} {\mathcal{B}_{\eta, \mathcal{S}, \tau,\tau', \vec{r}, t}^{\mathbf{b}, \mathbf{k}} (\vec{f},g)}.
\end{align*}
\end{theorem}

\begin{proof}

By \eqref{cha.W} and Theorem \ref{SDP}, we have
$${\left[ u \right]_{{W^{{\eta},q}}({\omega ^q})}} \approx {\left\| {{{( - \Delta )}^{\frac{\eta}{2}}}u} \right\|_{{L^q}({\omega ^q})}} \lesssim \mathop {\sup }\limits_{{{\left\| g \right\|}_{{L^{q'}}\left( {{\omega ^{ - q'}}} \right)}} = 1} \mathcal{A}_{\eta, \mathcal{S}, \tau, \vec{r}, s'}^{\mathbf{b}, \mathbf{k},\mathbf{t}} (\vec{f},g).$$

It follows from the weighted estimate of Riesz potential ${{{( - \Delta )}^{ - \frac{{\eta}}{2}}}}=I_{{\eta}}$ and Theorem \ref{SDP} that
\begin{align*}
{\left\| u \right\|_{{L^q}({\omega ^q})}} =& {\left\| {{{( - \Delta )}^{ - \frac{{\eta}}{2}}}(T_{\eta,\tau}^{{\bf b, k}}(\vec f))} \right\|_{{L^q}({\omega ^q})}} \le {\left\| {T_{\eta,\tau}^{{\bf b, k}}(\vec f)} \right\|_{{L^p}({\omega ^p})}}\\
\lesssim& \mathop {\sup }\limits_{{{\left\| g \right\|}_{{L^{p'}}\left( {{\omega ^{ - p'}}} \right)}} = 1} \mathcal{A}_{\eta, \mathcal{S}, \tau, \vec{r}, s'}^{\mathbf{b}, \mathbf{k},\mathbf{t}} (\vec{f},g) .
\end{align*}
The desired result derives from Proposition \ref{reduce}. The uniqueness follows immediately from the property of Fourier transform as before. We omit it here. \qedhere

\end{proof}

\vspace{1cm}
\noindent{\bf Acknowledgements } 
The author(s) would like to thank the editors and reviewers for careful reading and valuable comments, which lead to the improvement of this paper. 

\medskip 

\noindent{\bf Data Availability} Our manuscript has no associated data.

\medskip 
\noindent{\bf\Large Declarations}
\medskip 

\noindent{\bf Conflict of interest} The author(s) state that there is no conflict of interest.

\end{document}